\documentclass[10pt,reqno, notitlepage]{amsart} 

\usepackage{graphicx,subfig}
\usepackage{enumerate} 
\usepackage[colorlinks=true, pdfstartview=FitV,
linkcolor=cyan, citecolor=red, urlcolor=blue]{hyperref}
\usepackage{amsmath,amsfonts,latexsym,amssymb}
\usepackage[utf8]{inputenc} 
\usepackage[T1]{fontenc}
\usepackage{ae,aecompl}
\usepackage{dsfont} 
\usepackage{pxfonts}
\usepackage{microtype}
\usepackage{braket}

\usepackage{yhmath}

\usepackage{rotating}

\usepackage[all]{xy}
\usepackage{accents}
\newcommand*{\dt}[1]{\overset{\bullet}{#1}}
\newtheorem{theorem}{Theorem}[subsection]
\newtheorem{lemma}[theorem]{Lemma}
\newtheorem{proposition}[theorem]{Proposition}
\newtheorem{corollary}[theorem]{Corollary}

\newtheorem{definition}[theorem]{Definition}

\newtheorem{remark}[theorem]{Remark} 
\renewcommand{\leq}{\leqslant}
\renewcommand{\geq}{\geqslant}
\newcommand{\grf}{\pi_1(\Sigma)}
\newcommand{\bgrf}{\partial_\infty\grf}

\renewcommand{\epsilon}{\varepsilon}

 \newcommand{\ms }{\mathsf}
\newcommand{\mk}{\mathfrak} \newcommand{\mc }{\mathcal}

\newcommand{\tr}{\operatorname{Tr}}

\renewcommand{\d}{{\rm{d}}} 
\newcommand{\T}{\ms T} 
 
\newcommand{\ww}{\, {\pmb\wedge}\, }
\newcommand\defeq{{}\mathrel{\mathop:}={}} 
\newcommand{\eqdef}{{}=\mathrel{\mathop:}{}}
\newcommand{\D}{D}
\newcommand{\Rh}{{\operatorname{Rh}}}

\newcommand{\UX}{\ms U X}
  
  \newcommand{\isorightarrow}{\xrightarrow{
   \,\smash{\raisebox{-0.5ex}{\ensuremath{\sim}}}\,}}

  \newcommand{\starrightarrow}{\xrightarrow{
   \,\smash{\raisebox{-0.5ex}{\ensuremath{\ast}}}\,}}

\title{Variations along the Fuchsian locus}
\author{Fran\c cois LABOURIE AND Richard WENTWORTH}\address{Univ. Nice Sophia-Antipolis, Laboratoire Jean Dieudonné, Nice F-06000.}
 \address{Department of Mathematics, University of Maryland, College Park, MD
20742, USA}
  \thanks{F.L.'s research leading to these results has received funding from the European Research Council under the {\em European Community}'s seventh Framework Programme (FP7/2007-2013)/ERC {\em grant agreement} ${\rm n^o}$ FP7-246918. R.W. was supported in part by NSF grant DMS-1406513.
  The authors also acknowledge support from NSF grants DMS-1107452, -1107263, -1107367 ``RNMS: GEometric structures And Representation varieties'' (the GEAR Network).
  } 

\begin{document}

\begin{abstract}
The main result is an explicit expression for the Pressure Metric on the Hitchin component of surface group representations into $\mathsf{PSL}(n,{\mathbb R})$ along the Fuchsian locus. The expression is in terms of a parametrization of the tangent space by holomorphic differentials, and it gives a precise relationship with the Petersson pairing. Along the way, variational formulas are established that generalize results from classical Teichm\"uller theory, such as Gardiner's formula, the relationship between length functions and Fenchel-Nielsen deformations, and variations of cross ratios.

\medskip
\noindent {\sc R\'esum\'e}. Notre r\'esultat principal est une expression explicite de la m\'etrique de pression sur la composante de Hitchin de l'espace des repr\'esentations du groupe fondamental d'une surface dans $\mathsf{PSL}(n,{\mathbb R})$ le long du lieu fuchsien.  Cette formule utilise une paramétrisation de l'espace tangent à la composante de Hitchin en terme de  diff\'erentielles holomorphes, et elle s'exprime explicitement en fonction du produit de Petersson.  Au passage, nous établissons des relations qui g\'en\'eralisent les r\'esultats classiques de la th\'eorie de Teichm\"uller, telles que la formule de Gardiner, le rapport entre fonctions de  longueur et d\'eformations de Fenchel--Nielsen, et les variations des birapports.
\end{abstract}

\maketitle

\allowdisplaybreaks

\section{Introduction} Classical Teichmüller theory provides links between complex analytic and dynamical quantities defined on Riemann surfaces with conformal hyperbolic metrics. More precisely,  properties of the geodesic flow of a hyperbolic structure are related to  holomorphic objects  on the underlying Riemann surface. The Selberg trace formula is an instance of this correspondence. The goal of this paper is to extend this relationship in the context of higher rank Teichm\"uller theory.
Specifically, in the case of Hitchin representations we 
%shall concentrate on finding 
find
analogs  to the fundamental results of Wolpert -- as well as those of Hejhal and Gardiner -- that compute variations of dynamical quantities  for deformations of the complex structure parametrized by holomorphic  differentials. 
In particular, we refer here  to Gardiner's formula \cite{Gardiner:1975ts} which computes the variation of the length of a geodesic in terms of Hejhal's periods of  quadratic differentials; the relation between the Thurston and Weil--Petersson metrics \cite{Wolpert:1986ww}; the computation of the variation of the cross ratio on the boundary at
 infinity of  surface groups and the study of Fenchel--Nielsen twists \cite{Wolpert:1983td}. 
  
 Let us be more concrete. Let $X$ be a closed Riemann surface of genus at least two, and $\Sigma$ the underlying oriented differentiable manifold. Let $\delta_X$ be the monodromy of the unique conformal  hyperbolic metric on $X$. 
 Let $\iota_n$ be the irreducible representation of $\mathsf{PSL}(2,\mathbb R)$ in $\mathsf{PSL}(n,\mathbb R)$. 
 The {\em Fuchsian point} is the representation 
 $$\delta_{X,n}=\iota_n\circ\delta_X : \pi_1(\Sigma)\to \mathsf{PSL}(n,\mathbb R)\ .$$
  A {\em Hitchin representation}  is a homomorphism $\delta:\pi_1(\Sigma)\to\mathsf{PSL}(n,\mathbb R)$ that can be continuously deformed to the Fuchsian point.  We call the set $\mc H(\Sigma,n)$ of conjugacy classes of Hitchin representations  the \emph{Hitchin component}.
 The {\em Fuchsian locus} is the subset of $\mc H(\Sigma,n)$ consisting of Fuchsian points obtained by varying the complex structure on $X$. By an abuse of terminology, we shall refer to these Fuchsian points as \emph{Fuchsian representations}   $\pi_1(\Sigma)\to \mathsf{PSL}(n,\mathbb R)$.  Furthermore, throughout this paper 
 we can and will  assume a lift  of Hitchin representations from $ \mathsf{PSL}(n,\mathbb R)$  to $\mathsf{SL}(n,\mathbb R)$.
 
 Hitchin \cite{Hitchin:1992es} proves that $\mc H(\Sigma,n)$ can be globally parametrized by the {\em Hitchin base}: $Q(X,n)=\bigoplus_{k=2}^nH^0(X, K^k)$, where $K$ is the canonical bundle of $X$. Thus, the tangent space of the Fuchsian point of the Hitchin component can also be described as $Q(X,n)$. 
   This infinitesimal parametrization, which will be crucial for our calculations, depends on some choices, and it is natural to normalize so that the restriction to the Fuchsian locus corresponds to classical deformations in Teichm\"uller space.  In fact, 
    in this paper we shall use two natural families of deformations (that is vectors in $\T_{\delta_{X,n}}\mc H(\Sigma,n)$ associated to a point $q\in Q(X,n)$) which are related by a constant depending only on $n$ and $k$:
    \begin{enumerate}
   	\item The {\em standard deformation} $\psi^0(q)$, for which the result of the computations are easier to state,
   	\item  The {\em normalized deformation} $\psi(q)$, for which the Atiyah--Bott--Goldman symplectic structure of the Hitchin components coincides along the Fuchsian locus with the symplectic structure inherited from the $L^2$-metric on the Hitchin base  (see Corollary \ref{pro:normalization}). 
   	\end{enumerate}

	Recall that the moduli space of representations $\pi_1(\Sigma)\to \ms{SL}(n,{\mathbb C})$ is a hyperk\"ahler variety \cite{Hitchin:1987}.  This structure is reflected in three algebraically distinct descriptions: the Dolbeault (Higgs bundle) moduli space, the de Rham moduli space of flat connections, and the Betti moduli space of representations.
	We exhibit isomorphisms of the tangent space to the Fuchsian point in each of these manifestations 
	 as $Q(X,n)\oplus \overline{Q(X,n)}$. 
	We furthermore show that the different points of view actually give rise to the same parametrization of the tangent space at  the Fuchsian point. 
A key point is that the first variation of the \emph{harmonic} metric for certain variations of Higgs bundles vanishes (see Theorem \ref{thm:metric-variation}).  This result may be viewed as a generalization of Ahlfors' lemma on variations of the hyperbolic metric under quasiconformal deformations by harmonic Beltrami differentials \cite{Ahlfors:1961}.	
	All this occupies   Section \ref{sec:moduli0}.  
 
  \vskip 0.2 truecm
 The discussion above is  the complex analytic side of the Hitchin component, and we now wish to relate it  to the dynamical side. In \cite{Labourie:2006}, the first author shows that if $\delta$ is a Hitchin representation and $\gamma$ a nontrivial element in $\pi_1(\Sigma)$, then $\delta(\gamma)$ has $n$-distinct positive eigenvalues. The underlying idea is to associate to a Hitchin representation a {\em geodesic flow} (see also \cite{Guichard:2011tl} and \cite{Bridgeman:2013tn}), thus giving a dynamical characterization of the Hitchin component. 
 
 This leads to the main motivation for this paper.  In \cite{Bridgeman:2013tn}, Bridgeman, Canary, Sambarino and the first author constructed a {\em pressure metric} on the Hitchin component whose restriction to the Fuchsian locus is the Weil--Petersson metric. In Section \ref{sec:var}, we shall prove the following

\begin{theorem}\label{thm:main-press0}
Let $\delta$ be a Fuchsian representation into $\ms{SL}(n,\mathbb R)$ associated to a  Riemann surface $X$ with a conformal hyperbolic metric. Let $q$ be a holomorphic $k$-differential  on $X$, $2\leq k\leq n$,   and let $\psi^0(q)$ be the associated standard deformation. Then the pressure metric is proportional to the $L^2$-metric:
\begin{equation*}
	{\bf P}_\delta\left(\psi^0(q),\psi^0(q)\right)=\frac{1}{2^{k-1}\pi\, \vert \chi(X)\vert}\left[\frac{(k-1)!(n-1)!}{(n-k)!}\right]^2\int_X \Vert q\Vert^2 d\sigma\ .
\end{equation*} 
	Moreover, two deformations associated to holomorphic differentials of different degrees are orthogonal with respect to the pressure metric.
\end{theorem}

The first ingredient in the proof of this theorem is an extension,  Theorem \ref{thm:gardiner}, of  Gardiner's formula to Hitchin representations. This computes  the first  variation of the eigenvalues of $\delta(\gamma)$ as a function of $\delta$ under a standard deformation.  The result, proven in Section \ref{sec:gard}, is a generalization of the classical  formula for  holomorphic quadratic differentials \cite{Gardiner:1975ts}. We reproduce the statement here for the highest eigenvalue.
 
\begin{theorem}{\sc [Gardiner formula]} \label{thm:gardiner0}
For Hitchin representations, 
the first variation at  the Fuchsian locus   of the largest eigenvalue $\lambda_\gamma$ of the holonomy along a simple closed geodesic $\gamma$ of hyperbolic length $\ell_\gamma$ along a standard Hitchin deformation given by $q\in H^0(X, K^{k})$, 
is
$$
\d\log {\lambda}_\gamma (\psi^0(q))=
\frac{(-1)^k(n-1)!}{2^{k-2}(n-k)!} \int_0^{\ell_\gamma} \Re\left(q({\gamma},\ldots,{\gamma})\right)\ \d t\ .
$$
\end{theorem}
The complete result, Theorem \ref{thm:gardiner}, also gives the variation of the other eigenvalues, and  Corollary \ref{thm:trace-variation} gives
 the variation of the trace. The  proof of Gardiner's original formula makes use of the theory of quasiconformal maps. For Hitchin representations, no such technique is available, and our proof is purely gauge theoretic.
 Finally, by a formula in Hejhal \cite{Hejhal:1978} (attributed to Petersson) 
the right hand side of the equation above can be interpreted as the $L^2$-pairing of $q$ with the \emph{relative Poincaré series} $\Theta_\gamma^{(k)}$ associated to $\gamma$ (cf.\ Section \ref{sec:poinc} and Proposition \ref{pro:katok}). 

The second component in the proof of Theorem \ref{thm:main-press0} is a 
 relationship,  proved in Section \ref{sec:var}, between the \emph{variance} and the $L^2$-metrics for holomorphic differentials. 
The correspondence with the variance metric in the case of quadratic differentials has been discussed using a different framework -- but belonging to the same circle of ideas-- in McMullen \cite{McMullen:2008eh}. Observe that the computation of the actual coefficients in Theorem \ref{thm:main-press0} requires some technical and careful computations. However, the fact that the two metrics are proportional is a relatively easier result that is obtained earlier in the proof once the main background has been settled.
 We conclude with a remark on the family of pressure metrics that one can define on the Hitchin component.

Let us pause to note an interesting consequence of the dependence on the degree of the differential in Theorem \ref{thm:main-press0}. We shall see later on that the Atiyah--Bott--Goldman  symplectic form $\omega_n$ on $\mathcal H(\Sigma, n)$ is also related to the $L^2$-pairing. Given the pressure metric $\bf P$ and the symplectic form $\omega$, one obtains the {\em pressure endomorphism} $A$ so that ${\bf P}(u,v)=\omega_n(A\cdotp u,v)$. Observe that $A$ is analytic on the Hitchin component. Our result is the following

\begin{corollary}\label{cor:pres-end}
The eigenspaces of $A^2$ along the Fuchsian locus are, after the identification of the tangent space with the  Hitchin base, precisely the subbundles $H^0(X,K^k)$  consisting of holomorphic differentials of degree $k$. Moreover, the induced complex structure (by $A$) on these eigenspaces coincides with the complex structure on $H^0(X,K^k)$.  
\end{corollary}

In other words, two objects from the dynamical side, the pressure metric and symplectic form, detect the decomposition along the Fuchsian locus  of the tangent space in sum of complex bundles of holomorphic differentials, a decomposition coming from the analytic side. It therefore seems interesting to study the decomposition into eigenspaces of $A^2$ everywhere on the Hitchin component, and not only along the Fuchsian locus.
Thus, the fact that the pressure metric and the symplectic form are certainly not related in a Kähler way --- which may look disappointing at a first sight --- leads to an interesting and new structure on Hitchin components. 
 \vskip 0.2 truecm

\vskip 0.2 truecm
We further investigate the symplectic geometry of the Hitchin component in Section \ref{sec:symplectic-form}. To a simple closed geodesic $\gamma$ and an element $h$ of the Cartan subalgebra of $\ms{PSL}(n,\mathbb R)$ we associate a {\em higher Fenchel--Nielsen twist $\tau_\gamma(h)$}, which is a vector field on $\mathcal H(\Sigma, n)$. In Corollary \ref{cor:bending}, we show that these  twist deformations are represented by linear combinations of the relative Poincaré series of $\gamma$ (of different degrees). The  Hamiltonian vector fields of the eigenvalues of $\delta(\gamma)$, viewed as functions of the representation $\delta$, are expressed as linear combinations of the twist deformations about $\gamma$.  We also prove the following generalization of  \cite[Theorem 2.4]{Wolpert:1983td}.

\begin{theorem}{\sc[Reciprocity of the twist deformation]}
For all integers $k$, $p$, $1\leq p\leq n$, $2\leq k\leq n$,  and any simple closed geodesics $\alpha$, $\beta$, the following holds at the Fuchsian locus:
$$
\d\log\lambda^{(p)}_\alpha\left(\psi^0(i\cdotp\Theta_\beta^{(k)})\right)= -\d\log\lambda^{(p)}_{\beta}\left(\psi^0(i\cdotp\Theta_\alpha^{(k)})\right)\ ,
$$
where $\lambda_\gamma^{(p)}(\delta)$ is the $p^{\hbox{\tiny th}}$-largest eigenvalue of $\delta(\gamma)$.
\end{theorem}  
\vskip 0.2 truecm
In Section \ref{sec:opers}, we compute the variation for the cross ratio under Hitchin deformations, generalizing \cite[Lemma 1.1.]{Wolpert:1983td}. We give two formulations of this result: Theorem \ref{thm:cr1} using a generalized period that we call a {\em  rhombus function}, and  Theorem \ref{thm:cr-auto} using automorphic forms. 
%Both results requiring too much preliminaries to be detailed in this introduction. 
We also comment on the triple ratios for $\mathsf{SL}(3,\mathbb R)$.
\vskip 0.2 truecm
Finally, in Section \ref{sec:appl}, we provide  two applications to large $n$-asymptotics. 
First, in Theorem \ref{thm:large-n} we show that (after a further normalization) the pressure metric converges for large rank to a multiple of the $L^2$-metric. In this situation,  it is more natural to consider the {\em renormalized highest eigenvalue}
$
\mu_\gamma=\lambda^{\frac{1}{n-1}}_\gamma
$, 
and the associated renormalized pressure metric. The reason for this choice is so that  the highest eigenvalue does not depend on $n$ along the Fuchsian locus.
 
We prove the following 
\begin{theorem}{\sc [Large $n$-asymptotics]} \label{thm:large-n0}
The large $n$ asymptotics for the renormalized  pressure metric  and renormalized deformation $\psi(q)$ associated to  a holomorphic $k$-differential $q$ is given by
$$
{\bf P}(\psi(q),\psi(q))\sim\frac{(2k-1)!}{2^{k-1}\cdotp 3 \pi \,\vert\chi(X)\vert}\int_X \Vert q\Vert^2\d \sigma\ .
$$
\end{theorem}

This theorem provides a link between two large $n$-asymptotic theories of the Hitchin component. In \cite{Hitchin:2015wz}, Hitchin argues that one can build a Higgs bundle theory for $\ms{SU}(\infty)$, regarded as  the group of symplectic diffeomorphisms of the sphere.  The Hitchin base is $\bigoplus_{n=2}^\infty H^0(K^n)$, and thus describes a large $n$-analytic side. On the other hand, in \cite{Labourie:2005} the first author has shown  that there is a Hitchin component $\mc H(\infty,\mathbb R)$ of representations of $\grf$  in $\mathsf{SL}(\infty,\mathbb R)$, where the latter is considered to be (a subgroup) of the group of Hamiltonian diffeomorphisms of the annulus. In this approach, all Hitchin components embed in $\mc H(\infty,\mathbb R)$, and all representations in  $\mc H(\infty,\mathbb R)$ are associated to geodesic flows and spectra, thus providing a dynamical side to the story. It is therefore tempting to try to understand in which way, at least formally, these versions of $\mathsf{SL}(\infty)$ and $\mathsf{SU}(\infty)$ have the same complexification.

The second application is motivated by the question raised at the end of the introduction of \cite{Labourie:2005}, a somewhat simpler version of which was posed to us by Maryam Mirzakhani: are the geodesic currents arising from Hitchin representations dense in the space of all geodesic currents? An even simpler test question is the following. Given $p$ conjugacy classes of pairwise distinct primitive elements $\gamma_1,\ldots,\gamma_p$ in $\pi_1(\Sigma)$, let $$
\Lambda^\infty_{\vec\gamma}\defeq\cup_{n=2}^\infty \left\{(\log\lambda_{\gamma_1}(\delta), \ldots,\log\lambda_{\gamma_p}(\delta))\in \mathbb R^p \mid \delta\in \mc H(\Sigma,n)\right\}\ .
$$ 
Then
what is the closure of  $\Lambda^\infty_{\vec \gamma}$ in $(\mathbb R+)^p$? We prove in Theorem \ref{imageopen} that this set has nonempty interior. 
This result should be compared with a result of Lawton, Louder and McReynolds \cite{Lawton:2013tl} which states that two elements of $\pi_1(\Sigma)$ have different traces for a certain linear representation.

\vskip 0.2 truecm
\emph{Acknowledgments.}   The authors warmly thank J{\o}rgen Andersen, Marc Burger, Dick Canary, David Dumas, Maryam Mirzakhani, Andy Sanders, Mike Wolf, Scott Wolpert and Alex Wright for useful conversations related to the material in this paper. They also express their appreciation for the hospitality of the Institute for Mathematical Sciences at the National University of Singapore and the Mathematical Sciences Research Institute in Berkeley, where much of this work was carried out. Finally, the referee's careful reading of the manuscript and helpful suggestions are gratefully acknowledged.
 
 \tableofcontents

\section{Preliminaries}
In this section we present background material. First, in Section \ref{sec:holom} we give a review of holomorphic differentials, introducing periods, relative Poincaré series and the Petersson and $L^2$-pairings.  Section \ref{sec:Lie} is a review of what we shall need from Lie theory. Next, in Section \ref{sec:HB} we present a summary of Higgs bundles   for general complex semisimple Lie groups. Finally, we give the definition of opers (here only for the case of $\mathsf{SL}(n,\mathbb C)$) in Section \ref{sec:defoper}.

\subsection{Holomorphic differentials}\label{sec:holom}

\subsubsection{The $L^2$-metric} \label{sec:l2}
Let $X$ be a closed Riemann surface of genus $\geq 2$.
We will always assume a  conformal metric on $X$ is the hyperbolic metric of constant curvature $-1$.
 We denote the area form by $\d\sigma$, and the Hodge operator by $\ast$. The metric induces a hermitian structure $\braket{\ ,\ }$ on $K$, and hence on $K^k$. For $q_1, q_2$ smooth $k$-differentials,
 we define the {\em $L^2$-metric}
 \begin{equation*}
 \braket{q_1, q_2}_X\defeq \int_X \braket{q_1, q_2}\d\sigma	\ .
 \end{equation*}
In local holomorphic coordinates $z=x+iy$, and with slight abuse of notation 
the area form may be written $\d\sigma=\sigma(z) \d x \wedge \d y $ for a locally defined function $\sigma(z)$. Let $h=(\braket{\d z, \d z})^{-1}$.
Then
$2h=\sigma$, so that $\d\sigma=ih(z)\d z\wedge \d\bar z$.  
It follows that if $q_i=q_i(z)\d z^k$, 
\begin{equation} \label{eqn:L2}
\langle q_1, q_2\rangle_X = \int_X q_1(z)\overline{q_2(z)} h^{-k}(z)\,  
\d \sigma\ .
\end{equation}
\noindent
{\bf Warning:}  The definition above differs from the usual \emph{Petersson pairing} \cite{Petersson:1939}: 
$$
\langle q_1, q_2\rangle_{\rm P}\defeq \int_X q_1(z)\overline{q_2(z)} \sigma^{-k}(z)\,  \d\sigma=2^{-k}\langle q_1, q_2\rangle_X\ .
$$

Let $X=\Gamma\backslash {\mathbb H}^2$ be the uniformization of $X$ coming from the hyperbolic metric.  We will assume that $\Gamma$ has been lifted (once and for all) to a discrete subgroup of $\ms{SL}(2,\mathbb R)$.

\subsubsection{Integration along geodesics} \label{subsec:integration}
 Let $\pi:\UX\to X$ be the unit tangent bundle of $X$, equipped with the  Riemannian metric induced from $X$. Let $\phi_t$ be the geodesic flow and $\mu$ the Liouville measure normalized to be a probability measure. In general, if $f$ is a function on $\UX$, and $\gamma:[0,\ell_\gamma]\to \UX$ a geodesic arc, we shall  write
 \begin{equation}
 	\int_\gamma f\, \d s\defeq \int_0^\ell f(\phi_s(\gamma(0)))\,\d s
 \end{equation}
 For integers $k$, a smooth section  $q$ of $K^k$ defines a complex valued function   $\hat q: \UX\to \mathbb C$ that is homogeneous of degree $k$ with respect to the $S^1$-action on $\UX$.
%\begin{equation}
%\hat q: u\mapsto \hat q(u)=q(\underbrace{u,\ldots,u}_k).
%\end{equation}
If $\gamma$ is a unit speed geodesic with parameter $s$ in $X$, we will also  use the alternative notations 
\begin{equation*}
\int_0^{\ell_\gamma}  q( \gamma, \ldots,  \gamma) \, \d s \defeq \int_\gamma \hat q\, \d s.
\end{equation*}
In this notation we regard $q\in K^k$  as a $k$-$\mathbb C$-multilinear  form on $\T X$. Then $\overline q\in \overline K^k$ is defined by
$$
\overline q(u_1,\ldots,u_k)=\overline{q(u_1,\ldots,u_k)}\ .
$$
so that 
\begin{equation} \label{eqn:conjugate}
\int_0^{\ell_\gamma}   \overline q(\gamma, \ldots,  \gamma) \, \d s =\overline{\int_0^{\ell_\gamma}  q( \gamma, \ldots,  \gamma) \, \d s}\ .
\end{equation}

\subsubsection{Relative Poincaré series}\label{sec:poinc}
	Let $q$ be a holomorphic differential on $\mathbb H^2$. Let $\Gamma$ be a Fuchsian group then the following series, when it exists and converges,
$$
\Theta(q)=\sum_{\eta\in\Gamma}\eta^*(q)\ ,
$$
is called the {\em Poincaré series} of $q$ and is $\Gamma$ invariant. 
For any pair of distinct points $u,U\in\partial_\infty\mathbb H^2\sqcup \mathbb H^2$, let
\begin{equation} \label{eqn:theta}
\theta_{u,U}(z)=\frac{(\breve U-\breve u)\d z}{(z-\breve u)(z-\breve U)}\ ,
\end{equation}
where $\breve u$ and $\breve U$ are the endpoints at infinity of the geodesic joining $u$ and $U$ with the obvious convention.
Observe that for any $\eta\in\ms{PSL}(2,\mathbb R)$, we have
$$
\eta^*\theta_{u,U}=\theta_{\eta^{-1}(u),\eta^{-1}(U)}\ .
$$ 
 Let $\gamma$ be a closed geodesic associated to an element (also called $\gamma$) of $\Gamma$, and $u, U$ the repelling (resp.\ attracting) fixed points of $\gamma$ in $\partial\mathbb H^2$. Observe that $\gamma^*\Theta_{u,U}=\Theta_{u,U}$. The {\em relative Poincaré series (of order $k$) of $\gamma$}, is 
$$
\Theta^{(k)}_\gamma\defeq\sum_{\eta\in\Gamma/\langle\gamma\rangle}\eta^*\theta^{k}_{u,U}\ .
$$
From the definition, one immediately sees that
\begin{eqnarray}
	\Theta^{(k)}_{\gamma^{-1}}&=&(-1)^k\Theta^{(k)}_\gamma\label{thetag-1}\ .
\end{eqnarray}
By direct computation (see \cite[eq.\ (77)]{Hejhal:1978})  we find

\begin{proposition} \label{pro:katok}
For any automorphic form $q$ of degree $k$,
\begin{equation} \label{eqn:katok}
\int_\gamma\hat q\,\d s= r_k\cdotp\braket{q,\Theta^{(k)}_\gamma}_X\ ,
\end{equation}
where 
\begin{equation} \label{eqn:rk}
r_k=\frac{(-1)^k2^{k-2}((k-1)!)^2}{(2k-2)!\pi}\ .
\end{equation}
\end{proposition}

\begin{remark} \label{rem:hejhal}
Note that if $\ell_\gamma$ denotes the length of $\gamma$, then the theta series in \cite[eq.\ (56)]{Hejhal:1978} is 
$=(2\sinh(\ell_\gamma/2))^{-k}\cdotp \Theta^{(k)}_\gamma$.
\end{remark}

\subsection{Lie theory}\label{sec:Lie}
\subsubsection{Principal $\mk{sl}(2)$ subalgebras} Here we review material from Kostant
\cite{Kostant:1959wi} using (partly) his convention about the generators.
Let $\ms G$ be a complex semisimple Lie group $\ms G$ of rank $l$ with  Lie algebra $\mk g$.
Let $(\cdotp , \cdotp )_{\mk g}$ denote the Killing form of $\mk g$, normalized so that the squared length of a longest root is $2$.
Fix generators for $\mk{sl}(2)$:
\begin{equation}\label{eqn:convention}
[a,x]=x\ ,\ [a,y]=-y\ , \ [x,y]=-a\ .
\end{equation}
Here we use a different convention for the sign of $x$ from that of Kostant.
   With this understood, the defining representation is denoted $\kappa_2: \mk{sl}(2)\simeq \mk{sl}({\mathbb C}^2)$. Notice here that, for the sake of coherence with \cite{Kostant:1959wi}, we use a (nonalgebraic) convention, which differs from that of Bourbaki.
   \begin{equation} \label{eqn:k2}
   \kappa_2(a)=\left(\begin{matrix} 1/2&0\\ 0&-1/2\end{matrix}\right)\ ,\
   \kappa_2(x)=\left(\begin{matrix} 0&-1/\sqrt 2\\ 0&0\end{matrix}\right)\ ,\ 
   \kappa_2(y)=\left(\begin{matrix} 0&0\\ 1/\sqrt 2 &0\end{matrix}\right)\ .
  \end{equation}
 Consider a principal $\mk{sl}(2, {\mathbb C})$ embedding
 $\kappa_{\mk g}:  \mk{sl}(2)\hookrightarrow \mk g$. Then $\mk g$ decomposes into irreducible representations of the principal $\mk{sl}(2)$,
 \begin{equation} \label{eqn:sl2}
 \mk g=\bigoplus_{i=1}^l \mk v_i\ ,
 \end{equation}
 with $\dim_{\mathbb C} \mk v_i=2m_i+1$.
  The numbers $m_i$ are all distinct and ordered so that 
 $$
 m_{i}<m_{i+1}\ .
 $$
They are called  the {\em exponents} of the group $\ms G$.  The grading by the element
$a$ gives the decomposition
\begin{equation}
  \mk g=\bigoplus_{m=-m_{\ell}}^{m_{\ell}}\mk g_m\ .\label{eqn:grading}
\end{equation}

Given a Cartan involution  $\rho$ of $\mk g$, fixing the principal subalgebra, and such that $y=\rho(x)$,
 we shall say that a basis $\{ e_1, \ldots, e_l\}$ (resp.\  \ $\{ f_1, \ldots, f_l\}$) of highest (resp.\  \ lowest) weight vectors for the $\mk{sl}(2,{\mathbb C})$ action is \emph{normalized} if
\begin{enumerate}
\item $f_i=\rho(e_i)$, and
\item for each $k$, 
\begin{equation} \label{eqn:normalized-weights}
-(e_k, f_k)_{\mk g}=d({\mk g})\ , \end{equation}
where $d({\mk g})$
 is the Dynkin index for the principal embedding,
 %$\kappa_{\mk g}:\mk{sl}(2)\hookrightarrow {\mk g}$, 
which is  defined by
\begin{equation} \label{eqn:dynkin}
 d({\mk g})\defeq \frac{(\kappa_{\mk g}(a), \kappa_{\mk g}(a))_{\mk g}}{\tr(\kappa_2(a)^2)}\ .
\end{equation}
\end{enumerate}
Observe that a normalized highest weight vector is uniquely determined up to multiplication by an element of $S^1$. 
The justification for the introduction of this normalization will appear in Section \ref{sec:symplectic-form}.

In the case $\ms{G}=\ms{SL}(n, {\mathbb C})$, we can make a more explicit choice of  highest weight vectors.
Let $\kappa_n: \mk{sl}(2)\to \mk{sl}(n)$ denote the  principal embedding (unique up to conjugation), viewed as a linear embedding via the defining representation of $\ms{SL}(n, {\mathbb C})$.  
We will use two choices of highest and lowest weight vectors $E^0_k, F^0_k\in \mk{sl}(n, {\mathbb C})$, for the representation of $\mk{sl}(2)$. First, let
$X=\kappa_n(x)$, $Y=\kappa_n(y)=\rho(X)$, where $\rho(M)=-M^\ast$ is the standard Cartan involution.
Then define the {\em standard highest and lowest weight vectors}
 \begin{equation} \label{eqn:canonical-highest-weight}
 E^0_k\defeq \left(-\sqrt{2}X\right)^k\ , \ F^0_k\defeq\rho(E^0_k)= -\left(\sqrt{2}Y\right)^k\ .
 \end{equation}
We now renormalize these highest and lowest weights with respect to the trace to obtain the {\em normalized vectors}: \begin{equation} \label{eqn:normalized-highest-weight}
 E_k=\eta_k E^0_k \ ,\  F_k=\rho(E_k)\ ,
  \end{equation}
  where $\eta_k$ is a positive number so that $E_k$ and $F_k$ are normalized in the sense discussed above.
That is,
\begin{equation} \label{eqn:eta}
\eta^2_k=-\frac{d(n)}{\tr(E^0_k F^0_k)}\ ,
\end{equation}
 where  the Dynkin index $d(n)$  of the principal $\mk{sl}(2)\hookrightarrow\mk{sl}(n)$ is (cf.\  \cite{Panyushev:2015vy})
\begin{equation} \label{eqn:dynkin-sln}
d(n)= \frac{\tr(\kappa_n(a)^2)}{\tr(\kappa_2(a)^2)}=
{n+1\choose 3}\ .
\end{equation}
An explicit calculation of the normalization  and the constant $\eta_k$ is given in Corollary \ref{cor:etak} in the Appendix: 

\begin{equation}\label{eqn:etak}
	\eta_k=
	\frac{1}{k!}{n+1 \choose 3}^{1/2}
	\cdot
	{n+k\choose 2k+1}^{-1/2}\ .
\end{equation}

\subsubsection{A useful basis} 
We will need the following technical result which is probably well-known.  

\begin{lemma} \label{lem:dual-basis}
Given a choice of highest weight vectors $\{e_i\}_{i=1}^{\ell}$ there is a basis $\{h_j\}_{j=1}^{\ell}$ of the centralizer ${\mk z}(X-Y)_{\mathbb R}$, such that $(e_i,  h_j)_{\mk g}=\delta_{ij}$.
\end{lemma}

We shall call  $(h_i)_{i=1,\ldots,\ell}$ the {\em principal basis} of $\mk z (X-Y)$.
\begin{proof} Recall the decomposition \eqref{eqn:sl2}.  This is an orthogonal decomposition for the Killing form. 
%Let $e_i, f_i\in \mk v_i$ be highest and lowest weight vectors in $\mk v_i$. 
First, observe that since $X-Y$ is conjugate to a multiple of $a$, 
$$
\dim{\mk z}(X-Y)\cap \mk v_i=\dim{\mk z}(a)\cap \mk v_i=1\ .
$$
To conclude, we need to prove that the Killing pairing with $e_i$ defines a nonzero form on ${\mk z}(X-Y)\cap \mk v_i$. Recall that $\mk v_i$ is generated by $e_i$ and its iterated images under $\operatorname{ad}(X-Y)$. In particular, if $w\in{\mk z}(X-Y)$ and $(w,e_i)_{\mk g}=0$, then for all $k$, $(w,\left(\operatorname{ad}(X-Y)\right)^ke_i)_{\mk g}=0$ and thus $w=0$. We can therefore find an element $h_i$ in $\mk v_i$ so that $(h_i,e_i)_{\mk g}=1$. The result follows.

%
% we claim that $\dim{\mk z}(X-Y)\cap{\mathbb C}\cdotp f_i=1$. Indeed, suppose $v\in {\mk z}(X-Y)\cap V_i$.  
%Since ${\rm ad}_{X-Y}$ annihilates $v$, 
%
%
%the coefficients of the weight spaces with weights: $-m_i+1, -m_i+3, \ldots, m_i-1$  must vanish. In the same way, if the coefficient of $v$ on $f_i$ vanishes, then it follows that the coefficient of $(-m_i+2)$-weight vector also vanishes. Similarly for the  $(-m_i+4)$-weight vector, etc.  Hence, $v$ necessarily vanishes, and we conclude that 
%$$\dim{\mk z}(X-Y)\cap  V_i=\dim{\mk z}(X-Y)\cap{\mathbb C}\cdotp f_i\leq 1$$
% But $X-Y$ is conjugate to a multiple of $a$ and thus $\dim {\mk z}(X-Y)=\ell$, so the dimension of the intersections above is exactly one.
%Next, let us remark that for a given $i$  we have that $(e_i, v)_{\mk g}=0$ for $v\in V_j$, $j<k$. This is because the weights appearing in $V_j$ are strictly greater than $-m_i$.  
%  Finally, the lemma follows from the usual Gram-Schmidt procedure. More precisely,
%suppose we have chosen $h_i\in {\mk z}(X-Y)_{\mathbb R}\cap V_i$, $1\leq i\leq k$, such that $(e_j, h_i)_{\mk g}=\delta_{ij}$.
%By the previous argument we may choose $\tilde h_{k+1}\in {\mk z}(X-Y)_{\mathbb R}\cap V_{k+1}$ such that $(e_{k+1} ,\tilde h_{k+1})_{\mk g}=1$. Then we set
%$$
%h_{k+1}=\tilde h_{k+1}-\sum_{i=1}^k (e_i ,\tilde h_{k+1}) v_i
%$$
%In this way, we define $h_i$ for all $i$, and the result follows.
\end{proof}

\subsubsection{The involution associated to the real split form}\label{sec:invs}
We  briefly recall material from \cite{Kostant:1959wi}, which is also explained in \cite{Hitchin:1992es,Baraglia:2015}. The choice of a principal subalgebra with its standard generators defines a Cartan subalgebra and a ($\mathbb C$- antilinear) Cartan involution $\rho$. Moreover, one can define a $\mathbb C$-linear involution $\sigma$  characterized by 
$$
\sigma(e_i)=-e_i \ ,\  \sigma(Y)=-Y\ . 
$$
Then $\sigma$ and $\rho$ commute. Furthermore, the {\em real split involution} $\lambda=\sigma\circ\rho$ is such that its set of  fixed points is a real split subalgebra $\mk g_0$ of $\mk g$. The complexification of $\mk g_0$ is $\mk g$.

\subsection{Higgs bundles}\label{sec:HB}

\subsubsection{The Hitchin component}
A representation of  $\pi_1(\Sigma)$ into the real split form $\ms G_\mathbb R$ of $\ms G$ is {\em Fuchsian} if it is conjugated to a discrete faithful representation taking values in the principal $\ms{SL}(2,\mathbb R)$. A representation is {\em Hitchin} if it can be continuously deformed into a Fuchsian representation. The {\em Hitchin component} is then (a connected component) of the space whose points are equivalence classes of Hitchin representations up to conjugacy by an element of ${\ms G}_{\mathbb R}$. We denote the Hitchin component by   $\mc H(\Sigma,\ms G_{\mathbb R})$.

\subsubsection{Higgs bundles and the self duality equations}

\begin{definition} \label{def:higgs}
A \emph{$\ms G$-Higgs bundle} is a pair $(P,\Phi)$, where $P$ is a holomorphic principal ${\ms G}$-bundle  $P\to X$, 
${\mathcal G}_P=P\times_{\ms G}{\mk g}$ is the associated holomorphic adjoint bundle, and 
 $\Phi\in H^0(X, {\mathcal G}_P\otimes K)$. 
\end{definition}

Let ${\ms K}\subset{\ms G}$ be a maximal compact subgroup.  
We will regard a reduction of $P$ to a principal $\ms K$-bundle $P_{\ms K}$ as arising from a smooth family of Cartan involutions on the fibers of ${\mathcal G}_P$. By a slight abuse of notation, we denote such a  family by $\rho$. A connection on $P_{\ms K}$ induces a covariant derivative $\nabla$ on the vector bundle ${\mathcal G}_P$ satisfying $\nabla(\rho)=0$.  Conversely, given $\rho$ and a connection on $P$ that is compatible with the holomorphic structure, there is a uniquely determined connection on $P_K$
called the Chern connection.  The curvature $F_\nabla$ of such a connection is a section of $\Omega^2(X, {\mathcal G}_{P_{\ms K}})$.

With this understood, we introduce Hitchin's equations.
\begin{align}
\begin{split}  \label{eqn:hitchin}
F_\nabla-[\Phi, \rho(\Phi)]&=0\ ,\cr
\nabla(\rho)&=0\ .
\end{split}	
\end{align}
For a given $\ms G$-Higgs bundle $(P,\Phi)$, eq.\  \eqref{eqn:hitchin} may be viewed as equations for $\rho$.  A solution $\rho$ to eq.\ \eqref{eqn:hitchin} will be called a \emph{harmonic metric} \cite{Donaldson:1987, Corlette:1988, Labourie:1991}. As a consequence of the Hitchin equations, the {\em associated connection}
$$
D=\nabla+\Phi-\rho(\Phi)\ ,
$$
is flat, and the associated representation of $\grf\to {\ms G}$ is called the {\em monodromy} 
 of the solution of the self duality equations.

In this paper we will not need the details of the notion of (semi)stability  of Higgs bundles, other than to note the following (cf.\ \cite{Hitchin:1987, Simpson:1988ch, Bradlow:2003}).

\begin{theorem}[{\sc Hitchin-Kobayashi correspondence}] \label{thm:HK}
A stable $\ms G$-Higgs bundle admits a harmonic metric, {\em i.e.}\ a solution to \eqref{eqn:hitchin}.
\end{theorem}

\subsubsection{Hitchin sections}

We recall here the construction  of the Hitchin section, which depends on a normalization. 
First, the {\em Hitchin base} is defined as
\begin{equation} \label{eqn:q}
Q(X, {\mk g})\defeq\bigoplus_{i=1}^l H^0(X, K^{m_i+1})\ ,
\end{equation}
For  $q=(q_1,\ldots,q_l)\in Q(X,\mk g)$, we define the {\em normalized Hitchin deformation}
\begin{eqnarray}
\phi(q)&=&\sum_{i=1}^{l} q_{i}\otimes e_i\in\Omega^{1,0}(X,\mc  G)\ ,
\end{eqnarray}
where $\{e_i\}_{i=1}^l$ are \emph{normalized} highest weight vectors for  the principal $\mk{sl}(2)\hookrightarrow \mk{g}$ (see equation \eqref{eqn:normalized-weights}).
In the case of $\ms{SL}(n,\mathbb R)$ it is convenient to introduce the {\em standard Hitchin deformation}
\begin{eqnarray}
\phi^0(q)&=&\sum_{i=1}^{l} q_{i}\otimes E^0_i\in\Omega^{1,0}(X,\mc  G)\ ,
\end{eqnarray}
We shall choose our {\em Hitchin section} 
 $\mathcal L: Q(X, {\mk g})\rightarrow \mathcal M_{Dol}(X,\ms{G})$ to be
\begin{equation}\label{eqn:hitchin-section}
{\mathcal L}(q_1,\ldots,q_l)\defeq (\mathcal G,Y+\phi(q))\ .  
\end{equation}
  Then by  \cite{Hitchin:1992es} we have
  \begin{theorem}[{\sc Hitchin section}]
  	For any $q$, the Higgs bundle $\mc L(q)$ admits a unique solution of the self duality equations \eqref{eqn:hitchin}. If $\delta(q)$ the monodromy of the corresponding solution modulo conjugacy, then the map $q\to\delta(q)$  is a diffeomorphism from the Hitchin base $Q(X, {\mk g})$ to the Hitchin component $\mathcal H(\Sigma, \ms{G}_{\mathbb R})$.
  	\end{theorem}
  
\subsubsection{The Fuchsian Higgs bundle} \label{sec:fu-hi-bu}
 
 Let us consider the holomorphic vector bundle 
 $$
 \mc G_0= K^{-1}\oplus {\mathcal O}\oplus K\ .
 $$
 Choosing a spin bundle $S$ on $X$ identifies $\mc G_0$ with the $\mk{sl}(2,\mathbb C)$ bundle of trace free endomorphisms of $S\oplus S^{-1}$. Actually, this identification is independent on the choice of $S$. The hyperbolic metric defines a metric on $S\oplus S^{-1}$, and hence a Cartan involution on $\mc G_0$. Finally, the canonical bracket map viewed as a holomorphic section of $K\otimes K^{-1}\subset K\otimes\mc G_0$,
 defines a holomorphic section $\Phi_0\in\Omega^{1,0}(X,\mc G_0)$. The hyperbolic metric defines a connection on $\mc G_0$, and all together $\nabla,\rho,\Phi_0$ satisfy
\eqref{eqn:hitchin}.

 \subsubsection{The Fuchsian $\ms G$-Higgs bundle} 
 Let $\ms G$ be a complex Lie group equipped with a choice of a principal  $\mathsf{SL}(2,\mathbb C)$ with its canonical generators $a,X,Y$. We use the grading defined in eq.\ \eqref{eqn:grading} to define a holomorphic bundle $\mathcal G$ by
\begin{equation} \label{eqn:G}
\mc G\defeq \bigoplus_{m=-m_{l}}^{m_{l}}\mk g_m\otimes K^m\ .
\end{equation}
Since the complex vector bundle underlying $\mc G$ will later 
 be equipped with another holomorphic structure, 
we will refer to \eqref{eqn:G}  as the {\em split holomorphic structure}. Observe that now $\mc G_0$ maps into $\mc G$ by 
\begin{eqnarray*}
\mc G_0&\to&  \left(\mk g_{-1}\otimes K^{-1}\right)\oplus \mk g_0\oplus \left(\mk g_1\otimes K\right),\cr
(u,v,w)&\mapsto& (Y\otimes u, a\otimes v, W\otimes w_0)\ .
\end{eqnarray*}
Thus the Higgs field $\Phi_0\in\Omega^{1,0}(\mc G_0)$ defined in the previous paragraph gives rise to a Higgs field, also denoted by $\Phi_0\in \Omega^{1,0}(\mc G)$.
By definition, the equivalence class of the Higgs bundle $(\mc G, \Phi_0)$ is the {\em Fuchsian point} in ${\mathcal M}_{Dol}(X, \ms G)$.

Observe that the family of  Cartan involutions on $\mc G_0$ 
defined in the previous paragraph extends  to a section of  Cartan involutions on $\mc G$ (also denoted $\rho$).  Similarly the hyperbolic metric connection extends to a connection.  This connection, also denoted $\nabla$ on $\mc G$, is compatible with the holomorphic structure and metric: $\nabla(\rho)=0$. In other words, $\nabla$ is the unique  $\rho$-compatible Chern connection on $\mc G$.
Then, altogether the  connection $(\nabla,\Phi_0, \rho)$ solves Hitchin equations \eqref{eqn:hitchin}.

In the special case of $\ms{SL}(n,\mathbb C)$, it is useful to consider the vector bundle 
\begin{equation}
\mathcal E=\operatorname{Sym}^{n-1}(S\oplus S^*)=\bigoplus_{p=1}^n S^{2p-n-1}\ ,\label{eqn:E}
	\end{equation}
	where $S$ is a spin bundle. Then $\mc G$ will be the (trace free) endomorphism bundle of $\mc E$.

\subsection{Nonslipping connections and opers} \label{sec:defoper}
 In this section, we restrict ourselves to the case $\ms G=\ms{GL}(n,\mathbb C)$ and refer to \cite{Beilinson:2005uk, BenZvi:2001vk} for  general $\ms G$ (see also:  \cite{Dickey:1997un}, \cite{vanMoerbeke:1998wl}, \cite{Guha:2007wf} the original reference \cite{Drinfelcprimed:1981ua} and the geometric version \cite{Segal:1991} for further discussion). 
Let $\mc P\to X$ be a holomorphic vector bundle. A {\em holomorphic filtration} of $\mc P$ is a family  $\{\mc F_p\}_{1\leq p\leq n}$ of holomorphic subbundles of $\mc P$ such that

\begin{itemize}
	\item $\mc F_n=\mc P$, 
	\item $\mc F_{p-1}\subset \mc F_{p}$
	\item  ${\rm rank}(\mc F_p)=p$.
\end{itemize} 
%We  also say that a (local) family of sections $\{e_1,\ldots,e_n\}$ of $\mc P$ is a {\em basis for the filtration} if for every integer $p$ no greater than $n$, for every $x\in S^1$,  $\{e_1(x),\ldots,e_p(x)\}$ is a basis of the fiber of $\mc F_p$ at $x$.
%

\begin{definition}{\sc[Opers]}\label{def:opers}
A holomorphic connection $D$ on $\mc P$ equipped with a holomorphic filtration $\{\mc F_p\}_{1\leq p\leq n}$  is {\em nonslipping} if it satisfies the following conditions
\begin{itemize}
\item  $\nabla \mc F_p\subset \mc F_{p+1}$ for all $p$,
\item If $\alpha_p$ is the projection from $\mc F_{p+1}$ to $\mc F_{p+1}/F_p$, then the map $$
(X,u)\to \alpha_p(D_{X}(u)),
$$ 
considered as a linear map from $ \mc F_p/\mc F_{p-1}=S^{n+1-2p}\to K\otimes \mc F_{p+1}/\mc F_{p}=S^{n+1-2p}$, is the identity.
\end{itemize}
A {\em nonslipping connection} is also called a $\ms{GL}(n,\mathbb C)$-{\em oper}.
\end{definition}

\subsubsection{The Veronese oper and the nonsplit holomorphic structure}
\label{sec:vero-oper}

Let $S$ be a spin bundle on the Riemann surface $X$, so that $S^2=K$. Let ${\mathcal E}_{\operatorname{op}}:=J^{n-1}(S^{1-n})$ be the holomorphic rank $n$  bundle of $(n-1)$-jets of holomorphic sections of $S^{1-n}$.
Let $\mc F_p$ be the vector subbundle of ${\mathcal G}_{\operatorname{op}}$  defined by 
$$
\mc F_p:=\{j^{n-1}\sigma\mid j^{n-p-1}\sigma=0\}.
$$
The family  $\{\mc F_p\}_{1\leq p\leq n}$ is a holomorphic filtration of ${\mathcal G}_{\operatorname{op}}$: we have  $\mc F_n={\mathcal E}_{\operatorname{op}}$, $\mc F_{p-1}\subset \mc F_{p}$ and ${\rm rank}(\mc F_p)=p$. Observe furthermore that 
$$
\mc F_p/\mc F_{p-1}=K^{n-p}\otimes S^{1-n}=S^{n+1-2p}\ .
$$
In particular, the graded bundle associated to the filtration is given by $\mc E$ in \eqref{eqn:E}.
We let  ${\mathcal E}_{\operatorname{op}}$  be the same underlying complex vector bundle, but with the holomorphic structure induced by $\overline\partial_{D}$. We will call this the \emph{oper holomorphic structure}.
The {\em Veronese oper} or {\em Fuchsian oper} is $\mc E_{\operatorname{op}}$  equipped with the above filtration and the Fuchsian holomorphic connection $\D$.

\section{Moduli spaces and the Tangent Space at the  Fuchsian point}\label{sec:moduli0} 
By construction, every Riemann surface defines a {\em Hitchin parametrization} of the Hitchin component by the Hitchin base. In particular, the tangent space at the Fuchsian point (the image of zero under this parametrization) is identified with the Hitchin base.

However, working directly with this description of the tangent space at the Fuchsian point might not be very handy since, at first sight, it involves solving an elliptic PDE. We rather describe another approach which will be more helpful in the sequel: roughly speaking we will describe the tangent space of Hitchin component at the Fuchsian point as the real part of an {\em oper deformation}.

Let us summarize here the construction in the following  
\begin{proposition} \label{prop:derham-oper}
Let $q\in Q(X,{\mk g})$ be an element of the Hitchin base. Let $\lambda$ be the real split involution,  $\D$  the flat connection at the $\ms G$-Fuchsian point, and  $\psi(q)=\phi(q)+\lambda(\phi(q))\in \Omega^1(\Sigma,\mk g)$. Then,
$$
\d_D\phi(q)=0, \ \ \d_D\psi(q)=0\ .
$$ 
Furthermore, passing to cohomology, the map $\psi$ realizes an isomorphism  $\psi: Q(X, {\mk g})\isorightarrow H^1_\D(\mk g)$, which coincides with the isomorphism coming from the Hitchin parametrization.
	\end{proposition}

In this proposition, we refer to the renormalized Hitchin section, but the same statement clearly holds for the standard Hitchin section.	

By the previous proposition $\phi(q)$ can be considered as an element of $H^1_{dR}(\Sigma,\mc G)$. It is the tangent vector to a one parameter family of flat connections: $D_t=D+t\phi(q)$, which we will call an {\em oper deformation}.
%By the previous proposition $\phi(q)$   can be considered as an element of $H^1_{dR}(\Sigma,\mc G)$. In order to avoid confusion, we shall write $\phi^{\mc O}(q)$ when we think of $\phi(q)$ as an element of  $H^1_{dR}(\Sigma,\mc  G)$ that is as an {\em oper deformation}.
%

In the course of proving Proposition \ref{prop:derham-oper}, we will actually describe and relate the Fuchsian points in various moduli spaces, parametrize their tangent spaces and spend some time describing intermediate results of independent interest.

\subsection{Moduli spaces}

\label{sec:moduli}
% Let $X$ be
% a compact Riemann surface of genus $g\geq 2$ with canonical bundle $K$ and underlying topological oriented surface $\Sigma$. 
 We define the following moduli spaces (see \cite{Simpson:1994a, Simpson:1994b}):
 \begin{enumerate}
 	\item The \emph{Dolbeault moduli space} ${\mathcal M}_{Dol}(X,\ms G)$  of  $S$-equivalence classes of 
semistable $\ms G$-Higgs bundles on $X$, 
\item The \emph{de Rham moduli space} ${\mathcal M}_{dR}(\Sigma, \ms G)$  of gauge equivalence classes of reductive flat  $\ms G$-connections, 
 \item The \emph{Betti moduli space} ${\mathcal M}_B(\Sigma,\ms G)$ of conjugacy classes of completely reducible representations  $\pi_1(\Sigma)\to \ms G$.
 \end{enumerate}  
There are homeomorphisms, which are diffeomorphisms in the neighborhood of the Fuchsian point
\begin{equation} \label{eqn:hk}
{\mathcal M}_{Dol}(X, \ms G) \stackrel{HK}{\xrightarrow{\hspace*{.75cm}}}{\mathcal M}_{dR}(\Sigma, \ms G)\stackrel{RH}{\xrightarrow{\hspace*{.75cm}}} {\mathcal M}_B(\Sigma,\ms G)
\end{equation}
where $HK$ (resp.\  \ $RH$) is the  Hitchin--Kobayashi (resp.\   Riemann--Hilbert) correspondence. 

We also introduce two distinguished submanifolds of the moduli space:
 \begin{enumerate}
\setcounter{enumi}{3}
\item The \emph{Hitchin component} ${\mathcal H}(\Sigma, {\ms G}_{\mathbb R})$ of conjugacy classes of Hitchin representations into a split real form of $\ms G$;
\item The \emph{oper moduli space} $\ms{Op}(X,{\ms G})$ of gauge equivalence classes of $\ms G$-opers. 
\end{enumerate}
 
 This section focuses on the common (smooth) \emph{Fuchsian point} in the moduli spaces that we have encountered before: the Fuchsian $\ms G$-Higgs bundle, the Fuchsian or Veronese oper etc.  
The Fuchsian point is a point of transverse intersection of ${\mathcal H}(\Sigma, {\ms G}_{\mathbb R})$ and $\ms{Op}(X,{\ms G})$.
 The outcome will be to describe this tangent space in its various guises using the Hitchin base and the Hitchin section.
 
More precisely we have several goals in this section.
\begin{itemize}
	\item In Propositions \ref{pro:dolbeault} and \ref{pro:deRham} and \ref{lem:iso}, we describe, using Hitchin deformations, the tangent space at the Fuchsian point for the three moduli spaces ${\mathcal M}_{Dol}(X, \ms G)$, ${\mathcal M}_{dR}(\Sigma, \ms G)$, and ${\mathcal M}_B(\Sigma,\ms G)$. The description is in terms of isomorphisms with $Q(X, {\mk g})\oplus\overline {Q(X, {\mk g})}$,  where $Q(X, {\mk g})$ is the Hitchin base \eqref{eqn:q}. The corresponding tangent vectors are called {\em Dolbeault}, {\em Hodge} and {\em Betti deformations}, respectively.
	\item Then in Corollary \ref{cor:variation}, we show that all these descriptions coincide. In other words, the isomorphisms with the space of holomorphic differentials  commute with the Riemann--Hilbert and Hitchin--Kobayashi correspondences. This  follows from the vanishing of the first variation of the harmonic metric for Dolbeault deformations.
	\item Finally, we use this (now unambiguous) description to achieve our main goal: describing the tangent space of the Fuchsian point in the Hitchin component as the Hitchin base in Proposition \ref{pro:Hitchin}. In  particular, we relate this tangent space to the handy oper deformations.
\end{itemize}

We remark here that this discussion extends the various isomorphic  descriptions of the cotangent space of Teichmüller space as  holomorphic quadratic differentials \cite{Wolpert:1983td}, as well as the global parametrization via harmonic maps due to Wolf \cite{Wolf:1989uk}. We also point out the thesis of Dalakov \cite{Dalakov:2008}, which also studies the germ of the moduli space at the Fuchsian point.

\subsection{Dolbeault deformations for Higgs bundles} \label{sec:tangent}
Let us first consider the Dolbeault moduli space
  ${\mathcal M}_{Dol}(X,\ms G)$.
 The first order deformations of the Fuchsian point can be described as follows.
 
The tangent space  $\T_{(\mc G, \Phi)}{\mathcal M}_{Dol}(X, \ms G)$  is 
given by the first cohomology of a  deformation complex $C_{Dol}(\mc G, \Phi)$.
In the presence of the solution  $(\nabla,\rho)$ of Hitchin self duality equation, we may take harmonic representatives 
for $H^1(C_{Dol}(\mc G, \Phi))$ -- and denoting the corresponding vector space by $\mc H^1(C_{Dol}(\mc G, \Phi))$. Then  (cf.\  \cite{Simpson:1992vk} and \cite[Sec.\ 7]{Nitsure:1991}), 
\begin{align*}
{\mathcal H}^1(C_{Dol}(\mc G, \Phi))
=\biggl\{
(\varphi, \beta)\in \Omega^{1,0}(X, \mc G)&\oplus \Omega^{0,1}(X,\mc G)
: \\
& \overline\partial_\nabla\varphi+ [\Phi, \beta]=0
\ ,\ \partial_\nabla\beta-[\rho(\Phi), \varphi]=0
\biggr\}\ .
\end{align*}
%\marginf{We should say something about $\beta=0$ corresponding to fixing the holomorphic structure}
Here, we let $\nabla=\nabla^{1,0}+\nabla^{0,1}$ be the decomposition of the connection into type, and we have also introduced the notation:
$
\partial_\nabla\defeq \nabla^{1,0}$, 
$ \overline\partial_\nabla\defeq \nabla^{0,1}$.

In the expression of the deformation complex above,  $\beta$ is responsible for the infinitesimal change in the holomorphic structure of $\mathcal G$, whereas $\varphi$ is the change in the Higgs field $\Phi$. In general, the condition of holomorphicity of the Higgs field relates these two variations, but at the Fuchsian point they decouple, and we have 
 the following simple description (see \cite[Example 2.14]{Wentworth:2014vc} for the case $\ms{G}=\ms{SL}(n,{\mathbb C})$).  Let us denote for $b\in\overline {Q(X, {\mk g})}$, 

 \begin{proposition} 
\label{pro:dolbeault}{\sc[Dolbeault infinitesimal parametrization]} At the Fuchsian point, let
$$
T_{Dol}:{Q(X, {\mk g})}\oplus \overline {Q(X, {\mk g})}\longrightarrow H^1(C_{Dol}(\mc G,\Phi))\ :\ 
(q,b)\mapsto(\phi(q), \beta(b))
$$
where, $\rho$ being the Cartan involution,
 \begin{equation}
\beta(b) \defeq \rho(\phi(\overline{b})) \ . \label{eqn:dot-phi}	
 \end{equation}
 Then $T_{Dol}$
defines a complex linear isomorphism.
 \end{proposition}

\begin{proof}
Notice that for $(\varphi,\beta)=(\phi(q),\beta(b))$ as above,   $\varphi$ is holomorphic and $\beta$ is harmonic with respect to $\overline \partial_\nabla$. Moreover,
$ [\Phi, \beta(b)]=[Y,\beta(b)]=0$, and
$[\rho(\Phi), \varphi(q)]=-[X,\varphi(q)]=0$.  Hence, $(\varphi(q),\beta(b))\in {\mathcal H}^1(C_{Dol}(\mc G, \Phi))$.  This gives an inclusion  ${Q(X, {\mk g})}\oplus \overline{{Q(X, {\mk g})}}
\hookrightarrow {\mathcal H}^1(C_{Dol}(\mc G, \Phi))$, 
and now the result follows for dimensional reasons.
\end{proof}
 
 \begin{remark}
We could actually use Serre duality, the hyperbolic metric on $X$ and the Dolbeault isomorphism to identify $\overline{{Q(X, {\mk g})}}$ with $\bigoplus_{i=1}^l {\mathcal H}^{0,1}(X, K^{-m_i})$, where the script indicates harmonic forms.
 \end{remark}
 
\subsection{Hodge parametrization in the de Rham picture} 
Recall that
${\mathcal M}_{dR}(\Sigma,\ms G)$ is the moduli space of reductive flat $\ms G$-connections.
By Corlette's theorem
\cite{Corlette:1988, Donaldson:1987, Labourie:1991}, 
 for any reductive flat connection $\D$  and  conformal structure $X$ on $\Sigma$, there exists a unique harmonic metric $\rho$.  This completes the Hitchin-Kobayashi correspondence. Fixing a metric $\rho$,
we can take harmonic representatives for the  
the first cohomology of
deformation complex at a flat connection $\D$, and write:

\begin{equation} \label{eqn:de Rham-tangent}
\T_{D}{\mathcal M}_{dR}(\Sigma,\ms G)\simeq {\mathcal H}^1(C_{dR}(\D))=
\left\{ B\in \Omega^1(X, \mc G) : \D B= \D^\ast B=0\right\}\ ,
\end{equation}
where $\D^\ast$ is the formal adjoint of $\D$ for the metric
 $$A,B\mapsto\int_\Sigma(A,\rho(B\circ J))_\mk g\ ,$$
 so that $\D^\ast=\rho(\D\circ J)$.
Here, $J$ is the Hodge star operator on $1$-forms defined by the conformal structure on $\Sigma$.
Then at the Fuchsian point we have the following identification
%Then we can identify the tangent space of the character variety $\Rep(\grf,\ms G_\mathbb C)$ at $\rho_0$ with  $H^1(K\otimes\mathcal G)$. This is the {\em  complex de Rham tangent space}.  One gets a linear  map
%\begin{equation}
%\dt{\phi}:Q\longrightarrow H^0(K\otimes\mathcal G),
%\end{equation}
%where for $q\in Q$, 
%\begin{equation}\label{eqn:phi-dot}
%\dt{\phi}(q)=\Phi_q-\Phi_0= q_1\otimes e_1+ + q_l\otimes e_l
%\end{equation}
%
%In the context of the Fuchsian locus that is representation in the principal $\mathsf{PSL}(2,\mathbb R)$, the $G$ bundle comes equipped with an antilinear involution $\rho$ whose fixed $\mathcal G_\mathbb R$ is a $\ms G_\mathbb R$- bundle, where $\ms G_\mathbb R$ is the real split form of the group.  Then we consider the 1-form
\begin{proposition} \label{pro:deRham} {\sc[Hodge infinitesimal parametrization]}
At the Fuchsian point $[\D]$, the map
\begin{equation} 
 \label{eqn:iso-de Rham}
T_{dR}: {Q(X, {\mk g})}\oplus \overline{{Q(X, {\mk g})}}\longrightarrow {\mathcal H}^1(C_{dR}(\D))\ : \ (q,b) \mapsto {\phi}(q) -\rho({\phi}(q))+\beta(b)+\rho(\beta(b))
\end{equation} 
defines a  real linear isomorphism.
\end{proposition}

\begin{proof}
Let us decompose: $D=D'+D''$, where $D'=\partial_\nabla-\rho(\Phi)$, $D''=\overline\partial_\nabla+\Phi$.
Then $D^\ast=i\ast(D''-D')$.
Hence, harmonicity of 
$$B= {\phi}(q) -\rho({\phi}(q))+\beta(b)+\rho(\beta(b))$$
 is equivalent to: $D''B=D'B=0$.
Breaking into type, the first of these equations is
\begin{align}
D''B&=\overline\partial_\nabla B+[\Phi, B]\notag \\
&= \overline\partial_\nabla(\phi(q)+\rho(\beta(b)))+[\Phi, -\rho({\phi}(q))+\beta(b)]\notag\\
&=\overline\partial_\nabla \phi(q)+\rho(\partial_\nabla\beta(b))-\rho[\rho(\Phi),\phi(q))]+[\Phi, \beta(b)] \notag\\
&=\overline\partial_\nabla \phi(q)+\rho(\partial_\nabla\beta(b))-\rho[X,\phi(q))]+\rho[Y, \beta(b)] \label{eqn:vanishing}\\
&=0 \ ,\notag
\end{align}
since the first two terms terms on the right of eq.\ \eqref{eqn:vanishing} vanish because the $q_k$ (resp.\  \ $b_k$) are holomorphic (resp.\  \ harmonic), and the last two terms vanish because the $e_k$ (resp.\  \ $f_k$) are highest (resp.\  \ lowest) weight vectors, and $\rho(Y)=X$.  The second equation
follows similarly. Note that we have used the fact that $\rho$ is parallel with respect to $\nabla$. The fact that the map is an isomorphism follows from dimensional reasons.
 \end{proof}
 
\subsection{Oper parametrization in the Betti picture} \label{subsec:oper}
Let $\bf V$ be the local system determined by the holomorphic connection $D$ on ${\mathcal G}_{\operatorname{op}}$. By Weil \cite{Weil:1964ts}, the Zariski tangent space  $\T_{D}{\mathcal M}_{B}(\Sigma,\ms G)
$ is given by
$H^1({\bf V})$.

%
%Our main proposition here is the following

%
%This Proposition will follow easily after the next Lemma which uses the Betti picture.
%
Now there is an exact sequence of sheaves of $\underline {\mathbb C}$-modules, where $\underline {\mathbb C}$ is the locally constant sheaf.
$$
0\longrightarrow {\bf V}\longrightarrow {\mathcal
G}_{\operatorname{op}}\stackrel{D}{\longrightarrow} {\mathcal G}_{\operatorname{op}}\otimes
K\longrightarrow 0\ .
$$
By Serre duality and the fact that $\lambda\circ D=D\circ \lambda$, we see that $H^1({\mathcal G})\to H^1({\mathcal G}\otimes K)$ is surjective if and only if $H^0({\mathcal G})\to H^0({\mathcal G}\otimes K)$ is injective. The latter holds since
$D$ is irreducible, and hence $H^0({\bf V})=\{0\}$. This gives an exact sequence
in cohomology:
\begin{equation}\label{eqn:eichler} 
0\longrightarrow  H^0({\mathcal G}_{\operatorname{op}}) \longrightarrow H^0({\mathcal G}_{\operatorname{op}}\otimes K) \longrightarrow
 H^1({\bf V}) \longrightarrow H^1({\mathcal G}_{\operatorname{op}}) \longrightarrow H^1({\mathcal G}_{\operatorname{op}}\otimes K)\longrightarrow 0\ .
 \end{equation}
% Consider,
% \begin{equation} \label{eqn:rh}
% H^1({\bf V}) \simeq \T_{[D]}{\mathcal M}_{B}(\Sigma,\ms G)\stackrel{(RH)_\ast}{\xleftarrow{\hspace*{.75cm}}}{\mathcal M}_{dR}(X,\ms G) \ .
% \end{equation}
%The map on the right is the derivative of the Riemann--Hilbert correspondence. Under this map, an Oper deformation in 
% ${\mathcal M}_{dR}(X,\ms G)$ corresponds to the image of the coboundary map in \eqref{eqn:eichler}.
 The full tangent space to  $\mc M_{B}(\Sigma,\ms G)$ at the Veronese oper will be  described by the next lemma.
 \begin{lemma} \label{lem:iso} We have the following\begin{enumerate}
 	\item The inclusion:
 $Q(X,{\mk g})\hookrightarrow H^0({\mathcal G}_{\operatorname{op}}\otimes K)\, :\, q\mapsto \phi(q)$,
  induces an
 isomorphism with the cokernel of the map $H^0({\mathcal G}_{\operatorname{op}}) \rightarrow H^0({\mathcal G}_{\operatorname{op}}\otimes K)$. 
 \item The inclusion: 
 $\overline{Q(X,{\mk g})}\hookrightarrow H^1({\mathcal G}_{\operatorname{op}})\, :\, b\mapsto \lambda(\phi(\bar b))$,
induces an isomorphism 
 with the kernel of $H^1({\mathcal G}_{\operatorname{op}}) \rightarrow H^1({\mathcal G}_{\operatorname{op}}\otimes K)$.
 	\end{enumerate}
  \end{lemma}
 
 \begin{proof}
 As in the proof of Proposition \ref{pro:deRham}, $Q(X, {\mk g})\subset\ker D^\ast$, and hence $Q(X, {\mk g})$ is orthogonal to the image of $D$. Now by the Riemann--Roch formula, 
 $$
 \dim_{\mathbb C}H^0({\mathcal G}_{\operatorname{op}}\otimes K)-\dim_{\mathbb C}H^0({\mathcal G}_{\operatorname{op}})=\dim_{\mathbb C}Q(X, {\mk g})\ .
 $$
 This implies the first statement, and the second statement is proven similarly.
 \end{proof}

 By the lemma, the exact sequence \eqref{eqn:eichler} becomes a short exact sequence
 $$
 0\longrightarrow Q(X,{\mk g}) \longrightarrow H^1({\bf V})\longrightarrow \overline{Q(X,{\mk g})}\longrightarrow 0
 $$
 This in turn splits as follows.

%The involution $\lambda$ gives a splitting of the short exact sequence \eqref{eqn:eichler}, which in turn via Lemma \ref{lem:iso} provides a real linear isomorphism
%\begin{eqnarray} \label{eqn:iso-betti}
% Q(X, {\mk g})\oplus \overline{Q(X, {\mk g})}&\to& H^1({\bf V}),\\
% (q,b)&\mapsto &(\phi(q),\lambda(\phi(\overline b)).
%\end{eqnarray}
%Actually, this parametrization will prove to be  different from the previous ones. The compatible choice is given in the following

\begin{proposition}\label{pro:Betti}{\sc[Betti infinitesimal parametrization]} At the Veronese oper,
the mapping 
$$T_B:Q(X;\mk g)\oplus \overline{Q(X,\mk g)}\longrightarrow H^1({\bf V}) : (q,b)\mapsto
\phi(q)+\lambda(\phi(q))+ \phi(\overline b)-\lambda(\phi(\overline b))\ , 
$$ 
defines a real linear isomorphism.
\end{proposition}

\subsection{Identification of the different infinitesimal parametrizations}
Our main result here is Corollary  \ref{cor:variation}, which states that the three  descriptions of the tangent space to the moduli space at the Fuchsian point given in the previous section are compatible with the Hitchin-Kobayashi and Riemann-Hilbert correspondences \eqref{eqn:hk}. This relies on the following theorem, which may be regarded as a generalization of the classic result of Ahlfors \cite[Lemma 2]{Ahlfors:1961}.
\subsubsection{Variation of the harmonic metric}
In this section we prove
\begin{theorem} \label{thm:metric-variation}
The first variation of the harmonic metric at the Fuchsian point vanishes for the  Dolbeault deformations in Proposition \ref{pro:dolbeault}.
\end{theorem}

\begin{proof}
Fix $(q,b)\in Q(X, {\mk g})\oplus \overline{Q(X, {\mk g})}$.
Let $\dt\rho$ denote the first variation of the Cartan involution for the Dolbeault deformation $(\phi(q), \beta(b))$. Then $\rho\dt\rho$ is a family of derivations, and since $\ms G$ is semisimple there is a smooth section $Z$ of $\mc G$ such that $\rho\dt\rho={\rm ad}_Z$.  
Then for any other one parameter family of  sections $U$ of $\mc G$,
\begin{equation} \label{eqn:rho-dot}
\overset{\bullet}{\wideparen{\rho(U)}}=\rho(\dt U)+[\rho(Z), \rho(U)]\ .
\end{equation}
For convenience, set $\dt k=-\rho(Z)$.
To begin, we claim that the first variation of the connection satisfies:
\begin{equation} \label{eqn:nabla-dot}
(\dt{\nabla})^{1,0}=\partial_\nabla\dt k+\rho(\beta(b))\ .
\end{equation}
Indeed, since $\nabla$ is the Chern connection for $\rho$, for any fixed $U\in \Gamma(\mc G)$ independent of the variational parameter,
$
\nabla(\rho(U))=\rho(\nabla U)
$.
Hence differentiating, using eq.\ \eqref{eqn:rho-dot}, we have
\begin{equation*}
[\dt\nabla, \rho(U)]+\nabla[\rho(Z), \rho(U)]= [\rho(Z), \rho(\nabla U)]+\rho[\dt \nabla, U]\ .
\end{equation*}
Thus
\begin{align}
[\dt\nabla, \rho(U)]-[\nabla(\dt k), \rho(U)]-[\dt k, \nabla(\rho(U))] &= -[\dt k, \rho(\nabla U)]+[\rho(\dt \nabla), \rho(U)]\ ,\notag \\
[\dt\nabla -\nabla(\dt k), \rho(U)] &=[\rho(\dt \nabla), \rho(U)]\ , \notag \\
\dt\nabla&=\nabla(\dt k)+\rho(\dt\nabla) 
\label{eqn:dnabla}
\end{align}
since $U$ was arbitrary.
 Now by definition, $(\dt\nabla)^{0,1}=\beta(b)$, so \eqref{eqn:nabla-dot} follows by taking the $(1,0)$-part of \eqref{eqn:dnabla}.   Notice that from \eqref{eqn:nabla-dot},
\begin{align*}
\dt F_\nabla&=d_\nabla(\dt\nabla)=\overline\partial_\nabla(\dt\nabla)^{1,0}+\partial_\nabla(\dt\nabla)^{0,1} \\
&=\overline\partial_\nabla(\partial_\nabla\dt k+\rho(\beta(b)))+\partial_\nabla\beta(b)\\
&=\overline\partial_\nabla\partial_\nabla\dt k\ .
\end{align*}
Differentiating  \eqref{eqn:hitchin}, we then get

\begin{align*}
\overline\partial_\nabla\partial_\nabla\dt k-[\phi(q), \rho(\Phi)]-[\Phi, \rho(\phi(q))]+[\Phi, [\dt k, \rho(\Phi)]]&=0 \\
\overline\partial_\nabla\partial_\nabla\dt k-[\phi(q), X]+\rho[\phi(q), X]+[\Phi, [\dt k, \rho(\Phi)]]&=0\ . \\
\end{align*}
However, as in the proof of Proposition \ref{pro:deRham}, since $\phi(q)$ is a combination of highest weight vectors, 
$$
[\phi(q), X]=\rho[\phi(q), X]=0\ .
$$
We therefore obtain 
\begin{equation} \label{eqn:hitchin-dot}
\overline\partial_\nabla\partial_\nabla\dt k+[\Phi, [\dt k, \rho(\Phi)]]=0 \ .
\end{equation}
Let $\langle U,V\rangle_\mk g=-(U,\rho(V))_\mk g$.
% Then using the K\"ahler identities 
%\begin{equation} \label{eqn:kahler}
%\partial_\nabla^\ast=i\ast \overline\partial_\nabla\ ,\ \overline\partial_\nabla^\ast=-i\ast \partial_\nabla\ .
%\end{equation}
Using \eqref{eqn:hitchin-dot} and  integration by parts yields successively
\begin{align}
\int_X\bigl(\overline\partial_\nabla\partial_\nabla\dt k, \rho(\dt k)\bigr)_\mk g +
\int_X\bigl( [\Phi, [\dt k, \rho(\Phi)]],\rho(\dt k)\bigr)_\mk g&=0,\notag \\
\int_X\bigl(\partial_\nabla\dt k, \rho(\partial_\nabla\dt k)\bigr)_\mk g +
\int_X\bigl([\dt k, \rho(\Phi)],[\Phi, \rho(\dt k)]\bigr)_\mk g&=0, \notag \\
i\int_X\langle\partial_\nabla\dt k, \partial_\nabla\dt k\rangle_\mk g +
i\int_X\langle [\Phi, \rho(\dt k)],[\Phi, \rho(\dt k)]\rangle_\mk g&=0\ .  \label{eqn:last}
\end{align}
Both terms on the left hand side of \eqref{eqn:last} are nonnegative; hence, both vanish. Vanishing of the second term implies that $\rho(\dt k)$ is a linear combination of lowest weight vectors, so $\dt k$ is a linear combination of highest weight vectors. 
Since $m_j\geq 1$, the metric has constant positive curvature, $i\ast F_\nabla>0$, on the highest weight components
 ${\mathfrak g}_{m_j}\otimes K^{m_j}$ of $\mathcal G$.
By 
 the Bochner formula 
$$
\partial_\nabla^\ast\partial_\nabla=\overline\partial_\nabla^\ast\overline\partial_\nabla+i\ast F_\nabla\ ,
$$
 this
implies that $\ker\partial_\nabla=\{0\}$  on these components. Hence, the vanishing of the first term on the left hand side of \eqref{eqn:last} implies $\dt k\equiv 0$.  
\end{proof}

\subsubsection{All parametrizations coincide}
\begin{corollary} \label{cor:variation}

The Dolbeault, Hodge, and oper parametrizations coincide. 
 More precisely,
the following diagram commutes:

$$
\xymatrix{
& Q\oplus \overline{Q} \ar[dl]_{T_{Dol}} \ar[d]^{T_{dR}}\ar[dr]^{T_B} & \\
{\mathcal H}^1(C_{Dol}({\mathcal G}, \Phi))\ar[r]^{\ (HK)_\ast} 
&{\mathcal H}^1(C_{dR}(D))\ar[r]^{\ (RH)_\ast}
& H^1({\bf V})\ ,
}
$$
where the vertical isomorphisms are those described in Section \ref{sec:moduli}, and $(HK)_\ast$, $(RH)_\ast$ are the derivatives of the Hitchin--Kobayashi and Riemann--Hilbert maps.
\end{corollary}

\begin{proof}  The commutativity of the identification of de Rham and oper deformations follows from the following simple remark. By construction, we have $\sigma(\phi(q))=-\phi(q)$ where $\sigma$ is the involution defined in Section \ref{sec:invs}; thus $\lambda(\phi(q))=-\rho(\phi(q))$.  The content of the lemma is therefore in the Hitchin-Kobayashi correspondence.
For simplicity, abbreviate the notation $\phi=\phi(q)$, $\beta=\beta(b)$, etc.  We need to show that:
\begin{equation} \label{eqn:HK}
(HK)_\ast(\phi,\beta)=\beta-\lambda(\beta)+\phi+\lambda(\phi)\ .
\end{equation}
Because the Fuchsian point is a smooth point of the moduli space, deformations are unobstructed. We may therefore find a family 
of Higgs bundles $(\nabla_\varepsilon^{0,1}, \Phi_\varepsilon)$, passing through the Fuchsian point at $\varepsilon=0$, and satisfying
$$
(\dt\nabla)^{0,1}=\beta\ , \ \dt \Phi=\phi\ .
$$
Moreover, the Fuchsian bundle is stable, which is an open condition, so we may assume the Higgs bundles 
$(\nabla_\varepsilon^{0,1}, \Phi_\varepsilon)$ are stable for $\varepsilon$ sufficiently small.
Let $\rho_\varepsilon$ be a family of harmonic metrics, whose existence is guaranteed by Theorem \ref{thm:HK}. Then:
$$HK(\nabla_\varepsilon^{0,1}, \Phi_\varepsilon)=\nabla_\epsilon+\Phi_\varepsilon-\rho_\varepsilon(\Phi_\varepsilon)\ .$$
By Theorem \ref{thm:metric-variation}, the first variation of the harmonic metrics $\dt\rho$ vanishes. We therefore conclude that
\begin{align*}
(HK)_\ast (\phi, \beta)&=\dt\nabla+\phi-\rho(\phi) \\
&=\beta+\rho(\beta)+\phi-\rho(\phi) \\
&=\beta-\lambda(\beta)+\phi+\lambda(\phi)\ ,
\end{align*}
which verifies \eqref{eqn:HK}. 
%Next, we show
%\begin{equation} \label{eqn:RH}
%(RH)_\ast(\psi,\alpha)=\epsilon+\chi=\delta(\phi-\lambda(\beta))+\lambda\circ\delta(\phi+\lambda(\beta))\ .
%\end{equation}
%Let $B\in {\mathcal H}^1(C_{dR}(\D))$ be given by $B=\beta+\lambda(\phi)+\phi-\lambda(\beta)$,
%so that $(RH)_\ast(\psi,\alpha)$ is the image of $[B]$ in $H^1({\bf V})$. Then 
%$$
%B^{0,1}=\beta+\lambda(\phi)\ , \ B^{1,0}=\phi-\lambda(\beta)\ .
%$$
%It follows that
%\begin{equation} \label{eqn:lambdaB}
%\lambda\circ\delta\circ\lambda(B^{0,1})=\lambda\circ\delta\circ\lambda(\beta)+\lambda\circ\delta(\phi)
%=\lambda\circ\delta(\phi+\lambda(\beta))\ .
%\end{equation}
%On the other hand, by Lemma \ref{lem:betti10},
%$$
%(RH)_\ast[B]-\lambda\circ\delta\circ\lambda([B^{0,1}])=\delta([B^{1,0}])=\delta(\phi-\lambda(\beta))\ .
%$$
%This last equation with \eqref{eqn:lambdaB} gives \eqref{eqn:RH}.
The result follows.
\end{proof}

\subsection{The tangent spaces to opers and the Hitchin component}

We explain in this section our main technical tool, which we state in the de Rham picture:

\begin{proposition} \label{pro:Hitchin}
At the Fuchsian point $D$,
\begin{enumerate}
	\item
The map $\phi$ (regarded as taking values in $\T_D\mc M_{dR}(\Sigma,\ms G)$)	defines an isomorphism of $Q(X,\mk g)$ with the tangent space $\T_D \operatorname{Op}(X,\ms G)$.
	\item The map $\phi+\lambda(\phi)$ defines an isomorphism of $Q(X,\mk g)$ with the tangent space of the Hitchin component $\T_D \mc H(\Sigma,\ms G_{\mathbb R})$, which coincides with the infinitesimal version of the Hitchin parametrization.
\end{enumerate}
\end{proposition}

\begin{proof} The first point follows from the fact that the tangent space of opers is the set of variations of flat connections fixing the holomorphic structure $\mc G_{\operatorname{op}}$. Now the variation $\bar b=q$ (i.e.\ $\phi(2q)$) in Proposition \ref{pro:deRham}
defines such a variation, and for dimensional reasons $\phi(Q(X,\mk g))$ is then identified with $\T_D \operatorname{Op}(X,\ms G)$.

The second point follows at once from Corollary  \ref{cor:variation}. Indeed, the Hitchin infinitesimal parametrization is interpreted in the Dolbeault parametrization as the map $q\to\phi(q)$, but by Corollary  \ref{cor:variation}, $(HK)_*(\phi(q))=\phi(q)+\lambda(\phi(q))$.
\end{proof}

\section{First Variation of Holonomy}\label{sec:gard}

The main result of this section is a Gardiner type formula for the variation of the eigenvalues of the holonomy under deformations of the Fuchsian point. Although the approach can be generalized to linear representations for all split groups we shall concentrate here on the case of $\mathsf{SL}(n,\mathbb R)$. Then, if $\gamma$ is a closed geodesic of length $\ell_\gamma$, the largest eigenvalue $\lambda_\gamma$
(resp.\  \ $p$-th largest $\lambda_\gamma^{(p)}$) at the Fuchsian point  is
\begin{equation} \label{eqn:largest}
\lambda_\gamma=e^{(n-1)\ell_\gamma/2}\quad\text{(resp.\  \ $\lambda_\gamma^{(p)}=e^{(n+1-2p)\ell_\gamma/2}$)}\ .
\end{equation}

Recall that in this context, we have associated to an element $q=(q_2,\ldots, q_n)\in Q(X,n)$:
\begin{itemize}
	\item The {\em standard oper deformation} $\phi^0(q)=\sum_{k=2}^n q_k \otimes E^0_{k-1}$,
	\item The  {\em normalized oper deformation} $\phi(q)=\sum_{k=2}^n q_k \otimes E_{k-1}$,
\item The {\em standard Hitchin deformation} $\psi^0(q)=\phi^0(q)+\lambda(\phi^0(q))$,
\item The {\em normalized  Hitchin deformation} $\psi(q)=\phi(q)+\lambda(\phi(q))$,
\end{itemize}
%For this section, we specialize to the case $\ms G=\ms{SL}(n,\mathbb C)$, $\ms G_{\mathbb R}=\ms{SL}(n,\mathbb R)$. \marginr{I have removed the "P"s?}
Our main results in this section are the following.

\begin{theorem}{\sc [Gardiner formula]} \label{thm:gardiner}
Along the Fuchsian locus, 
the first variation of the largest eigenvalue $\lambda_\gamma$ of the holonomy along a simple closed geodesic $\gamma$ of length $\ell_\gamma$ along a standard Hitchin deformation given by $q_k\in H^0(X, K^{k})$, 
is
$$
\d\log {\lambda}_\gamma (\psi^0(q_k))=
\frac{(-1)^k(n-1)!}{2^{k-2}(n-k)!} \int_0^{\ell_\gamma} \Re\left(q_k({\gamma},\ldots,{\gamma})\right)\ \d s\ .
$$
More generally,  the first variation of the  $p^{\hbox{\tiny th}}$-largest eigenvalue is
\begin{equation*}
\d\log {\lambda}^{(p)}_\gamma (\psi(q_k))=c_{n,k}^{(p)} \int_0^{\ell_\gamma} \Re\left(q_k({\gamma},\ldots,{\gamma})\right)\ \d s\ .
\end{equation*}
where 
\begin{equation} \label{eqn:cnkp}
c^{(p)}_{n,k}=\frac{(p-1)!(n-p)!}{2^{k-2}(n-k)!}
\sum_{j=\max(1, k+p-n)}^{\min(k, p)} { n-k\choose p-j }{k-1\choose j-1}^2(-1)^{j+k+1} \ .
\end{equation}
In particular,  for $k=n$,
\begin{equation} \label{eqn:k=n}
\d\log {\lambda}^{(p)}_\gamma (\psi(q_n))= (-1)^{p+n+1}\frac{(n-1)!}{2^{n-2}}{{n-1}\choose{p-1}} \int_0^{\ell_\gamma} \Re\left(q_{n}({\gamma},\ldots,{\gamma})\right)\ \d s\ .
\end{equation}
\end{theorem}

\begin{remark} {\em
For the special case of deforming in the direction of quadratic differentials, notice that $c_{n,2}^{(p)}=n+1-2p$, and by \eqref{eqn:rk}, $r_2=1/2\pi$. It follows from Theorem \ref{thm:gardiner}, Proposition \ref{pro:katok}, and eq.\ \eqref{eqn:largest}, that if $q\in H^0(X, K^2)$,
$$
d\ell_\gamma(\psi(q))=\frac{1}{\pi}\Re\langle q, \Theta_\gamma^{(2)}\rangle=\frac{2}{\pi}\Re\int_X \mu(z)\cdotp\Theta_\gamma^{(2)}(z)\, \d x\d y\ ,
$$
where $\mu(z)=\overline{q(z)}(h(z))^{-1}$ is the harmonic Beltrami differential associated to $q$.  This is the formula in \cite[Theorem 2]{Gardiner:1975ts}. Hence, Theorem \ref{thm:gardiner} is indeed a higher rank generalization of Gardiner's formula. }
\end{remark}
Summing over in the index $p$ in \eqref{eqn:k=n}, and using \eqref{eqn:largest}, we have the following
\begin{corollary}{\sc [Variation of the trace]} \label{thm:trace-variation}
For a standard Hitchin deformation given by  an element $q_{n}\in H^0(X, K^{n})$,
the first variation of the trace of the holonomy along a closed geodesic $\gamma$ of length $\ell_\gamma$ at the Fuchsian locus
is
\begin{equation*}
\d {\tr(\delta(\gamma))}(\psi^0(q_n))=2(-1)^n(n-1)!\left(\sinh(\ell_\gamma/2)\right)^{n-1}\int_0^{\ell_\gamma} \Re\left(q_{n}({\gamma},\ldots,{\gamma})\right)\ \d s\ .
\end{equation*}
\end{corollary}
As another corollary, one gets

\begin{corollary}{\sc [Variation of consecutive eigenvalues]} \label{thm:consec-eig}
Using the notation of Theorem \ref{thm:gardiner}, we have 
\begin{equation*}
\d\log \frac{{\lambda}^{(p)}_\gamma}{{\lambda}^{(p+1)}_\gamma} (\psi(q_k))=A(p,k,n)\int_0^{\ell_\gamma} \Re\left(q_{k}({\gamma},\ldots,{\gamma})\right)\ \d s\ ,
\end{equation*}
where $A(p,n,k)$ only depen\d s on the integers $(p,n,k)$ and is nonzero.
\end{corollary}

Finally, we also have 

\begin{theorem}{\sc [Oper deformations]} \label{thm:oper-gardiner}
The first variation of the $p^{\hbox{\tiny th}}$-largest eigenvalue of the holonomy along a closed geodesic $\gamma$ of length $\ell_\gamma$ along a standard oper deformation given by  an element $q_k\in H^0(X, K^{k})$, 
is
$$
\d\log\lambda_\gamma^{(p)}\left(\phi^0(q_k)\right)=
\frac{c_{n,k}^{(p)}}{2}\int_0^{\ell_\gamma} q_k({\gamma},\ldots,{\gamma})\ \d s\ ,
$$
where $c_{n,k}^{(p)}$ is defined in  \eqref{eqn:cnkp}.
\end{theorem}

%\begin{remark} \label{rem:hejhal}
%Corresponding to Corollary \ref{thm:trace-variation}, the formula for the variation of the trace for Oper deformations is obtained by simply adding the imaginary part.  That is,
%\begin{equation*}
%\d {\tr(\rho(\gamma))}(\psi^0(q_n))=2(-1)^n(n-1)!\left(\sinh(\ell_\gamma/2)\right)^{n-1}\int_0^{\ell_\gamma} q_{n}(\dt{\gamma},\ldots,\dt{\gamma})\ \d t\ .
%\end{equation*}
%This agrees with the result in \cite{Hejhal78} for $2\leq n\leq 6$, up to a sign. The change in  sign comes about when one passes from the $n$-th order equation to a system of first order equations.
%\end{remark}

In order to prove these results, in Section \ref{sec:gen}
we first recall a general formula computing the variation of eigenvalues by the variation of parameters method. In  Section \ref{sec:fbg}, we give an explicit description of the Fuchsian bundle along a geodesic and in Section \ref{sec:proof-gard} prove our results.

\subsection{A general formula}\label{sec:gen}
 Consider a connection $\nabla$ and a closed curved $\gamma$ in $\Sigma$ so that the holonomy of $\nabla$ along $\gamma$ has an eigenvalue $\lambda_\gamma$ of multiplicity 1. We denote by $L_\gamma$ the corresponding eigenline along $\gamma$, $H_\gamma$ the supplementary hyperplane stable by the holonomy, and $\pi$ the projection on $L_\gamma$ along $H_\gamma$. In the sequel,
let us use the the general notation,
$$
\dt{f}=\left.\frac{\d}{\d t}\right\vert_{t=0}f(t)\ .
$$ Then we have

\begin{lemma}
let $\nabla^t$ be a smooth one parameter family of connections. Then there exists (for $t$ small enough) a  unique smooth function $\lambda_\gamma(t)$ such that 
\begin{itemize}
\item $\lambda_\gamma(0)=\lambda_\gamma$,
\item $\lambda_\gamma(t)$ is an eigenvalue of the holonomy of $\nabla_t$ of multiplicity 1.
\end{itemize}
Moreover,
\begin{equation}\label{lambda}
\dt{\lambda}_\gamma=-\lambda_\gamma\cdotp \int_0^{\ell_\gamma}\tr\bigl(\dt{\nabla}(s)\cdotp \pi\bigr)\ \d s\ .
\end{equation}
\end{lemma}
\begin{proof} The first part of the lemma is classical.  Observe now that the left hand side of eq.\ \eqref{lambda} is invariant under gauge transform of the family $\nabla^t$. For the right hand side, let $\nabla_0^t=A_t^*\nabla^t$, where $A_0=\operatorname{Id}$.
Then $\dt{\nabla_0}=\dt{\nabla}+\nabla \dt{A}$. In particular,
\begin{eqnarray*}
\int_0^{\ell_\gamma}\tr\bigl(\dt{\nabla_0}(s)\cdotp \pi\bigr)\ \d s&=&\int_0^{\ell_\gamma}\tr\bigl(\dt{\nabla}(s)\cdotp \pi\bigr)\ \d s+\int_0^{\ell_\gamma}\tr\bigl(\nabla(\dt{A})(s)\cdotp \pi\bigr)\ \d s\ .
\end{eqnarray*}
But, since $\pi$ is $\nabla$-parallel,
\begin{eqnarray*}
\int_0^{\ell_\gamma}\tr\bigl(\nabla(\dt{A})(s)\cdotp \pi\bigr)\ \d s&=&\int_0^{\ell_\gamma}\frac{\d}{\d s}\tr\bigl(\dt{A}(s)\cdotp \pi\bigr)\ \d s\\
&=&\Bigl[\tr\bigl(\dt{A}(s)\cdotp \pi\bigr)\Bigr]_0^{\ell_\gamma}\\
&=&0\ .
\end{eqnarray*}
Thus, the right hand side of  \eqref{lambda} is also invariant under gauge transform. 

We may now use  a gauge transform so that the eigenline $L_{\gamma}^t$ and the transverse parallel hyperplane $H_{\gamma}^t$ are constant. Let $e$ be a nowhere vanishing section of $L_\gamma$ (after possibly taking an irrelevant double cover of $\gamma$), let $f$ be a section of $H_\gamma^\perp \subset \mc E^*$. 
Then for all $t$ 
$$
\log \lambda_{\gamma}^t=-\int_{0}^{\ell_\gamma}\frac{\braket{f(s)\mid\nabla^t e(s)}}{\braket{f(s)\mid e(s)}}\d s\ ,
$$
where the bracket denotes the duality. Then a standard derivation yields
$$
\dt{\lambda_{\gamma}}=-\lambda_{\gamma}\cdotp\int_{0}^{\ell_\gamma}\frac{\braket{f(s)\mid\dt{\nabla} e(s)}}{\braket{f(s)\mid e(s)}}\d s=-\lambda_{\gamma}\cdotp\int_{0}^{\ell_\gamma}\tr\bigl( \dt{\nabla}(s)\cdotp \pi\bigr)\ \d s\ .
$$
We have completed the proof of formula \eqref{lambda}.
\end{proof}

\subsection{The Fuchsian bundle along a geodesic}\label{sec:fbg}
Consider the case of the principal $\ms{SL}(2)$ in
 $\ms{SL}(n,\mathbb C)$. Let $S\to X$ be a choice of spin structure on $X$, so 
that $S^{\otimes 2}=K$. Then we can describe the associated bundle to the
 defining representation of $\mathsf{SL}(n,\mathbb C)$  by \eqref{eqn:E}.
%$$
%\mathcal E=\operatorname{Sym}^{n-1}(S\oplus S^*)=\bigoplus_{p=1}^n S^{2p-n-1}\ .
%$$
Let $\gamma$ be a geodesic in $X$, we then have the following structure on $\mathcal E$ along $\gamma$.
\begin{enumerate}
\item {\em a harmonic metric:} If $\Phi_0$ is defined as in Section \ref{sec:fu-hi-bu}, then the hyperbolic metric on $X$ induces a (split) hermitian structure on $\mathcal E$ which solves Hitchin's equations \eqref{eqn:hitchin}.
\item {\em a trivialization:}  We have a canonical trivialization (up to sign) of $S$ given by a section $\sigma$ along $\gamma$ so that $\sigma^2(\gamma')=1$. We also have an identification of $S^*$ with $S$ using the metric and denote by $\bar \sigma$ the  dual section of $S^*$ corresponding to $\sigma$. Both give a trivialization of $\mathcal E$ along $\gamma$ by  the frame 
$$\{\hat w_p\defeq \sigma^{p-1}\bar \sigma^{n-p}\}_{p=1,\ldots,n}\ , \hat w_p\in S^{2p-n-1}\ .$$
\item {\em a real structure:} We have a real structure on 
$\mathcal E$ characterized by $\hat w_p\mapsto \overline{\hat w_p}\defeq\hat w_{n+1-p}$.
 \end{enumerate}

Observe that if 
$$
A(w)=B(w)+\overline{ B({\overline w})},\ \ \ \ \ \overline{C({\overline w})}=C(w)\ ,
$$
then 
\begin{equation}
\tr(A\cdotp C)=2\Re\left(\tr(B\cdotp C)\right).	\label{realtrace}
\end{equation}
%Finally let $x=\frac{1}{2}(s+\overline{s})$ and $y=\frac{1}{2}(i\overline{s}-is)$, so that $x+iy=s$ and $x-iy=\overline{s}$ then $x$ and $y$ are real sections of $S$.

%\subsection{The connection along geodesics}\label{fbgc}

With respect to the Chern connection $\nabla$ on $\mathcal E$,  the frame $\{w_p\}$ described above is parallel. The flat connection $\D = \nabla + \Phi_0-\rho(\Phi_0)$ along the geodesic $\gamma$  may now be expressed
$$
\gamma^\ast \D=\gamma^\ast\nabla+\left(Y-X\right)\otimes\frac{ \d s}{\sqrt 2}\ ,
$$
where $Y$ and $X$ are endomorphisms satisfying
\begin{align*}
(-X)(\hat w_p)&=c_{p} \cdotp\hat w_{p+1}\ ,\\
Y(\hat w_{p})&=c_{p-1}\cdotp\hat w_{p-1}\ ,
\end{align*}
where $c_p=(p(n-p)/2)^{1/2}$ (see \cite[p.\ 978]{Kostant:1959wi}). The factor of $1/\sqrt 2$ is due to the fact that the hyperbolic metric is twice the real part of the hermitian metric on $K^{-1}$. 
If we change to a new $\nabla$-parallel
frame defined by 
\begin{equation} \label{eqn:wp}
 w_p=\left\{(p-1)!(n-p)!\right\}^{1/2}\cdotp \hat w_p
 \end{equation}
 then
\begin{align}
(-X)( w_p)&=\tfrac{1}{\sqrt 2}(n-p)\cdotp   w_{p+1}\ ,\label{eqn:X}\\
Y( w_{p})&=\tfrac{1}{\sqrt 2}(p-1)\cdotp  w_{p-1}\ .\label{eqn:Y}
\end{align}
%The action of $(Y-X)\otimes dt/\sqrt{2}$ on $\mathcal E$
%may therefore be identified
% with the action of a first order differential operator on  homogeneous polynomials of degree $n-1$ in 
%a complex variable $z$ and $\bar z$ (see \eqref{eqn:diffops} below). This motivates the definition
Observe that in this new frame the linear action of $-\sqrt{2}X$ and $\sqrt{2}Y$ on each fiber corresponds to the action of $z\partial_{\bar z}$ and $\bar z\partial_z$ on the monomials $z^{p-1}\bar z^{n-p}$, thus generating the $n$-dimensional (irreducible) action of the principal $\ms{SL}(2,\mathbb C)$ and being coherent with paragraph \ref{sec:Lie}. Passing to real monomials $x^{p-1}y^{n-p}$, $z=x+iy$, and the action of the real operator $Y-X$ motivates the following definition.

Let
\begin{equation}
	u_p\defeq 
	\frac{i^{n-p}}{2^{n-1}}\sum_{r=0}^{p-1}\sum_{s=0}^{n-p}(-1)^s{p-1\choose r}{n-p\choose s}w_{r+s+1}
	\label{eqn:up}\ .
\end{equation}
%$s$ and $\bar s$. It is given by
%$$
%(Y-X)\otimes \frac{dt}{\sqrt{2}}=\frac{1}{2}\left(\bar s\frac{\partial}{\partial s}+
%s\frac{\partial}{\partial \bar s}\right)\otimes dt\ .
%$$
%Introducing the sections $x=\frac{1}{2}(s+\overline{s})$ and $y=\frac{i}{2}(\overline{s}-s)$, we have 
%$$
%(Y-X)\otimes \frac{dt}{\sqrt{2}}= \frac{1}{2}\left(x\frac{\partial}{\partial x}-
%y\frac{\partial}{\partial  y}\right)\otimes dt\ .
%$$
%In particular, let us now define
%\begin{equation} 
%	u_p\defeq x^{p-1}y^{n-p}\label{eqn:up}\ .
%\end{equation}
We then observe that the system $\{u_p\}_{p=1,\ldots,n}$ is a $\nabla$-parallel  frame
for the real sections of $\mathcal E$, and by a direct calculation using eqs.\ \eqref{eqn:X} and \eqref{eqn:Y}, we see that each $u_p$ is an eigensection 
 of 
$(Y-X)/\sqrt{2}$, with eigenvalue $(2p-n-1)/2$. In particular,
\begin{equation}
	D_{\gamma'}u_p= \frac{2p-n-1}{2}\cdotp u_p\ . \label{eqn:Dup}
\end{equation}
Furthermore, 
 if $\gamma$ is a closed geodesic of length $\ell_\gamma$, the  $p$-th largest  eigenvalue for the holonomy of $D$ is given by \eqref{eqn:largest}.

 The lines $L_p$ generated by $u_p$ form a parallel frame along any geodesic for both $\nabla$ and $D$. Let $\pi_p$ be the projection on the eigenline $L_p$ orthogonal to $\bigoplus_{j\not=p}L_j$. 
For $k=1,\ldots, n-1$, let $E^0_k=(-\sqrt{2}X)^k$, $F^0_k=\rho(E^0_k)=-(\sqrt{2}Y)^k$.
The proof of the following proposition is in the Appendix  (see Proposition \ref{A:pretrace}).

\begin{proposition}\label{pretrace} For all $k,p$, we have $\tr(F^0_k\pi_p)=-\tr(E^0_k\pi_p)$. Moreover,
\begin{align}
\tr(E^0_k\cdotp \pi_1)&=\frac{(-1)^k(n-1)!}{2^k(n-k-1)!}\label{tr-high} \\
\tr(E^0_k\cdotp
\pi_p)&=\frac{(p-1)!(n-p)!}{2^k(n-k-1)!}\sum_{j=\max(0,k+p-n)}^{\min(k,p-1)} 
{{n-k-1}\choose{p-j-1}}{{k}\choose{j}}^2(-1)^{j+k}\ . \label{tr-all}
\end{align}
\end{proposition}
It is interesting to remark that $\tr(E^0_k\cdotp \pi_1)\not=0$ for all $k$ and $n$. On the other hand, for $p=n-1$, $k\leq n-2$ the second equation yields
\begin{eqnarray*}
\tr(E^0_k\cdotp \pi_p)&=&\frac{(n-2)!}{2^k(n-k-1)!}\sum_{j=k-1}^{k} 
{{n-k-1}\choose{n-j-2}}{{k}\choose{j}}^2(-1)^{j+k}\cr
&=&\frac{(n-2)!}{2^k(n-k-1)!}\left(
{{n-k-1}\choose{n-k-2}}
-
{{k}\choose{k-1}}^2\right)\cr
&=&\frac{(n-2)!}{2^k(n-k-1)!}\left(
n-k^2-k-1\right)\ ,
\end{eqnarray*}
which vanishes for $n=k^2+k+1$ (e.g.\ $n=3$, $k=1$).

\subsection{Proof of Theorems \ref{thm:gardiner} and \ref{thm:oper-gardiner}}\label{sec:proof-gard} We now have the following consequence of the calculations in the previous section.
For the standard Hitchin variation $\psi^0(q_k)$, the corresponding variation of flat connection is 
$$
\dt{\D}=q_k\otimes E^0_{k-1}-\rho\left( q_k\otimes E^0_{k-1}\right)\ .
$$
In particular, for any endomorphism $B$ commuting with the real involution, as a function on $\UX$,
\begin{equation}
\int_0^{\ell_\gamma} \tr\left(\dt{\D}\cdotp B\right) \, \d s= \int_0^{\ell_\gamma} 2\Re\left(q_k\right)\tr(E^0_{k-1}\cdotp B)\, \d s\ .\label{DEk}
\end{equation}
 The theorem now follows from Proposition \ref{pretrace} and  eqs.\ \eqref{DEk} and  \eqref{lambda}.

We note that the same argument applies to oper deformations at the Veronese oper, where now $\dt{\D}=q_k\otimes E^0_{k-1}$ .  In particular,  we immediately obtain \ref{thm:oper-gardiner}.

\section{The Symplectic Structure and Twist Deformations} \label{sec:twi-symp}
The purpose of this section is to investigate some relations between the Hitchin parametrization along the Fuchsian locus and the ``symplectic nature of the fundamental group of a surface,'' to paraphrase Goldman \cite{Goldman:1984}.
\begin{enumerate}
	\item 
	 In Corollary \ref{pro:normalization}, we
	   show that
	   via  the (normalized) Hitchin parametrization 
	     the symplectic structure of the $L^2$-pairing coincides with (a multiple of) the Atiyah--Bott--Goldman symplectic structure. 
	\item 
	In Section \ref{sec:twi}, we
	relate the twist or bending deformations to relative Poincaré series,  in Section \ref{sec:twi} we prove the Reciprocity Theorem \ref{thm:recip}.
	\item  In Theorem \ref{thm:hamil} we study the Hamiltonian flow of the length functions in term of twists and the relative Poincaré series.
\end{enumerate}

\subsection{Symplectic structure}\label{sec:symplectic-form}

Recall \cite{Atiyah:1983, Goldman:1984}  that the de Rham moduli space 
${\mathcal M}_{dR}(\Sigma, \ms{G})$
 carries  the {\em Atiyah--Bott--Goldman} symplectic form   
 $\omega^0_{\ms G}$ given by:
\begin{equation} \label{eqn:symplectic-form}
\omega^0_{\ms G}(A,B)=\int_\Sigma (A\ww B)_{\mk g}
\ .
\end{equation}
It will also be convenient to consider the {\em normalized symplectic form}.
\begin{equation} \label{eqn:symplectic-form-normalized}
\omega_{\ms G}(A,B)=\frac{1}{d(\mk{g})}\omega^0_{\ms G}(A,B)\ ,
\end{equation}
where $d(\mk g)$ is the Dynkin index of a principal $\mk{sl}(2)\hookrightarrow \mk{sl}(n)$ subalgebra (see \eqref{eqn:dynkin}).
Then  restricting to the Fuchsian locus one has
$
\omega_{\ms G}=\omega_2
$,
where $\omega_2$ is the standard Weil--Petersson symplectic form on Teichm\"uller space. In general, when $\ms{G}=\ms{SL}(n,{\mathbb C})$, we denote the Atiyah--Bott--Goldman symplectic form by $\omega^0_n$ and the normalized one by $\omega_n$.

As will become clear from the computations below, 
the Hitchin space ${\mathcal H}(\Sigma, {\ms G}_{\mathbb R})$ a symplectic submanifold of ${\mathcal M}_{dR}(\Sigma, \ms{G})$, and the oper moduli space $\ms{Op}(X,{\ms G})$ is Lagrangian.
To illustrate, let us evaluate  $\omega^0_{\ms G}$ on real de Rham deformations.  
\begin{lemma} \label{lem:symplectic}
Let $q_1, q_2$ be differentials of weight $m_i+1$ and $m_j+1$, respectively.  Then
$$
\omega_{\ms G}(\psi(q_1),\psi(q_2))=\begin{cases} -2\Im\langle q_1, q_2\rangle & \text{if $i=j$;} \\
0& \text{if $i\not=j$.}
\end{cases}
$$
\end{lemma}

\begin{proof} Observe first that if $a=\alpha\otimes A$ and $b=\beta\otimes B$ then $(a\wedge b)_\mk g=\alpha\wedge\beta\cdotp (A, B)_\mk g$. Thus if $v\in K^{m_i+1}$ and $w\in K^{m_j+1}$
\begin{eqnarray}
((v\otimes E^0_i)\wedge (\rho(w\otimes e_j)))_\mk g&=&i\delta_{ij}d(\mk g)\braket{v,w}\d\sigma\ ,\cr
((v\otimes E^0_i)\wedge (w\otimes e_j))_\mk g&=&0\ .	
\end{eqnarray}
It then follows that
\begin{align*}
\omega^0_{\ms G}(\psi(q_1), \psi(q_2))&=\int_X \left(\left((q_1\otimes E^0_i-\rho(q_1\otimes e_{i})\right)\wedge \left(q_2\otimes e_j-\rho(q_2\otimes e_{j})\right)\right)_{\mk g}\\
&=-\int_X ((q_1\otimes E^0_i)\wedge \rho(q_2\otimes e_j))_\mk g -\int_X (\rho(q_1\otimes E^0_i)\wedge (q_2\otimes e_j))_\mk g \\
&=i\cdot\delta_{ij}\cdotp d({\mk g})\left(-\int_X \braket{q_1,q_2}\d\sigma+\int_X \braket{q_2,q_1}\d\sigma\right)\\
&=-\delta_{ij}\cdotp d({\mk g})\cdotp 2\Im\braket{q_1, q_2}_X\ .
\end{align*}
\end{proof}

As an immediate consequence of this lemma we have the following corollary, which explains the normalization that we have chosen.

\begin{corollary} \label{pro:normalization}
At the origin in the Hitchin base, 
the pull-back  of the normalized Atiyah--Bott--Goldman symplectic form $\omega_{\ms G}$ on ${\mc H}(\Sigma, {\ms G}_{\mathbb R})$ by the normalized Hitchin section coincides with the symplectic form on $Q(X, \mk{g})$ associated to the $L^2$-metric.
\end{corollary}

%Finally, \marginf{ I wonder why this pargraph is here : may I have moved it inadvertently ?}
%the tangent space $H^1_{\mathbb R}({\bf V})$ to ${\mathcal M}_{B}(\Sigma,\ms G_{\mathbb R})$ at $D$ is given by the fixed point set of $\lambda$. Given $q\in Q$, let 
%$$\chi(q)=\delta(\phi(q))+\lambda(\delta(\phi(q)))\ . $$
%(cf.\ \eqref{eqn:betti}).
% This gives an embedding 
%$
%\chi : Q\hookrightarrow H^1({\bf V})
%$,
%which is an isomorphism onto $H^1_{\mathbb R}({\bf V})$. Note also that $\psi$ and $\chi$ commute with respect to the natural map 
%${\mathcal H}^1_{\mathbb R}(C_{dR}(D))\to H^1_{\mathbb R}({\bf V})$ coming from the Riemann-Hilbert correspondence, and by Corollary \ref{cor:variation} these identifications commute with $\phi$ as well.

\subsection{Twist deformations}\label{sec:twi}
Let $\D$ be a flat connection on a principal $G$-bundle $P$ over $\Sigma$. Let $\gamma$ be a closed simple curve, whose holonomy is $g$, and let finally  $h$ be an element in the Lie algebra $\mk{z}(g)$ of the centralizer of $g$. We can then construct an element
$\tau_{\gamma}(h)$ in $H^1_\D(\Sigma,\operatorname{ad}(P))$ by the following construction. 
Choose a tubular neighborhood $U=S^1\times]0,1[$ of $\gamma$, let $f$ be a function on $U$ which is constant equal to $1$ on a neighborhood of $S^1\times\{1\}$ and $0$ on a neighborhood of $S^1\times\{0\}$. Observe finally that the parallel sections $\tilde h$ of $\operatorname{ad}(P)$ restricted to $U$ are identified canonically to elements $h\in\mk{z}(g)$. Then we define
\begin{equation} \label{eqn:twist-def}
\tau_\gamma(h)\defeq \d f \cdotp \tilde h=\D(f\cdotp \tilde h)\in\Omega^1(U,\operatorname{ad}(P))\ ,
\end{equation}
and observe that $\tau_\gamma(h)$ exten\d s (by zero) to all of $\Sigma$ (see Goldman \cite{Goldman:86}).

For Hitchin representations, the monodromy is regular \cite{Labourie:2006} and so using the previous notation, $\mk{z}(g)$ is identified with a Cartan subalgebra $\mk h$.
Thus the previous construction associates to every element $h\in\mk h$ a vector field $\tau_\gamma(h)$, called the {\em twist vector field}, on the Hitchin component. 

At the Fuchsian locus we furthermore have an identification $\mk{z}(g)=\mk{z}(X-Y)_{\mathbb R}$. This leads to  

\begin{proposition}
For any  $B\in \T_{D} {\mathcal M}_{dR}(\Sigma, \ms{G})$ we have
\begin{equation*}
	\omega^0_{\ms G}(B,\tau_\gamma(h))=\int_\gamma (B, h)_{\mk g}\ .
\end{equation*}	
In particular, if $q_k$ is a $k$-holomorphic differential, $k=m_i+1$ for some $i$, then
\begin{equation} \label{eqn:symplectic-twist}
\omega_{\ms G}(\psi(q_k), \tau_\gamma(h))=\frac{-2 r_k\cdotp (h,  e_{i})_{\mk g}}{d({\mk g})}\cdotp \Im\langle q_k, i\cdotp\Theta_\gamma^{(k)}\rangle\ ,
\end{equation} 
where $r_k$ is given in \eqref{eqn:rk}.
\end{proposition}

\begin{proof}  The first assertion follows from eqs.\ \eqref{eqn:symplectic-form} and \eqref{eqn:twist-def} from integration by parts (recall that $\D B=0$). For the second, since $\lambda(h)=h$, and using Proposition \ref{pro:katok},
\begin{align*}
	\omega^0_{\ms G}(\psi(q_k), \tau_\gamma(h))&= \int_\gamma (h, (q_k\otimes e_{k-1}-\rho(q_k\otimes e_{k-1}))_{\mk g} \\
	&=(h,e_i)_{\mk g}\left( \int_{\gamma} q_k+\int_\gamma \overline{q_k} \right) \\
	&=(h,e_i)_{\mk g}r_k\cdotp \left( \langle q_k, \Theta_\gamma^{(k)}\rangle + \overline {\langle q_k, \Theta_\gamma^{(k)}\rangle}\right)\\
	&=-2 (h,e_i)_{\mk g}r_k\cdotp\Im\langle q_k, i\cdotp\Theta_\gamma^{(k)}\rangle\ .
\end{align*}
The statement now follows from \eqref{eqn:symplectic-form-normalized}.
\end{proof}
%
%The final result will be that 
%\begin{align}
%	\omega(\Phi_q, X_\gamma(h))&= \tr(h\cdotp E^0_k)\cdotp\Re\left(\int_\gamma \breve q\,\d t\right)\cr
%	&=\Re\left(\braket{q|\tr(h\cdotp E^0_k)\theta_\gamma^{(k)}}_X\right)\cr
%	&=\Im\left(\braket{q|i\cdotp\tr(h\cdotp E^0_k)\theta_\gamma^{(k)}}_X\right)\cr
%	&=\omega(\P hi_q,i\cdotp\tr(h\cdotp E^0_k)\theta_\gamma^{(k)}).
%\end{align}
%\begin{equation}
%	X_\gamma(h)=\sum_k c(k,h) i\cdotp \theta^{(k)}_\gamma.
%\end{equation}
%where $c(k,h)=\tr(E^0_k h)$ are real coefficients.
Recall Lemma \ref{lem:dual-basis} and the notion of a principal basis.
\begin{corollary} \label{cor:bending}
	The space of twist deformations at a simple geodesic $\gamma$ is the $\mathbb R$-span of 
	 the  relative $k$-Poincar\'e series of $\gamma$, for degrees $k=m_1+1,\ldots, m_l+1$.  More precisely, if  $(h_i)_{i\in 1,\ldots l}$ is the principal basis of $\mk{z}(X-Y)_{\mathbb R}$, then
	$$
\tau_\gamma(h_i)=\frac{r_{m_i+1}}{d({\mk g})}\cdotp\psi\left(i\cdotp \Theta_\gamma^{(m_i+1)}\right)\ .
$$
\end{corollary}

\begin{proof} Write
$$
\tau_\gamma(h)=\sum_{i=1}^{l} \psi(Q_i)\ ,
$$
where $Q_i\in H^0(X, K^{m_i+1})$ (as a real vector space), and let $q_k$ be an arbitrary $k$-differential, $k=m_i+1$ for some $i$. Then using Lemma \ref{lem:symplectic} and eq.\ \eqref{eqn:symplectic-twist}, we find
$$
\omega_{\ms G}(\psi(q_k), \tau_\gamma(h))= -2\Im\langle q_k, Q_i\rangle 
= \frac{-2 r_k\cdotp(h, e_i)_{\mk g}}{d({\mk g})}\cdotp \Im\langle q_k, i\cdotp\Theta_\gamma^{(k)}\rangle \ ,
$$
and,
$$
\omega_{\ms G}(\psi(iq_k), \tau_\gamma(h))= -2\Re\langle q_k, Q_i\rangle 
=\frac{-2 r_k\cdotp(h, e_i)_{\mk g}}{d({\mk g})}\cdotp\Re\langle q_k, i\cdotp\Theta_\gamma^{(k)}\rangle \ ,
$$
from which 
$$
\langle q_k, Q_i\rangle= \frac{r_k\cdotp (h, e_i)_{\mk g}}{d({\mk g})} \Braket{q_k, i\cdotp\Theta_\gamma^{(k)}}\ .
$$
Since $k$ and $q_k$ were arbitrary,  the result follows immediately.
\end{proof}

\subsection{Reciprocity: twists and lengths}
By Corollary \ref{cor:bending}, it follows that each relative Poincar\'e series $i\cdotp\Theta_\gamma^{(k)}$ corresponds to a unique twist deformation
$\psi(i\cdotp\Theta_\gamma^{(k)})
$.
 Recall that $\lambda^{(p)}_\gamma$ denotes the $p$-th largest eigenvalue for the holonomy about $\gamma$.

\begin{theorem} {\sc[Reciprocity of the twist deformation]}\label{thm:recip}
For any $k$, $p$,  and any simple closed geodesics $\alpha$, $\beta$, then at the Fuchsian locus,
$$
\d\log\lambda^{(p)}_\alpha\left(\psi(i\cdotp\Theta_\beta^{(k)})\right)= -\d\log\lambda^{(p)}_{\beta}\left(\psi(i\cdotp\Theta_\alpha^{(k)})\right)\ .
$$
\end{theorem}

\begin{proof}
From Theorem \ref{thm:gardiner} and eq.\ \eqref{eqn:katok}, 
\begin{align*}
\d\log\lambda^{(p)}_\alpha\left(\psi^0(i\cdotp\Theta_\beta^{(k)})\right)&=c^{(p)}_{n,k}\Re\int_\alpha i\cdotp\Theta_\beta^{(k)}
=-c^{(p)}_{n,k} r_k\cdotp \Im \langle  i\cdotp\Theta_\beta^{(k)}, i\cdotp\Theta_\alpha^{(k)}\rangle \\
&=+c^{(p)}_{n,k} r_k\cdotp \Im \langle  i\cdotp\Theta_\alpha^{(k)}, i\cdotp\Theta_\beta^{(k)}\rangle \\
&=-\d\log\lambda^{(p)}_\beta\left(\psi^0(i\cdotp\Theta_\alpha^{(k)})\right)\ .
\end{align*}
\end{proof}

\subsection{Hamiltonian functions and twists}\label{sec:hamil}
Recall that the Hamiltonian vector field $H_f$ of a $C^1$ function $f$ on a symplectic manifold $(M,\omega)$  is defined by 
$$
\d f(V)=\omega(V, H_f)\ ,
$$
for any tangent vector field $V$ on $M$.

\begin{theorem}\label{thm:hamil}
Fix a simple closed curve $\gamma$.
 Then along the Fuchsian locus
 the Hamiltonian vector field of the function
$$\log{\lambda^{(p)}_\gamma}: \mathcal H(\Sigma, n)\longrightarrow\mathbb R$$ 
 with respect to the normalized Atiyah-Bott-Goldman symplectic form $\omega_n$  is given by 
\begin{equation} \label{eqn:Hamiltonian}
H_{\log \lambda^{(p)}_\gamma}=\sum_{k=2}^{n}\frac{ c^{(p)}_{n,k} r_k\cdotp\eta_{k-1}}{2}\,\psi(i\cdotp\Theta_\gamma^{(k)})\ ,
\end{equation}
where $c_{n,k}^{(p)}$ is defined in  \eqref{eqn:cnkp}, $r_k$ in \eqref{eqn:rk}, and $\eta_{k-1}$ in \eqref{eqn:etak}.
 
\end{theorem}
\begin{proof}
Let $q_k$ be a $k$-differential.
From Theorem \ref{thm:gardiner} and Proposition \ref{pro:katok}, 
$$
\d\log\lambda^{(p)}_\gamma(\psi(q_k))=c^{(p)}_{n,k}\cdotp\eta_{k-1}\cdotp\Re\int_\gamma q_k=-c_{n,k}^{(p)} r_k\cdotp\eta_{k-1}\cdotp \Im \langle q_k, i\cdotp\Theta_\gamma^{(k)}\rangle\ .
$$
On the other hand, from Lemma \ref{lem:symplectic}, 
$$
\omega_n(\psi(q_k),  \psi(i\cdotp\Theta_\gamma^{(k)}) )=
-2\Im  \langle q_k, i\cdotp\Theta_\gamma^{(k)}\rangle\ .
$$
  Since $k$ and $q_k$ are arbitrary, the result follows.  
\end{proof}

\begin{remark} \label{rem:nonzero-coefficients}
We point out that for the highest eigenvalue $\lambda_\gamma$ (i.e.\ $p=1$), the coefficients in the expression \eqref{eqn:Hamiltonian} are nonzero for all $2\leq k\leq n$ (see Theorem \ref{thm:gardiner}).
\end{remark}

\section{The Variance and the Pressure Metrics}\label{sec:var}
\subsection{The pressure metric}
In the following paragraphs we shall recall some of the results of
 \cite{Bridgeman:2013tn}, where the authors introduced the pressure metric $\bf P$ on the   Hitchin component. We will prove  the following
\begin{theorem}{\sc[Pressure metric for standard deformations]}\label{thm:main-press}
Let $\delta$ be a Fuchsian representation into $\ms{SL}(n,\mathbb R)$ associated to a Riemann surface $X$ with conformal hyperbolic metric. Let $q$ be a holomorphic differential of degree $k$ on $X$, and let $\psi^0(q)$ be the associated standard deformation, then

\begin{equation*}
	{\bf P}_\delta\left(\psi^0(q),\psi^0(q)\right)=\frac{1}{2^{k-1}\pi \vert \chi(X)\vert}\left[\frac{(k-1)!(n-1)!}{(n-k)!}\right]^2\braket{q,q}_X\ .
\end{equation*} 
	Moreover two deformations associated to holomorphic differentials of different degrees are orthogonal with respect to the pressure metric.
\end{theorem}

The pressure metric for normalized deformations $\psi(q)=\eta_{k-1}\cdotp\psi^0(q)$, when $q$ has degree $k$ now follows:
\begin{corollary}{\sc[Pressure metric for normalized deformations]}\label{coro:main-press}
Let $\delta$ be a Fuchsian representation into $\ms{SL}(n,\mathbb R)$ associated to a Riemann surface $X$ with conformal hyperbolic metric. Let $q$ be a holomorphic differential of degree $k$ on $X$, and let $\psi(q)$ be the associated normalized deformation, then

\begin{equation*}
	{\bf P}_\delta\left(\psi(q),\psi(q)\right)=\frac{[(n-1)!]^2}{2^{k-1}\pi \vert \chi(X)\vert}
	{n+1 \choose 3}
	\frac{(2k-1)!}{(n+k-1)!(n-k)!}
	\braket{q,q}_X\ .
\end{equation*} 
	Moreover two deformations associated to holomorphic differentials of different degrees are orthogonal with respect to the pressure metric.
\end{corollary}

Corollary \ref{cor:pres-end}, stated in the introduction, follows at once from Corollary \ref{coro:main-press} and Corollary \ref{pro:normalization}.

The proof of the theorem and the structure of this section fall into several parts. First, we recall the definition of the pressure metric  for projective Anosov representation in Section \ref{sec:conv}. More importantly, we shall introduce the notion of {\em variation of paramatrization} for a deformation of representation in Section \ref{sec:reparam} and state eq.\ \ref{eqn:press-param} which relates the pressure metric and the variation of reparametrization. In Section \ref{sec:gard-param}, we identify the variation of reparametrization associated to the standard variation $\psi^0(q)$ for $q$ a holomorphic differential,  and we deduce eq.\ \eqref{eqn:press-param2}, which identifies the pressure metric as a multiple of the variance metric. The rest of the section is devoted to the proof of Theorem \ref{varl2},  which computes the variance metric for holomorphic differentials in terms  of the $L^2$-metric. In a concluding Section \ref{sec:conc-press}, we sketch how one might extend these results to other pressure metrics.

The pressure metric has been discussed in the quasi-Fuchsian context in \cite{McMullen:2008eh, Bridgeman:2008ju, Bridgeman:2010kt}, and the construction of the pressure metric in \cite{McMullen:2008eh} ties in  Wolpert's approach of the identification of the Thurston metric with the Weil--Petersson metric \cite{Wolpert:1986ww}.

\subsection{Projective Anosov representations}\label{sec:conv}
Introduced in \cite{Labourie:2006} and studied further  in \cite{Guichard:2011tl} and \cite{Bridgeman:2013tn}, a {\em projective Anosov representation} $\delta$ of a hyperbolic group $\Gamma$ in $\mathsf{SL}(n,\mathbb R)$ is characterized by the following features:
\begin{itemize}
	\item A {\em spectrum} that is a map which associates to every nontrivial element $\gamma$ of $\Gamma$, the number $\ell_\gamma(\delta)\defeq \log\vert \lambda_\gamma(\delta)\vert$, where  $\lambda_\gamma(\delta)$ is  the   highest eigenvalue (in modulus)  of $\delta(\gamma)$,
	\item An {\em entropy} defined as 
	$$ h(\delta)\defeq\lim_{T\rightarrow\infty}\frac{1}{T}\log\left(\sharp\,  L_\delta(T)\right),$$
	where $L_\delta(T)\defeq\{[\gamma] \mid \ell_\gamma(\delta)\leq T\}$,  where $[\gamma]$ denotes a conjugacy class in $\Gamma$.
	\item Moreover, one can define the {\em intersection} of two projective Anosov representations $\delta_1$ and $\delta_2$ as
	$$I(\delta_1,\delta_2)\defeq\lim_{T\to\infty}\frac{1}{\sharp\,  L_{\delta_1}(T)}\sum_{\gamma\in L_{\delta_1}(T)}\frac{\ell_\gamma(\delta_2)}{\ell_\gamma(\delta_1)}\ .$$
	and the {\em renormalized intersection} 
	$
	J(\delta_1,\delta_2)\defeq\frac{h(\delta_2)}{h(\delta_1)}I(\delta_1,\delta_2)
	$.
	 \end{itemize} 
As proved by the first author,	 Hitchin representations are examples of projective Anosov representations,  and indeed they were the initial motivation for the definition \cite{Labourie:2006}. The article \cite{Bridgeman:2013tn} introduces entropy and intersections. The {\em pressure metric} $\bf P$, which by definition is the Hessian at a representation $\delta_0$ of the function $\delta\mapsto J(\delta_0,\delta)$,  is a consequence.
	 \subsubsection{Variation of reparametrization} \label{sec:reparam} In \cite{Bridgeman:2013tn}, an  explicit formula for the pressure metric was given as follows. To simplify the exposition, we work in the context of deformations of Fuchsian representations. Let $\{\delta_t\}_{t\in (-\varepsilon,\varepsilon)}$ be a family of deformations of a Fuchsian representation $\delta_0$, associated to a hyperbolic surface $X$. A Hölder function $f$ on $\UX$ is a {\em variation of reparametrization associated to $\delta_0$} if for every $\gamma$ in $\pi_1(\Sigma)$,
	  	 \begin{equation*}
	 	\int_\gamma f\, \d s=\left.\frac{\d}{\d t}\right\vert_{t=0} \ell_\gamma(\delta_t)\ .
	 \end{equation*}
	 Then the {\em pressure metric} is given by the {\em variance of $f$}, that is 
	\begin{equation}
	 {\bf P}_\delta\left(\dt\delta,\dt\delta\right)=\operatorname{Var}(f)\defeq \lim_{r\rightarrow\infty}\frac{1}{r}\int_{\UX}\left(\int_0^r f \left(\phi_s(x)\right)\d s\right)^2\d \mu(x)\ ,\label{eqn:press-param}
\end{equation}
where $\mu$, $\phi_s$ are as in Section \ref{subsec:integration}.

\subsubsection{A consequence of Gardiner Formula}\label{sec:gard-param}
From the Gardiner formula, Theorem \ref{thm:gardiner}, and the definition of variation of reparametrization in Section \ref{sec:reparam}, we immediately obtain that for $q$ be a holomorphic $k$-differential, $\hat q$ the associated complex valued function on $\UX$ and $\breve q=\Re(\hat q)$, the function 
$$
\frac{(-1)^k (n-1)!}{2^{k-2}(n-k)!}\cdotp\breve q\ ,
$$
is the variation of parametrization associated to the standard deformation $\psi^0(q)$. Thus it follows from  \eqref{eqn:press-param}, that 
	\begin{equation}
	 {\bf P}(\psi^0(q),\psi^0(q))=\left[\frac{ (n-1)!}{2^{k-2}(n-k)!}\right]^2\cdotp\operatorname{Var}(q)\ ,
	 \label{eqn:press-param2}
\end{equation}
where by a slight abuse of notation we write $\operatorname{Var}(q)\defeq \operatorname{Var}(\breve q)$. We note again that Curt McMullen has treated the case of quadratic differentials  in \cite{McMullen:2008eh}.

\subsection{Variance and the $L^2$-metric}
 
The goal of this section is to prove \begin{theorem}\label{varl2}
\begin{equation*}
\operatorname{Var}(q)=\frac{2^{k-2}[(k-1)!]^2}{2\pi\vert\chi(X)\vert}  \langle q, q\rangle_X\ .
\end{equation*}
Moreover, if two holomorphic forms have different degrees, then they are orthogonal for the bilinear form underlying the quadratic form $\operatorname{Var}$.
\end{theorem}
%\FFF{
%The fact that the two metrics are proportional is relatively easier to obtain, even resorting to the intuition explained by McMullen in \cite{McMullen:2008eh}. We explain that in the introduction of  paragraph \ref{sec:intuition}, for which  }
This theorem, together with eq.\ \eqref{eqn:press-param2}, concludes the proof of Theorem \ref{thm:main-press}. 

The fact that the variance and $L^2$-metrics are proportional is relatively easier to obtain, even resorting to the intuition explained by McMullen in \cite{McMullen:2008eh}. We give a rough idea in the next paragraph.
\subsection{A rough idea of the proof}
Any point in $\ms U X$ gives an identification of $\widetilde X$ with the Poincaré unit disk $\mathbb D$. Thus given a holomorphic differential $q$, we obtain functions $a_n$ on $\ms U X$, which are the Fourier coefficients of $q$ seen as a holomorphic function on $\mathbb D$. Let then $A_n$  be the $L^2$-norm of $a_n$ on $\ms U X$.

Now we have the key steps which are actually
easy to work out.
\begin{itemize}
\item 	The $L^2$-norm of $q$ on $X$ is easily interpreted as the $L^2$-norm of $F=\Vert\check q\Vert$ on $\ms U X$. Then a  computation in Fourier coefficients yields that for all $s$,
$$
\int_{\ms U X} F\circ \phi_s(x) \d\mu(x)=\sum_{n=0}^\infty A_n f_n(s),
$$
where $f_n(s)$ are some explicit functions of $s$ independent of $q$. This is detailed in Proposition \ref{pro:l2}.
\item Then, using the invariance of the Liouville form by the flow, one gets that the $A_n$ are all proportional to the $L^2$-norm of $q$, just from the fact the right hand side of the above equation does not depend on $s$.
\item A similar analysis using Fourier coefficients yields that $$
\operatorname{Var}(q)=\lim_{s\to\infty}\sum_{n=0}^\infty A_n g_n(s),$$
where again $g_n$ do not depend on $q$ (see Proposition \ref{pro:var-ana}).
\end{itemize}
The proportionality of the two metrics follows immediately. However, obtaining the actual coefficients requires some gruesome effort.

\subsection{Preliminary computations}
\subsubsection{Hypergeometric integrals}
For nonnegative integers $m$ and $d$, and $R\in [0,1]$,  the following 
%(hypergeometric) integrals 
incomplete Beta functions
will play a technical role in the sequel:
\begin{eqnarray*}
I_{m,d}(R)\defeq\int_0^R S^m(1-S^2)^d\cdotp \d S\ .
\end{eqnarray*}
We postpone  the proof of the following result to the Appendix.
\begin{theorem}\label{tech}
\begin{equation*}
\lim_{R\to 1}\left(\frac{1}{\vert\log{(1-R)}\vert}\cdotp\sum_{n=0}^\infty {{n+2p-1}\choose{n}} I_{n,p-1}^2(R)\right)=2^{2p-2}[(p-1)!]^2\ .
\end{equation*}
\end{theorem}

\subsubsection{In the Poincaré disk model}
Let us define the function 
\begin{equation} \label{eqn:rR}
r(R)\defeq \frac{1}{2}\log\left(\frac{1+R}{1-R}\right)\ .
\end{equation}
We recall that if $d_H$ is the hyperbolic distance in the Poincaré disk model, then
\begin{equation*}
d_H(0,Re^{i\theta})=r(R)\ .
\end{equation*}
Let 
 $$
q(z)=\d z^k \sum_{n=0}^\infty a_n\cdotp z^n\ ,$$
 be a holomorphic $k$-differential on the Poincaré disk.
 Let us also consider the real valued function
 $$
\tilde  q(z)\defeq\Re\Bigg(q\left(z\right)\underbrace{\left(\frac{\partial}{\partial r},\ldots,\frac{\partial}{\partial r}\right)}_k\Bigg) \:\hbox{ where }\:
\frac{\partial}{\partial r}=(1-R^2)\frac{\partial}{\partial R}\ .$$
If furthermore $q$ and $q^0$ are two holomorphic forms of degree $k$ and $k_0$, respectively, we will consider the complex valued function
\begin{equation} \label{eqn:odot}
q\odot q^0(z)=q\left(z\right)\underbrace{\left(\frac{\partial}{\partial r},\ldots,\frac{\partial}{\partial r}\right)}_k \overline{q^0\left(z\right)\underbrace{\left(\frac{\partial}{\partial r},\ldots,\frac{\partial}{\partial r}\right)}_{k_0}}\ .
\end{equation}
The main result in this section  computes integrals related to the holomorphic differentials $q$ and $q^0$.
\begin{proposition}
We have
\begin{eqnarray}
\frac{1}{2\pi}\int_{0}^{2\pi}\int_0^R\Vert q(Se^{i\theta})\Vert^2\frac{\d S\d\theta}{1-S^2}&=& 2^k\sum_{n\in\mathbb N}a_n\overline{a}_n\cdotp I_{2n,2k-1}(R)\ , \label{l2}\\
\frac{1}{2\pi}\int_{0}^{2\pi}\left(\int_0^R \tilde  q(Se^{i\theta})\frac{\d S}{1-S^2}\right)^2\d\theta &=&\frac{1}{2}\sum_{n=0}^\infty a_n\overline{a}_n\cdotp I_{n,k-1}^2(R)\ . \label{var}
\end{eqnarray}
Moreover, if $q^0=(\d z)^{k_0}\sum_{n=0}^\infty b_nz^n$,
then
\begin{eqnarray} \label{eqn:bilinear-expansion}
& &\frac{1}{2\pi}\int_{0}^{2\pi}\left(\int_0^R \tilde  q(Se^{i\theta})\frac{\d S}{1-S^2}\right)\left(\int_0^R \tilde  q^0(Se^{i\theta})\frac{\d S}{1-S^2}\right)\d\theta\cr
& &=\frac{1}{2}\sum_{n=\max(0,k_0-k)}^\infty \Re (b_{n+k-k_0}\overline{a}_n)\cdotp I_{n,k-1}(R) I_{n+k-k_0,k_0-1}(R)\ . \label{bil-var}
\end{eqnarray}
Finally,
\begin{equation}\label{bil-l2}
\frac{1}{2\pi}\int_{0}^{2\pi}\left(\int_0^R q\odot q^0(Se^{i\theta})\frac{\d S}{1-S^2}\right)\d\theta =\sum_{n=\max(0,k_0-k)}^\infty \overline{b}_{n+k-k_0}a_n\cdotp I_{2n+k-k_0,k+k_0-1}(R)\ . \end{equation}
\end{proposition}
\begin{proof}
We first prove
\begin{equation} 
\Vert q(Re^{i\theta})\Vert^2=2^k(1-R^2)^{2k}\left(\sum_{n\in\mathbb N}a_n\overline{a_n}\cdotp R^{2n}
+\sum_{m,n\in {\mathbb N}\, ,\, m\neq n} a_n \overline{a}_m\cdotp R^{n+m}\cdotp e^{i(n-m)\theta}\right)\ .\label{l2-1}
\end{equation}
The hyperbolic metric of the Poincar\'e disk model is
$$
\sigma=\frac{1}{(1-R^2)^2}\left(\d x^2+\d y^2\right)\ .
$$
Thus $\Vert \d z\Vert^2=2(1-R^2)^2$ (recall our conventions from Section \ref{sec:l2}). It follows that $\Vert f(z) \d z^k\Vert^2=\vert f(z)\vert^2 2^k(1-R^2)^{2k}$, and hence \eqref{l2-1}. Now \eqref{l2} is an immediate consequence of \eqref{l2-1}. 

Let us now move to the proof of  \eqref{var} and \eqref{eqn:bilinear-expansion}. By \eqref{eqn:rR}, we have that 
\begin{equation*}
\d z
\left(\frac{\partial}{\partial r}\right)\biggr|_{Re^{i\theta}} 
=(1-R^2)\cdotp 
\d z
\left(\frac{\partial}{\partial R}\right)\biggr|_{Re^{i\theta}}
=(1-R^2)e^{i\theta}\ .
\end{equation*}
Thus
\begin{equation*}
\tilde  q(Re^{i\theta})=
\Re\left(\sum_{n=0}^\infty a_n R^n(1-R^2)^k e^{i(n+k)\theta}\right)\ .
\end{equation*}
It follows that
\begin{eqnarray*}
\int_0^R\tilde  q(S e^{i\theta})\frac{\d S}{1-S^2}&=&
\Re\left(\sum_{n=0}^\infty a_n  e^{i(n+k)\theta} \int_0^R S^n(1-S^2)^{k-1}\d S\right)\cr
&=&\frac{1}{2}\sum_{n=0}^\infty I_{n,k-1}(R)\left(a_ne^{i(n+k)\theta}+\overline{a}_ne^{-i(n+k)\theta}\right).
\end{eqnarray*}
 After taking the square and integrating over $\theta$, we obtain that
\begin{eqnarray*}
\frac{1}{2\pi}\int_0^{2\pi}\left(\int_0^R\tilde  q(S e^{i\theta})\frac{\d S}{1-S^2}\right)^2\d\theta
&=&\frac{1}{2}\sum_{n=0}^\infty a_n\overline{a}_n\cdotp I_{n,k-1}^2(R)\ .
\end{eqnarray*}
Moreover, if $q$ if of degree $k$, and $q^0$ of degree $k_0\geq k$, then the same argument yields
\begin{eqnarray*}
& &\frac{1}{2\pi}\int_{0}^{2\pi}\left(\int_0^R \tilde  q(Se^{i\theta})\frac{\d S}{1-S^2}\right)\left(\int_0^R \tilde  q^0(Se^{i\theta})\frac{\d S}{1-S^2}\right)\d\theta\cr
& &=\frac{1}{2}\sum_{n=\max(0,k_0-k)}^\infty \Re (b_{n+k-k_0}\overline{a}_n)\cdotp I_{n,k-1}(R) I_{n+k-k_0,k_0-1}(R)\ . \label{bil-var-1}
\end{eqnarray*}
Let us move to  \eqref{bil-l2}.
We have that
$$
q\odot q^0(Re^{i\theta})=\left(\sum_{n=0}^\infty a_n R^n(1-R^2)^k e^{i(n+k)\theta}\right)\\
\overline{\left(\sum_{m=0}^\infty b_m R^m(1-R^2)^{k_0} e^{i(m+k_0)\theta}\right)}
$$
Thus again, integrating over the circle yields
$$
\frac{1}{2\pi}\int_0^{2\pi} q\odot q^0(Re^{i\theta})\d \theta=\sum_{n=0}^\infty a_n \overline{b_{n+k-k_0}}R^{2n+k-k_0}(1-R^2)^{k+k_0} \ .
$$
Further integration  now gives
\begin{eqnarray*}
\frac{1}{2\pi}\int_0^{2\pi}\int_0^R q\odot q^0(Se^{i\theta})\frac{\d S\d \theta}{1-S^2}&=&\sum_{n=\max(0,k_0-k)}^\infty a_n \overline{b_{n+k-k_0}}\int_0^R S^{2n+k-k_0}(1-S^2)^{k+k_0-1}\d S \cr
&=&\sum_{n=\max(0,k_0-k)}^\infty a_n \overline{b_{n+k-k_0}}I_{2n+k-k_0,k+k_0-1}(R)\ .
\end{eqnarray*}

\end{proof}
\subsubsection{Functions on the unit tangent bundle}
%Let as before $q\in H^0(X, K^k)$. If $q^0$ is a holomorphic differential of degree $k_0$, let us define the following complex valued function on the unit tangent bundle $\pi:\UX\to X$,
%$$
%q\odot q^0(u)=q(\underbrace{u,\ldots,u}_{k})\overline{q^0(\underbrace{u,\ldots,u}_{k_0})}\ .
%$$
Observe that any $x$ in $\UX$ gives an identification of $\widetilde X$ the universal cover of $X$ with $\mathbb D$ the Poincaré disk model: the identification sends $\pi(x)$ to $0$ and $x$ to $(1,0)$. We can then write the {\em analytic expansion} of a $k$-differential $q$  in these coordinates:
\begin{equation}
q_x(z)=\d z^k\sum_{n=1}^\infty a_n(x)z^n\ .\label{qnot}
\end{equation}
In particular, in this way we obtain  complex valued functions $a_n$ on $\UX$. The functions $a_n$ contain all the information about $q$ with some redundancy. 
We can use the family of holomorphic forms $q_x$ to describe the functions that we wish to study  on $\UX$. As we have done previously, we also view $q$ as a function $\hat q$ on $
\UX$, homogeneous of degree $k$. Let us furthermore recall that
the circle acts on both $\UX$ and the Poincaré disk model (hence on holomorphic forms). With this understood, we have the following
\begin{proposition}
\begin{eqnarray}
\Vert\hat q(\phi_r(e^{i\theta}x))\Vert^2
&=&\Vert q_{x}(Re^{i\theta})\Vert^2\label{qx-l2}\\
\breve q (\phi_r(e^{i\theta}x))&=&\tilde  q_x (Re^{i\theta})\label{qx-var}\ ,\\
q\odot q^0 (\phi_r(e^{i\theta}x))&=& q_x\odot q_x^0 (Re^{i\theta})\label{qx-bil}\ .
\end{eqnarray}
\end{proposition}
\begin{proof}
By construction, 
\begin{equation*}
q_{e^{i\theta}x}=(e^{i\theta})^*q_x\ .
\end{equation*}
Therefore,
\begin{eqnarray*}
\Vert q_{e^{i\theta}x}(z)\Vert^2&=&\Vert q_{x}(e^{i\theta} z)\Vert^2\ , \\
\tilde  q_{e^{i\theta}x}(z)&=&\tilde  q_x (e^{i\theta} z)\ ,\notag \\
q_{e^{i\theta}x}\odot q^0_{e^{i\theta}x}(z)&=&q_{x}\odot q^0_{x}(e^{i\theta} z)\ .\notag
\end{eqnarray*}
We now describe the action of the geodesic flow. We have
\begin{eqnarray*}
\Vert \hat q(\phi_r(x))\Vert^2&=&\Vert q_{x}(R)\Vert^2\ ,\\
\breve q(\phi_r(x))&=&\tilde  q_x (R)\ ,\\
q\odot q^0(\phi_r(x))&=&q_{x}\odot q^0_{x}(R)\ .
\end{eqnarray*}
Combining the two actions, one gets the proposition. \end{proof}

\subsubsection{The Hilbert norm and the analytic expansion}
Let
\begin{equation} \label{eqn:A}
A_n\defeq\int_{\UX} a_n(x)\overline{a_n}(x)\d \mu(x)\ .
\end{equation}
We now prove
\begin{proposition}\label{pro:l2}
For any $1>R>0$,
\begin{equation}
\frac{1}{2\pi\vert\chi(X)\vert}\cdotp \Vert q\Vert^2_X=\frac{2^{k+1}}{\log\frac{1+R}{1-R}}\left(\sum_{n=0}^\infty A_n\cdotp I_{2n,2k-1}(R)\right)\ .\label{eqn:fundL2}
\end{equation}
Moreover, if $q^0$ is a holomorphic  
$k_0$-differential on $X$ with associated functions $b_n$ on $\UX$, then for any $1>R>0$,
\begin{equation}
\int_{\UX}q\odot q_0\, \d\mu=\frac{2}{\log\frac{1+R}{1-R}}\left(\sum_{n=\max(0,k_0-k)}^\infty I_{2n+k-k_0,k+k_0-1}(R)\cdotp\int_{\UX}a_n(x)\overline{b_{n+k-k_0}(x)}\cdotp \d\mu(x)\right)\ .\label{eqn:fundbil}
\end{equation}
 
\end{proposition}

As a corollary, we obtain
\begin{corollary}\label{coro:l2}
The following hold:
\begin{equation} \label{eqn:An-expression}
 A_n=\frac{2^{-k}}{2\pi\vert\chi(X)\vert}{{n+2k-1}\choose{n}}\cdotp \Vert q\Vert^2_X\ ;  
 \end{equation}
 and for $k_0\neq k$,
\begin{equation} \label{eqn:bilinear-vanishing}
0=\int_{\UX} a_n(x)\overline{b_{n+k-k_0}(x)}\cdotp \d\mu(x)\ .
\end{equation}
\end{corollary}
We can now proceed to the 

\begin{proof}[Proof of Proposition \ref{pro:l2}] We view $\Vert q\Vert$ as a function on $\UX$. Then, by the normalization of the Liouville measure, 
\begin{equation*}
\Vert q\Vert^2_X=\int_X \Vert q\Vert^2 \d\sigma
=2\pi\vert\chi(X)\vert\int_{\UX} \Vert q(x)\Vert^2 \d\mu (x)\ .
\end{equation*}
Since the Liouville measure is invariant by the geodesic flow, we have for all $r$,
\begin{equation*}
\Vert q\Vert^2_X
=\frac{2\pi\vert\chi(X)\vert}{r}\cdotp\int_{\UX} \int_0^r \Vert q(\phi_s(x))\Vert^2\d s\cdotp \d\mu(x)\ .
\end{equation*}
Let us further use the action of the circle on $\UX$ and the invariance of the Liouville measure to get
\begin{equation}\label{eqn:circle}
\Vert q\Vert^2_X
=\frac{\vert\chi(X)\vert}{ r}\cdotp\int_{\UX} \int_0^{2\pi}\int_0^r \Vert q(\phi_s(e^{i\theta}x))\Vert^2\d s\cdotp \d \theta\cdotp \d\mu(x)\ .
\end{equation}
Using the notation defined in eq.\ \eqref{qnot}, we have  by eq.\ \eqref{qx-l2} that if $s=r(S)$,
\begin{eqnarray} \label{eqn:renorm-circle}
\Vert q(\phi_s(e^{i\theta}x))\Vert^2&=&\Vert q_x(Se^{i\theta})\Vert^2\ .
\end{eqnarray}
By  \eqref{l2} and using that $\d s=\frac{\d S}{1-S^2}$, 
\begin{eqnarray*}
\frac{1}{2\pi}\int_0^{2\pi}\int_0^R\Vert q_x(Se^{i\theta})\Vert^2\frac{\d S \d\theta}{1-S^2}&=&2^k\sum_{n\in\mathbb N}a_n(x)\overline{a}_n(x)\cdotp  I_{2n,2k-1}(R)\ .\end{eqnarray*}
Combining eqs.\ \eqref{eqn:circle} and \eqref{eqn:renorm-circle} with a  further integration over $\UX$, 
eq.\ \eqref{eqn:fundL2}
 follows. Eq.\ \eqref{eqn:fundbil} follows from the same ideas using eqs.\ \eqref{qx-bil} and \eqref{bil-l2}.
\end{proof}

\begin{proof}[Proof of Corollary \ref{coro:l2}]
Near $R=0$, $I_{m,d}(R)\sim R^{m+1}$. Taking the limit when $R\to 0$ of  \eqref{eqn:fundL2} therefore leads to
\begin{equation}
2\pi\vert\chi(X)\vert \cdotp 2^kA_0= \Vert q\Vert^2_X \label{A0-l2}\ ,\\
\end{equation}
Thus, we can rewrite  \eqref{eqn:fundL2} as 
\begin{equation*}
r(R)\cdotp A_0=\sum_{n=0}^\infty A_n\int_0^{R} S^{2n}(1-S^2)^{2k-1}\d S\ .
\end{equation*}
Taking the derivative with respect to $R$ lea\d s to
\begin{equation*}
(1-R^2)^{-1}\cdotp A_0=\sum_{n=0}^\infty A_n\cdotp R^{2n}(1-R^2)^{2k-1}\ .
\end{equation*}
Taking $L=R^2$, we get
\begin{equation*}
\sum_{n=0}^\infty A_n L^n=\frac{A_0}{(1-L)^{2k}}\ .\label{AnA0-1}
\end{equation*}
From the asymptotic expansion
\begin{eqnarray}
\frac{1}{(1-L)^{N}}
&=&\sum_{m=0}^\infty {{m+N-1}\choose{m}}\cdotp L^m,\label{eqn:asym}
\end{eqnarray}
(which follows inductively by differentiation)
we obtain that
\begin{eqnarray*}  
A_n&=&{{n+2k-1}\choose{n}}\cdotp A_0\ .
\end{eqnarray*}
Eq.\ \eqref{eqn:An-expression} now follows from \eqref{A0-l2}.
%
%For the sake of safety we recall the proof of the expansion \eqref{eqn:asym}.
%\begin{eqnarray*}
%\frac{1}{(1-L)^{N}}&=&\frac{1}{(N-1)!}\left(\frac{1}{1-L}\right)^{(N-1)}\cr
%&=&\frac{1}{(N-1)!}\left(\sum_{n=0}^\infty L^n\right)^{(N-1)}\cr
%&=&\frac{1}{(N-1)!}\left(\sum_{n=N+1}^\infty n(n-1)\ldots (n-N+2)\cdotp L^{n-N+1} \right)\cr
%&=&\frac{1}{(N-1)!}\left(\sum_{n=N-1}^\infty n(n-1)\ldots (n-N+2) \cdotp L^{n-N+1}\right)\\
%&=&\frac{1}{(N-1)!}\left(\sum_{m=0}^\infty (m+N-1)(m+N)\ldots (m+1)\cdotp L^{m} \right)\\
%&=&\sum_{m=0}^\infty \frac{(m+N-1)!}{m!(N-1)!} \cdotp L^m\\
%&=&\sum_{m=0}^\infty {{m+N-1}\choose{m}}\cdotp L^m
%\end{eqnarray*}
%
%
We  note that $q\odot q_0(e^{i\theta}x)=e^{i(k-k_0)\theta}q\odot q_0(x)$. Thus, for $k_0\neq k$,
$$
\int_{\UX}q\odot q_0\cdotp \d \mu=0\ .
$$
Then  \eqref{eqn:fundbil} yields
 \begin{equation*}
0=\sum_{n=\max(0,k_0-k)}^\infty I_{2n+k-k_0,k+k_0-1}(R)\cdotp\int_{\UX}a_n(x)\overline{b_{n+k-k_0}(x)}\cdotp \d\mu(x)\ .
\end{equation*}
Taking the derivatives as a function of $R$ we get
\begin{equation*}
0=\sum_{n=\max(0,k_0-k)}^\infty R^{2n+k-k_0}(1-R^2)^{k+k_0-1}\cdotp\int_{\UX}a_n(x)\overline{b_{n+k-k_0}(x)}\cdotp \d\mu(x)\ .
\end{equation*}
Thus,
\begin{equation*}
0=\sum_{n=\max(0,k_0-k)}^\infty R^{2n}\cdotp\int_{\UX}a_n(x)\overline{b_{n+k-k_0}(x)}\cdotp \d\mu(x)\ ,
\end{equation*}
and this last assertion concludes the proof of the corollary.
\end{proof}

\subsubsection{The variance and the analytic expansion}
We perform a similar analysis for the variance to get
\begin{proposition}\label{pro:var-ana}
We have
\begin{equation*}
\operatorname{Var}(q)=\lim_{R\to 1}\frac{1}{\vert\log(1-R)\vert}\cdotp\sum_{n=0}^\infty A_n I_{n,k-1}^2(R)\ .
\end{equation*}
\end{proposition}

\begin{proof} 
Recall \eqref{eqn:rR}. Then by the invariance of the Liouville measure,
\begin{align*}
\operatorname{Var}(q) &= \lim_{r\to \infty} \frac{1}{r}\int_{\UX}\left( \int_0^r \breve q(\phi_s(x))\d s \right)^2 \d\mu(x) \\
&=\lim_{r\to \infty} \frac{1}{r}\int_{\UX}\frac{1}{2\pi}\int_0^{2\pi} \left( \int_0^r \breve q(\phi_s(e^{i\theta}x))\d s \right)^2 \d \theta\d\mu(x) \\
&=\lim_{R\to 1} \frac{2}{\log\frac{1+R}{1-R}}\int_{\UX}\frac{1}{2\pi}\int_0^{2\pi} \left( \int_0^R \tilde q(Se^{i\theta})\frac{\d S}{1-S^2}\right)^2 \d \theta\d\mu(x)\ , \quad\text{by \eqref{qx-var}} \\
&=\lim_{R\to 1} \frac{1}{\log\frac{1+R}{1-R}}\int_{\UX}\sum_{n=0}^\infty a_n\overline a_n I^2_{n,k-1}(R) \d\mu(x)\ ,
\quad\text{by \eqref{var}} \\
&=\lim_{R\to 1} \frac{1}{\vert\log(1-R)\vert}\sum_{n=0}^\infty A_n I^2_{n,k-1}(R) \ ,
\quad\text{by \eqref{eqn:A}} 
\end{align*}
where in the last line we have used the fact  that as $R\to 1$,  $$\log\frac{(1+R)}{(1-R)}\sim \vert\log(1-R)\vert\ .$$
\end{proof}

\subsubsection{The bilinear part of the variance and the asymptotic expansion}
We have
\begin{proposition}\label{pro:bilvar}
If $q$ and $q^0$ are holomorphic forms with different degrees, then
$$
\lim_{r\to\infty}\frac{1}{r}\int_{\UX}\left(\int_0^r\breve q(\phi_s(x))\d s\right)\left(\int_0^r\breve q^0(\phi_s(x))\d s\right)\d \mu(x)=0\ .
$$
\end{proposition}

\begin{proof}
We use the invariance by the action of the circle to get that
\begin{eqnarray*}
B&:=&\int_{\UX}\left(\int_0^r\breve q(\phi_s(x))\d s\right)\left(\int_0^r\breve q^0(\phi_s(x))\d s\right)\d \mu(x)\\
&=&\frac{1}{2\pi}\int_{\UX}\int_0^{2\pi}\left(\int_0^r\breve q(\phi_s(e^{i\theta}x))\d s\right)\left(\int_0^r\breve q^0(\phi_s(e^{i\theta}x))\d s\right)\d \mu(x)\d \theta\ .
\end{eqnarray*}
Using the identification \eqref{qx-var}, and $r=r(R)$, $s=r(S)$, we get that
$$
B=\frac{1}{2\pi}\int_{\UX}\int_0^{2\pi}\left(\int_0^R\tilde q_x(Se^{i\theta}))\frac{\d S}{1-S^2}\right)\left(\int_0^R\tilde q^0_x(Se^{i\theta}))\frac{\d S}{1-S^2}\right)\d \mu(x)\d \theta\ .
$$
We now use  \eqref{eqn:bilinear-expansion} to get that
$$
B=\frac{1}{2}\sum_{n=\max(0,k_0-k)}^\infty\Re\left(\int_{\UX}b_{n+k-k_0}(x)\overline{a_n}(x)\d\mu(x)\right)I_{n,k-1}(R)I_{n+k-k_0,k_0-1}(R)\ .
$$
But from   \eqref{eqn:bilinear-vanishing}, $B=0$. This concludes the proof of the result.
\end{proof}

\subsection{Proof of Theorem \ref{varl2}}

From Proposition \ref{pro:var-ana}, we get that
\begin{equation*}
\operatorname{Var}(q)=\lim_{R\to 1}\frac{1}{\vert\log(1-R)\vert}\cdotp\sum_{n=0}^\infty A_n I_{n,k-1}^2(R)\ .
\end{equation*}
By Corollary \ref{coro:l2},
%$$
% A_n={{n+2p-1}\choose{n}}\cdotp 2^{-p}\Vert q\Vert_X^2.
%$$
%Thus
\begin{equation*}
\operatorname{Var}(q)=\frac{2^{-k}}{2\pi\vert\chi(X)\vert}\Vert q\Vert_X^2\cdotp\left(\lim_{R\to 1}\frac{1}{\vert\log(1-R)\vert}\cdotp\sum_{n=0}^\infty {{n+2k-1}\choose{n}}I_{n,k-1}^2(R)\right)\ .
\end{equation*}
Thus the first part of the theorem follows from  Theorem \ref{tech}.
The second part follows from Proposition \ref{pro:bilvar}.

\subsection{A family of pressure metrics}\label{sec:conc-press}
The pressure metric (which may sometimes be degenerate) was defined in \cite{Bridgeman:2013tn} for projective Anosov representations using the logarithm of the highest eigenvalue as an initial input for ``length''.  The {\em spectrum} of a representation is defined as the indexed collection of lengths. Then,  from this data, the intersection and pressure metric are defined. In the context of Hitchin representations, one has many choices for lengths. 
For instance, one could take the quotient of two consecutive eigenvalues, or any polynomial in these. 
Via the Gardiner formula that we have obtained, we can in principle  compute all the associated pressure metrics at the Fuchsian locus. Jörgen Andersen and Scott Wolpert have suggested to us that for one of these lengths the pressure metric should actually be the $L^2$-metric. A count of parameters  is consistent with this idea. Starting form the work presented here, a somewhat tedious calculation might verify this conjecture.
 
\section{First Variation of Cross Ratios} \label{sec:opers}

We recall that one can associate a cross ratio to a Hitchin representation:   that is, a function on 4 points on the boundary at infinity
(see \cite{Labourie:2005}). More precisely a {\em cross ratio} is a function $b$ on $$
\partial_\infty\grf^{4*}\defeq\{(x,y,z,w)\in\partial_\infty\grf^{4}\mid x\not=y\not=z\not=w\}\ .$$
satisfying some functional rules. Conversely, the cross ratio determines the representation. Moreover, one can characterize those cross ratios which arise from  Hitchin representations \cite[Theorem 1.1]{Labourie:2005}. 
For example, the cross ratio for a Fuchsian representation is the cross ratio that comes from the identification of $\partial_\infty \grf$ with $\mathbb P^1(\mathbb R)$ associated to the hyperbolic structure.

More generally, a cross ratio is associated to projective Anosov representations.  Since the construction only depends on the limit curve of the Anosov representation, 
 the value of the cross ratio at four given points $(x,y,z,w)$ depends analytically on the representation (cf.\ \cite[Theorem 6.1]{Bridgeman:2013tn}).

The purpose of this section is to describe the variation of the cross ratio under oper and Hitchin deformations  associated to a holomorphic differential of degree $k$.
This is achieved in two theorems: Theorem \ref{thm:cr1} and Theorem \ref{thm:rhomb} stated in Section \ref{sec:stat-cr}. The first theorem interprets the variation as a ``generalized period'' which we call a \emph{rhombus function}. This requires further definition. The second theorem gives a description in terms of automorphic forms,  and we can state it right now. This result is a generalization of \cite[Lemma 1.1]{Wolpert:1983td}.  The latter uses techniques of quasiconformal maps which are not available in this context.

\begin{theorem}\label{thm:cr-auto}  Let $\delta_t$ be a family of Hitchin representations coming from the standard deformation associated to  a $\grf$-invariant holomorphic differential  $q$ of degree $k$. Let $b^H_t$ be the corresponding family of cross ratios.  For $(x_1,x_2,X_1,X_2)\in \partial_\infty\grf^{4*}$, 
\begin{eqnarray*}
		& &\left.\frac{\d }{\d t}\right\vert_{t=0}\log b^H_t(x_1,x_2,X_1,X_2)\\
		&=&\left(\frac{(-1)^{k}}{2^{k-2}}\frac{(n-1)!}{(n-k)!}\right)
		\cdotp r_k\cdotp\int_{\mathbb H^2}\Re\Braket{q(z),
	\sum_{i,j\in\{1,2\}}(-1)^{i+j}\left(\theta_{x_i,X_j}\right)^k}\d\sigma(z)\  ,
\end{eqnarray*}
where  $r_k$ is defined in \eqref{eqn:rk}.
	\end{theorem}

After a preliminary definition of the rhombus function, we state our two main results. Then we proceed through the proof. It is enough to study oper variations, since Hitchin variations are twice the ``real part'' of oper variations. The idea is to first  study the variation of cross ratios for points inside $\mathbb H^2$ mapped into $\mathbb{CP}(E)$ and  $\mathbb{CP}(E^*)$ by the associated Frenet immersion (see Section \ref{sec:opers}). Once this is done we extend this variational formula to points in the boundary at infinity in Section \ref{sec:proof-var-oper}. Finally in Section \ref{sec:proof-var-oper} we relate the rhombus function to automorphic forms.

\subsection{Preliminary: the rhombus function}

We recall the following obvious  and classical lemma which follows from the fact that geodesics ending at the same point at infinity in $\mathbb H^2$ approach each other exponentially fast.
\begin{lemma} Let $f$ be a Hölder function on $\ms U\mathbb H^2$.
	Let $\gamma_1$ and $\gamma_2$ be two geodesics in $\mathbb H^2$ so that $\gamma_1(+\infty)=\gamma_2(+\infty)$. Then
	$$
	t\to \int_0^t f({\gamma}_1(s))\,\d s-\int_0^t f({\gamma}_2(s))\,\d s\ ,
	$$
	admits a limit when $t$ goes to $+\infty$.
\end{lemma}

The previous lemma allows us to make the following definition.
Given four points 
$x,y,X,Y\in\partial_\infty\mathbb H^2$,
let $h_x,h_y,h_X,h_Y$ be corresponding Busemann functions. Let $\gamma_{z,Z}$ be the geodesic going from $z$ to $Z$.   For any  $a\in\{x,y\}$, $A\in\{X,Y\}$ and $c\in\{a,A\}$, choose $t^c_{a,A}$ so that
$$
h_c\left(\gamma_{a,A}(t^c_{a,A})\right)=t\ .
$$
For  a Hölder function $f$,
we define the {\em rhombus function} by
\begin{eqnarray*}
 & &\Rh(x,y,X,Y;f)\cr
 &\defeq&\lim_{t\to\infty}\left\{\int_{t^x_{x,X}}^{t^X_{x,X}}f({\gamma}_{x,X})\ \d s
-\int_{t^y_{y,X}}^{t^X_{y,X}}f({\gamma}_{y,X})\ \d s
-\int_{t^x_{x,Y}}^{t^Y_{x,Y}}f({\gamma}_{x,Y})\ \d s
-\int_{t^y_{y,Y}}^{t^Y_{y,Y}}f({\gamma}_{y,Y})\ \d s\right\}
\end{eqnarray*}
Intuitively speaking, $\Rh(x,y,X,Y;f)$ is the alternating sum of integrals along geodesics that pairwise meet at points at infinity. See Figure \ref{fig:rhomb}.

\begin{figure}
  \centering
  \includegraphics[width=0.5\textwidth]{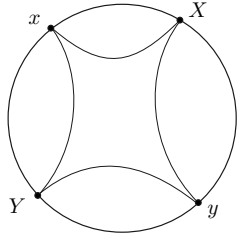}
  \caption{Integrating over a rhombus}
  \label{fig:rhomb}
\end{figure}

The following proposition is a consequence of hyperbolicity.

\begin{proposition}\label{approxRhomb}
Let $f$ be a Hölder function on $\ms U\mathbb H^2$. Let $x_1,x_2,X_1,X_2$ be pairwise distinct points in 
$\partial_\infty\mathbb H^2$.
Let  $\{x_{i,j}^n\}_{n\in\mathbb N}$ and $\{X_{i,j}^n\}_{n\in\mathbb N}$ be points in $\mathbb H^2$, $i,j\in\{1,2\}$. Assume that 
\begin{itemize}
	\item $\lim_{n\to\infty}(x_{i,j}^n)=x_i$.
	\item $\lim_{n\to\infty}(X_{i,j}^n)=X_j$.
	\item $\lim_{n\to\infty}d(x_{i,j}^n,x_{i,k}^n)=0$.
    \item $\lim_{n\to\infty}
  d(X_{i,j}^n,X_{k,j}^n)=0$.
\end{itemize}
Then
$$\Rh(x_1,x_2,X_1,X_2; f)=
\lim_{n\to\infty}\sum_{i,j\in\{1,2\}}(-1)^{i+j}\int_{x_{i,j}^n}^{X_{i,j}^n} f\left({\eta^n}_{i,j}\right)\ \d s\ ,
$$
where we consider the geodesics $\eta_{i,j}^n$ joining $x_{i,j}^n$ to $X_{i,j}^n$ as curves in 
$\ms U\mathbb H^2$.
\end{proposition}
\begin{proof} We  may as well assume that 
$$
h_{x_i}\left(x_{i,j}^n\right)=h_{x_i}\left(x_{i,k}^n\right)\eqdef t_i^n\ ,
$$
and
$$
h_{X_j}\left(X_{i,j}^n\right)=h_{X_j}\left(X_{k,j}^n\right)\eqdef T_j^n\ .
$$
Let then $z_{i,j}^n$ and $Z_{i,j}^n$ be the points in the geodesic $\gamma_{i,j}\defeq\gamma_{x_i,X_j}$ so that
$$
h_{X_j}\left(Z_{i,j}^n\right)\eqdef T_j^n, \ \ h_{x_j}\left(z_{i,j}^n)\right)\eqdef t_i^n\ .
$$  
Elementary hyperbolic geometry implies that the sequences $\{d(z_{i,j}^n,x_{i,j}^n)\}_{n\in\mathbb N}$ and $\{d(Z_{i,j}^n,X_{i,j}^n)\}_{n\in\mathbb N}$ all converge to zero. For a  positive number $\alpha$, there exists a positive $\epsilon$ such that if $d(x,z)\leq\epsilon$ and $d(X,Z)\leq\epsilon$, then
$$
\left\vert\int_x^X f\left({\gamma}_{x,X}\right)\d s-\int_z^Z f\left({\gamma}_{z,Z}\right)\d s\right\vert\leq \alpha\ .
$$
It follows that
$$
\lim_{n\to\infty}
\left\vert
 \int_{x_{i,j}^n}^{X_{i,j}^n}
  f\left({\eta}_{i,j}^n\right)\ \d s-\int_{z_{i,j}^n}^{Z_{i,j}^n} f\left({\gamma}_{i,j}^n\right)\ \d s  \right\vert=0\ .
$$
The proposition follows from this last assertion.
	
\end{proof}

\subsection{Statement of the results}\label{sec:stat-cr} We can now state our two main results on the rhombus function.

\subsubsection{The period formulation for variations of cross ratio}

The following will be proved in  Section \ref{sec:proof-var-oper}.

\begin{theorem}\label{thm:cr1}
	Let $(x,y,X,Y)\in \partial_\infty\grf^{4*}$. Let $\delta_0$ be a Fuchsian representation giving rise to a hyperbolic structure on $\Sigma$, and let $\delta_t$ be a smooth family of representations of $\grf\to\ms{SL}(2,\mathbb C)$ so that the tangent vector at $\delta_0$ is a standard oper variation given by a holomorphic $k$-differential $q$. Let $b^O_t$ be the corresponding family of cross ratios on $\partial_\infty\pi_1(\Sigma)$. Then
	
	\begin{equation*}
		\left.\frac{\d }{\d t}\right\vert_{t=0}\log b^O_t(x,y,X,Y)=\left(\frac{(-1)^{k}}{2^{k-1}}\frac{(n-1)!}{(n-k)!}\right)\cdotp\Rh(x,y,X,Y; \hat q)\ ,
	\end{equation*}
where $\hat q$ is the function defined in Section \ref{subsec:integration}.
%$\ms U\mathbb H^2$ by
%$
%\hat q(u)=q(\underbrace{u,\ldots,u}_k)
%$.

Similarly, assume that the tangent vector at $\delta_0$ is a standard Hitchin variation given by a holomorphic $k$ differential $q$. Let $b^H_t$ be the corresponding family of cross ratios on $\partial_\infty\pi_1(\Sigma)$.
Then
	
\begin{equation*}
		\left.\frac{\d }{\d t}\right\vert_{t=0}\log b^H_t(x,y,X,Y)=\left(\frac{(-1)^{k}}{2^{k-2}}\frac{(n-1)!}{(n-k)!}\right)
		\cdotp\Re\left(\Rh(x,y,X,Y; \hat q)\right)\ .
		\end{equation*}
\end{theorem}

%	The second assertion is an immediate consequence of the first one using the description of the tangent space to the Fuchsian point of Section \ref{sec:moduli}.
%	
%The proof now follows from Proposition 	\ref{pro:varoper} and the following Proposition whose proof occupies the rest of this paragraph
%
%\begin{proposition}\label{pro:dt-b}
%Assume that the function defined for every quadruple of pairwise distinct points 	in $\mathbb D$ by
%$$
%(x,y,X,Y)\mapsto \dt {b^O}(x,y,X,Y),
%$$
%extends continuously to a function $f$ defined on quadruple of pairwise distinct points of $\partial\mathbb D$. Then, for every $(x,y,X,Y)$ pairwise distinct in $\partial\mathbb D$,
%$$
%f(x,y,X,Y)=\dt {b^O}(x,y,X,Y).
%$$
%\end{proposition}

\subsubsection{The automorphic form formulation}

An alternative formula for the rhombus function, and hence for the variation of cross ratios, is given in the following result. We prove this  theorem in Section \ref{sec:proof-rhomb-auto}.

\begin{theorem}\label{thm:rhomb} Let $q$ be a $\grf$-automorphic holomorphic $k$-differential on ${\mathbb H}^2$, and let $x,y,X,Y\in \mathbb H^2$ be pairwise distinct points. Then
	$$
	\Rh(x_1,x_2,X_1,X_2; \hat q)=r_k\cdotp\int_{\mathbb H^2}\Braket{q(z),
	\sum_{i,j\in\{1,2\}}(-1)^{i+j}\left(\frac{(X_j-x_i)\ \d z}{(z-x_i)(z-X_j)}\right)^k}\d\sigma(z)\  ,
$$
where  $r_k$ is defined in \eqref{eqn:rk}.
	\end{theorem}

\subsection{Opers, Frenet immersions and cross ratios}\label{sec:frenet}

let $D$ be a nonslipping connection on $\mc P$ equipped with the flag structure $\mc F_1\subset\mc F_p\subset\cdots\subset\mc F_n=\mc P$. Let $\mathbb D$ be the universal cover of  $X$. Let $\delta$ be the holonomy of this connection. Then $\mc F_1$, viewed
 as a  line subbundle over $\mathbb D$, defines  a  $\delta$-equivariant holomorphic map  $\mathbb D\to\mathbb{CP}(\mc P)$. 
Moreover, this map is a {\em Frenet immersion}: for every $k$, the derivatives  up to order $k$ generate  a $k$-plane (called the $k^{\hbox{\tiny th}}$-osculating plane of $\mc F$) which is actually identified to the projective subspace generated by $\mc F_k$.
Conversely, any such Frenet immersion uniquely defines  an oper, and we have thus described an isomorphism between opers up to gauge equivalence, and Frenet immersions up to precomposition by  projective transformations.

To an oper we can also associate  a {\em cross ratio}, which in this case is a function of four generic points of $\mathbb D$. The construction goes as follows: let $x,y,X,Y$ be four points on $\mathbb D$. We trivialize $P$ over $\mathbb D$ as $\mathbb D\times E$ and consider $\mc F_1\eqdef\xi$ and $\mc F_{n-1}\eqdef \xi^*$ as maps from $\mathbb D$ to $\mathbb{CP}(E)$ and $\mathbb{CP}(E^\ast)$, respectively. 
Then by definition the cross ratio of  four generic points is 
\begin{equation*}
	b^O(x,y,X,Y)=\frac{\braket{\hat \xi(x)\mid \hat\xi^*(X)}\braket{\hat \xi(y)\mid \hat \xi^*(Y)}}{\braket{\hat \xi(x)\mid \hat \xi^*(Y)}\braket{\hat \xi(y)\mid \hat \xi^*(X)}}\ ,
\end{equation*}
where $\hat\xi(z)$, $\hat\xi^*(z)$ is a nonzero vector in $\xi(z)$ and  $\xi^*(z)$, respectively.
 %for $z=x,y,X,\hbox {or} Y$ and $\ast=1$ or $\ast=n-1$. 
 The generic condition is that $\mc F_1(x)\not\subset \mc F^0_{n-1}(Y)$ and $\mc F_1(y)\not\subset\mc F^0_{n-1}(X)$. For the Veronese oper, this condition is always satisfied provided $x\not=Y$ and $y\not= X$. Thus given $x,y,X,Y$, there exists an open set of opers for which the cross ratio is defined at $x,y,X,Y$.

It then makes sense to compute the variation of $ b^O(x,y,X,Y)$ along an oper deformation. 
The main result of this section is an explicit formula for this variation (for points inside the disk
 $\mathbb D$) at the Veronese oper. More precisely, we shall prove

\begin{proposition} \label{pro:varoper}
Let $\{x_1^n,x_2^n,X_1^n,X_2^n\}_{n\in\mathbb N}$ be a sequence of quadruples in $\mathbb D$ converging to $x_1,x_2,X_1,X_2$ in $\partial\mathbb D$. Then along an oper deformation associated to a holomorphic $k$-differential $q$, 
$$
\lim_{n\to\infty}\frac{\dt{b^O} (x^n_1,x^n_2,X^n_3,Y^n_3)}{b^O (x^n_1,x^n_2,X^n_1,Y^n_2)}=	A(k,n) \cdotp \Rh(x_1,x_2,X_1,X_2;\hat q)\ .
$$	
where $A(k,n)=\frac{(-1)^{k}}{2^{k-1}}\frac{(n-1)!}{(n-k)!}$.
\end{proposition}

We will prove this proposition in Section \ref{subsec:proof-varoper}.

\subsubsection{Lifting and elementary functions}
Let us choose points $z$ and $Z$ in $\mathbb D$ and a path $\gamma:[0,1]\to\mathbb D$ so that $\gamma(0)=z$ and $\gamma(1)=Z$. Let $\zeta_0\in \xi(z)$ and $\zeta_1^*\in\xi^*(Z)$ be nonzero vectors and covectors. Consider the
 {\em elementary function} $G_{z,Z}$ on the space of connections on $E$ defined by
$$
G_{x,Z}(\nabla)=\braket{\zeta_1^*,\zeta_1(\nabla)},
$$
where $t\to\zeta_t(\nabla)$ is the $\nabla$-parallel section along $\gamma$ starting at $\zeta_0$. Let also $t\to\zeta^*_t(\nabla)$ be the $\nabla$ parallel section so that $\zeta^*_1(\nabla)=\xi^*_1$.
From the definition of the cross ratio we immediately get

\begin{proposition}\label{pro:cr-in}
For $x_1,x_2,X_1,X_2$ in $\mathbb D$, and an oper variation associated to a $k$-holomorphic differential $q$, one has
\begin{equation*}
\frac{\dt{b^O} (x_1,x_2,X_1,X_2)}{b^O (x_1,x_2,X_1,X_2)}=\sum_{(i,j)\in\{1,2\}}(-1)^{i+j}\frac{ \d_\nabla G_{x_i,X_j}(\dt \nabla)}{G_{x_i,X_j}(\nabla)}\ 
\end{equation*}

\end{proposition}

 Observe that, fixing $z$ and $Z$, $G_{z,Z}$ is well defined up to a multiplicative constant. Then let 
$$
p_{z,Z}(s)=\frac{\zeta_s^*\otimes\zeta_s(\nabla)}{\braket{\zeta_s^*\mid \zeta_s(\nabla)}}\ ,
$$
be the projection on the $\nabla$-parallel line $L_z$ generated by $\zeta_s$,  along the $\nabla$-parallel hyperplane $P_Z=\ker(\zeta^*_s)$.
Then the variation of parameters method gives
\begin{proposition}\label{ElemFunction}
We have
\begin{equation*}
\frac{ \d_\nabla G_{z,Z}(\dt\nabla)}{G_{z,Z}(\nabla)}=\int_0^1 \tr\left( p_{z,Z}(\gamma)\cdotp \dt\nabla\right)\,\d s\ .
\end{equation*}
\end{proposition}
\begin{proof} Let $\nabla_s$ be a 1-parameter family of connections so that $\left.\frac{\d}{\d s}\right\vert_{s=0}\nabla_s=\dt \nabla$.
Let $\beta$ be the section of  $\operatorname{End(P)}$ along $\gamma$ so that $\beta(\gamma(0))=0$ and $\nabla_{\dt{\gamma}(t)}\beta=\dt{\nabla}\left(\dt{\gamma}\right)$. Let $G^s$ be the family of sections of $\operatorname{End(P)}$ along $\gamma$ so that $G^s(z)={\rm Id}$ and $(G^s)^*\nabla=\nabla^s$. Then by construction 
$$
  \left.\frac{\d }{\d s}\right\vert_{s=0}G^s=\beta\ .
  $$
  Moreover, 
  $$
  G_{z,Z}(\nabla^s)=\langle \zeta_1^*(\nabla)\mid G^s\left(\zeta_1(\nabla)\right)\rangle\ .
  $$
It follows that
 $$
  \d_\nabla G_{z,Z}(\dt\nabla)=\braket{\zeta_1^*(\nabla)\mid \beta\left(\zeta_1(\nabla)\right)}\ .
  $$
Let now
$$
f(s)=\braket{ \zeta_s^*(\nabla)\mid \beta\left(\zeta_s(\nabla))\right)},
$$
so that
\begin{eqnarray*}
  \d_\nabla G_{z,Z}(\dt\nabla)&=&f(1)-f(0)\cr
  &=&\int_0^1 \frac{\d f}{\d s}\,\d s\cr
  &=&\int_0^1 \braket{ \zeta_s^*(\nabla)\mid\dt\nabla\left(\zeta_s(\nabla))\right)}\,\d s\ .
\end{eqnarray*}
Since $G_{z,Z}(\dt\nabla)=\tr(\zeta^*_s(\nabla)\otimes\zeta_s(\nabla))$, this completes the proof of the proposition. \end{proof}

\subsubsection{The Fuchsian bundle along a geodesic again}

Let $\gamma$ be a geodesic, we need the following

\begin{lemma}\label{limproj} There exist positive number $\alpha$ and $t_0$ only depending on $n$, so that
 for any $k\in\{2,\ldots,n\}$, for any $s, t>t_0$
	$$
	\left\vert\tr\left(E^0_{k-1}\cdotp p_{\gamma(-s),\gamma(u)}\right)-A(k,n)\right\vert\leq e^{-\alpha u}+e^{-\alpha s}\ .
	$$	
\end{lemma}
\begin{proof} We use the notation of Section \ref{sec:fbg}. Let $\gamma$ be a geodesic and lets consider the bundle $\mc E$ in \eqref{eqn:E}.
%$$
%S^{1-n}\oplus S^{1-n+2}\oplus\ldots\oplus S^{n-1-2}\oplus S^{n-1}=\mc E\ .
%$$
As noted previously, this bundle inherits a (split) connection $\nabla$ from the hyperbolic structure, as well as a (nonsplit) connection $D$ from the flat connection.
From the description of the Veronese oper in Section \ref{sec:vero-oper} we have $\xi(z)=S^{n-1}$ and $\xi^*(z)=S^{1-n}\oplus\ldots S^{n-1-2}$.
Finally we have a $\nabla$-parallel frame $\{u_p\}$ defined in eq.\
\eqref{eqn:up} and which satisfies by eq.\ \eqref{eqn:Dup}
$$
D_{\gamma'} u_p=\alpha_p.u_p\ ,
$$ 
where $\alpha_n>\ldots \alpha_p>\alpha_{p-1}>\alpha_1>0$. Then $L_p$ is the (parallel) complex line generated by $u_p$ and let $\pi_p$ be the orthogonal projection on $L_p$ for all $p$. 

By construction, $\pi_1(\xi(z))$ is always nonzero. Then, since $\xi(z)$ is $\nabla$ parallel, it follows that there exist positive $\alpha$, $K$ and $t_0$  so that for $s>t_0$
$$
d(\xi(\gamma(s)),L_1)\leq K\cdotp e^{-\alpha s},
$$
where $d$ is the metric on $\mathbb{CP}(\mc E)$ inherited from the metric on $\mc E$. It then follows (using a similar argument on $E$) that we have a constant $K$ so that for $s,u>t_0$
$$
d(\pi_1,p_{\gamma(-s),\gamma(u)})\leq K(e^{-\alpha s}+e^{-\alpha u})\ .
$$
The result now follows from eq.\ \eqref{A:tr-high}.
\end{proof}

As a corollary of 
Proposition \ref{ElemFunction} and Lemma \ref{limproj} we get
\begin{corollary}\label{coro:elem}
There exist positive number $\alpha$ and $t_0$ only depending on $n$, so that
 for any $k\in\{2,\ldots,n\}$, for any $s, u>t_0$
\begin{equation*}
\left\vert
\frac{ \d_\nabla G_{\gamma(-s),\gamma(u)}(\dt\nabla)}{G_{\gamma(-s),\gamma(u)}(\nabla)} - A(k,n) \cdotp \int_{-s}^u q(\underbrace{\gamma,\ldots,\gamma}_k)\,\d s\right\vert \leq e^{-\alpha s}+e^{-\alpha u}\ .
\end{equation*}
\end{corollary}
\begin{proof}
Indeed, for a standard  oper deformation associated to a holomophic differential of degree $k$, we have  
$$
\tr\left( p_{z,Z}(\gamma)\cdotp \dt\nabla\right)=q({\gamma},\ldots,{\gamma})\tr(E^0_{k-1}\cdotp p_{z,Z})\ .
$$
The corollary now follows.
\end{proof}
\subsubsection{Proof of Proposition \ref{pro:varoper}} \label{subsec:proof-varoper}

Let $\{x_1^n,x_2^n,X_1^n,X_2^n\}_{n\in\mathbb N}$ be a sequence of quadruples in $\mathbb D$ converging to $x_1,x_2,X_1,X_2\in \partial\mathbb D$. Let $\gamma^n_{i,j}$ be the geodesics joining $x^n_i$ to $X^n_j$. Since $x_1,x_2,X_1,X_2$ are all pairwise distinct, we may choose a parametrization so that $\gamma_{i,j}^n(0)$ stays in a compact region. Thus let $\gamma^n_{i,j}(s^n_{i,j})=x_i$ and $\gamma^n_{i,j}(s^n_{i,j})=X_j$.

From Corollary \ref{coro:elem} and Proposition \ref{pro:cr-in}, it follows that 
$$
\left\vert  \frac{\dt{b^O} (x^n_1,x^n_2,X^n_3,Y^n_3)}{b^O (x^n_1,x^n_2,X^n_1,Y^n_2)} -A(k,n)\cdotp\left(\sum_{i,j\in\{1,2\}}(-1)^{i+j} \int_{\gamma^n_{i,j}} \hat q\,\d s  \right)  \right\vert\leq \sum_{i,j}e^{-\alpha s^n_{i,j}}+e^{-\alpha S^n_{i,j}}\ .
$$
The result now follows from Proposition \ref{approxRhomb}.

\subsection{Extension to the boundary at infinity}\label{sec:proof-var-oper}

We can now  prove our first result on variation of cross ratios: Theorem \ref{thm:cr1}. The statement consists of two assertions.
The second assertion is an immediate consequence of the first one using the description of the tangent space to the Fuchsian point of Section \ref{sec:moduli}.
	
The proof of the theorem therefore  follows from Proposition 	\ref{pro:varoper} and the following result, whose proof occupies the rest of this section

\begin{proposition}\label{pro:dt-b}
Assume that the function defined for every quadruple of pairwise distinct points 	in $\mathbb D$ by
$$
(x,y,X,Y)\mapsto \left.\frac{\d }{\d t}\right\vert_{t=0} {b^O}(x,y,X,Y)\ ,
$$
extends continuously to a function $f$ defined on quadruple of pairwise distinct points of $\partial\mathbb D$. Then, for every $(x,y,X,Y)$ pairwise distinct in $\partial\mathbb D$, 
$$
f(x,y,X,Y)=\left.\frac{\d }{\d t}\right\vert_{t=0}{b^O}(x,y,X,Y)\ .
$$
\end{proposition}

\subsubsection{Opers and the boundary at infinity}
For the moment we prove the following result. Let us first notice that, by the openness property of Anosov representations there exists a neighborhood  $U_0$ of the Veronese oper, so that if an oper $O$ belongs to $U_0$ its monodromy is also Anosov. In particular, one can associate to $O$ a cross ratio on $\partial\mathbb D$ (identified with $\bgrf$), which will also be denoted by $b^O$.

Although it is quite likely that on a smaller neighborhood, the cross ratio is actually defined for all quadruple of distinct points and extends continuously to the boundary at infinity, we content ourselves with a weaker and easier result.

In this section and the next one, $\xi_O$, $\xi^*_O$ are  the Frenet immersions from $\mathbb D$ to $\mathbb{CP}(E)$ and  $\mathbb{CP}(E^*)$, respectively, associated to the oper $O$, and $\xi$, $\xi^*$ those associated to the Veronese oper as in Section \ref{sec:frenet}. Let also denote by $\delta_O$ (and similarly $\delta$) the monodromy of the oper $O$ and the Veronese oper, respectively. Finally the contragredient of a representation $\eta$ is denoted $\eta^*$.

\begin{proposition}\label{pro:lim-cr}
Let $x$ be a point in $\mathbb D$, let $\alpha_1$, $\alpha_2$, $\alpha_3$ and $\alpha_4$ be four nontrivial elements of $\grf$. Let $\alpha_i^+$ be the attracting point of $\alpha_i$ in $\partial\mathbb D$. Assume that $\alpha_1^+$, $\alpha_2^+$, $\alpha_3^+$ and $\alpha_4^+$ are all pairwise distinct. Then there exists $n_0$ and a neighborhood $U$ of the Veronese oper, such that if $O\in U$, then
\begin{itemize}
	\item First, for $n>n_0$, $b^O(\alpha_1^n(x),\alpha_2^n(x),\alpha_3^n(x),\alpha_4^n(x))$ is well defined,
	\item Moreover, 
	$\lim_{n\to\infty}b^O(\alpha_1^n(x),\alpha_2^n(x),\alpha_3^n(x),\alpha_4^n(x))=b^O(\alpha_1^+,\alpha_2^+,\alpha_3^+,\alpha_4^+)$,
	and the limit is uniform in $O$.
\end{itemize}	
\end{proposition}
\begin{proof}
 Observe that $\xi(x)$ and $\xi^*(x)$ both lie in the basin of attraction of the unique fixed point  of $\delta(\alpha_i)$ in either $\mathbb{CP}(E)$ or $\mathbb{CP}(E^*)$. By continuity the same holds for $\xi_O$ an $\xi_O^*$ for an oper $O$ close enough to the Veronese oper.  
  It then follows that $\xi_O(\alpha_i^n(x))$ converges to the attracting fixed point of $\delta_O(\alpha_i)$, and similarly $\xi^*_O(\alpha_i^n(x))$ converges to the attracting fixed point of $\delta^*_O(\alpha_i)$. The results follow  immediately, and the uniform convergence in $O$ is similarly obtained.
\end{proof}
\subsubsection{Analyticity}

We note the following

\begin{proposition}\label{pro:analytic}
Let $x,y,X,Y$ in $\mathbb D$ (or all in $\partial\mathbb D$). Then
the function $O\mapsto b_O(x,y,X,Y)$ is complex analytic in a neighborhood of the Veronese oper. 	
\end{proposition}

\begin{proof}
	For points in $\partial\mathbb D$, this follows from \cite[Theorem 6.1]{Bridgeman:2013tn}. For points in $\mathbb D$ this follows from the fact that for a differential equation whose parameters vary  complex analytically, the solution also varies  complex analytically.
\end{proof}

\subsubsection{Proof of Proposition \ref{pro:dt-b}}

We remark that by continuity it is enough to prove that on a dense set of points in $\partial\mathbb D$,
$$
f(x,y,X,Y)=\dt{b^O}(x,y,X,Y)\ .
$$
We thus will only prove this proposition when $x,y,X,Y$ are the attracting points of elements $\alpha_i$ in $\grf$. But by Proposition \ref{pro:lim-cr}, there exists sequences $\{x_n\}_{i\in\mathbb N}$, $\{y_n\}_{i\in\mathbb N}$
,$\{X_n\}_{i\in\mathbb N}$
 and $\{Y_n\}_{i\in\mathbb N}$ converging respectively to $x$, $y$, $X$ and $Y$ so that the function 
 $$
 O\mapsto b^O(x_n,y_n,X_n,Y_n),
 $$
 converges uniformly to $O\mapsto b^O(x,y,X,Y)$. Since all the  functions involved are analytic in $O$ by Proposition \ref{pro:analytic}, it follows that their derivatives in $O$ also converges. This is what we needed to prove, and the proof of Proposition \ref{pro:dt-b} is now complete.

\subsection{The rhombus function and  automorphic forms}\label{sec:proof-rhomb-auto}

In this section, we prove  Theorem \ref{thm:rhomb} to give an alternative formula for the rhombus function, and hence also for the variation of cross ratios.

\subsubsection{Slabs and cusps}
For $u$ and $v$ in $\mathbb H^2$, let $c$ be the geodesic arc between $u$ and $v$. The {\em slab} $S(u,v)$ defined by $u$ and $v$ is the region bounded by the two orthogonal geodesics $c_cx$ and $c_y$ to $c$ at  $u$ and $v$, respectively. We extend the notion of slabs to $u$ and being possibly at infinity.

We need the following  two propositions, proved in the next two sections,  and their corollaries.
	%\subsubsection{Integration over slabs}
Our first proposition and corollary deal with integration over slabs. Recall \eqref{eqn:theta}.
	\begin{proposition}\label{pro:slab}{\sc[Integration over slabs]}
There exists a constant $K$, such that such that for any bounded $k$-differential $q$, and $u,v\in\mathbb H^2$, we have 
	$$
	\left\vert\int_{S(u,v)}\braket{q,\theta_{u,v}^k}\d\sigma\right\vert\leq K\cdotp\Vert q\Vert_\infty\cdotp d(u,v)\ .
	$$
	\end{proposition} 
	\begin{proof}
We can work in the upper half plane model and assume that $u$ and $v$ belong to the imaginary axis so that 
$$
\theta_{u,v}=-\frac{\d z}{z}\ .
$$
Then the slab $S(u,v)$ is defined in polar coordinates as
$$
S(u,v)=\{(r,\theta)|\theta\in[-\frac{\pi}{2},\frac{\pi}{2}],\,\, \vert u\vert\leq r\vert v\vert\}\ .
$$
Then 
$$
\d\sigma= \frac{1}{r\sin(\theta)^2}\d r\d\theta, \, \, \, \Vert \theta^{k}_{u,v}\Vert=\sin^k(\theta)\ .
$$
Thus the result follows from a simple integration with $K=\int_{-\frac{\pi}{2}}^{\frac{\pi}{2}}\sin^{k-2}(\theta)\d\theta$.
		
	\end{proof}
	
It will be useful to note the following
\begin{corollary}\label{coro:slab}
	Let $q$ be a bounded $k$-differential. Then there exist constants $K$ and $\epsilon_0$ such that if $\epsilon<\epsilon_0$ and  $d(u_0,v_0)+d(u_1,v_1)\leq \epsilon$ and $d(u_0,u_1)\geq K$ then
	\begin{equation*}
		\left\vert\int_{S(u_0,u_1)}\braket{q,\theta^k_{u_0,u_1}}\d\sigma- \int_{S(v_0,v_1)}\braket{q,\theta^k_{v_0,v_1}}\d\sigma\right\vert\leq K\cdotp\epsilon\ .
	\end{equation*} 
\end{corollary}	
\begin{proof}
	Let $\gamma$ be the geodesic passing through $u_0$ and $u_1$ so that $\gamma(0)=u_0$ and $\gamma(\ell_\gamma)=u_1$. Then, by continuity and elementary hyperbolic geometry, there exist constants $A$ and $\epsilon_0$ such that if $\epsilon<\epsilon_0$ and  $d(u_0,v_0)+d(u_1,v_1)\leq \epsilon$ and $d(u_0,u_1)\geq K$ then, writing $u_0^\pm=\gamma(\pm A\epsilon)$ and $u_1^\pm=\gamma(A(\ell_\gamma\pm\epsilon))$, we have
$$
S(u_0^+,u_1^-)\subset S(v_0,v_1)\subset S(u_0^-,u_1^+)\ .
$$ 
Now let $\chi$ be the characteristic function of $S(v_0,v_1)$. Then using  the previous proposition twice yields
\begin{eqnarray*}
	& &\left\vert\int_{S(u_0,u_1)}\braket{q,\theta^k_{u_0,u_1}}\d\sigma- \int_{S(v_0,v_1)}\braket{q,\theta^k_{v_0,v_1}}\d\sigma\right\vert\cr
	&\leq & 2AK\cdotp\epsilon\Vert q\Vert_\infty +\left\vert\int_{S(u_0^-,u_1^+)}\braket{q,\theta^k_{u_0,u_1}}\d\sigma- \int_{S(v_0,v_1)}\braket{q,\theta^k_{v_0,v_1}}\d\sigma\right\vert\cr
	&\leq & 2AK\cdotp\epsilon\Vert q\Vert_\infty +\left\vert\int_{S(u_0^-,u_1^+)}\braket{(1-\chi)q,\theta^k_{u_0,u_1}}\d\sigma\right\vert\cr
	&\leq & 2AK\cdotp\epsilon\Vert q\Vert_\infty + \left\vert\int_{S(u_0^-,u_0^+)}\braket{(1-\chi)q,\theta^k_{u_0,u_1}}\d\sigma\right\vert+\left\vert\int_{S(u_1^-,u_1^+)}\braket{(1-\chi)q,\theta^k_{u_0,u_1}}\d\sigma\right\vert
	\cr
	&\leq  &6AK\cdotp\epsilon \Vert q\Vert_\infty\ .
\end{eqnarray*}
The result follows from this last inequality.	
\end{proof}
%\subsubsection{Cusp annihilation}	
\begin{proposition}\label{cusp-annihil}
	Let $q$ be a Lipschitz differential. Let $\gamma$ be a geodesic from $x$ to $X$ in $\partial_\infty\mathbb H^2$. Let $\tau$ be a parabolic transformation fixing $X$ then
		\begin{equation*}
	\lim_{t\to +\infty}\int_{S(\gamma(t),X)} \braket{q-\tau^*q, \theta^k_{\gamma^-,\gamma^+}}=0\ .	
	\end{equation*}
\end{proposition}
\begin{proof} We use the upper half plane model and the geodesic joining $0$ to $\infty$.
Let also $$
\tau : z\mapsto\frac{\alpha z}{z+\alpha}\ .
$$
Let $\chi_t$ be the indicatrix of the set $S(\gamma(t),X)$. We want to prove that 
	$$
	\braket{q-\tau^*q, \left(\frac{\d z}{z}\right)^k}\cdotp\chi_{t}\xrightarrow{L^1} 0,
	$$
	where the convergence in $L^1(\mathbb H^2,\sigma)$.
	Observe that since $q$ is Lipschitz, there exists a constant $K$ so that 
	$$
	\Vert q-\tau^*q\Vert_z\leq K\cdotp d(z,\tau(z))
	$$
	Moreover, there exists a constant $K_2$, so that $d_{\mathbb H^2}(z,\tau(z))\leq K_2\Im(z)$.
	It follows that we have a constant $K_3$ so that, using polar coordinates,
	$$
	\left\vert\braket{q-\tau^*q, \left(\frac{\d z}{z}\right)^k}\right\vert\leq K_3 \frac{\Im(z)^{k+1}}{\vert z^k\vert }=K_3\cdotp r\sin^{k+1}(\theta)\ .
	$$
	Integrating along the area form of the hyperbolic space 
	$$
	\d \sigma=\frac{\d r\d\theta}{r\sin(\theta)^2}\ ,$$ we get 
	$$
	\left\vert\braket{q-\tau^*q, \left(\frac{\d z}{z}\right)^k}\right\vert \,\d\sigma\leq K_3\sin^{k-1}(\theta)\d r\d\theta\ .
	$$
The result follows.
\end{proof}

As a corollary of Proposition \ref{cusp-annihil}, we have

\begin{corollary}\label{defRhauto}
	Let $x,X_1,X_2\in\partial_\infty\mathbb H^2$. Let $h$ be a Busemann function associated to $x$, so that $h(-\infty)=x$. Let $\gamma_i$ be   geodesics from $x$ to $X_i$ for $i\in\{1,2\}$. We parametrise the geodesics so that $h\circ\gamma_i(t)=t$. Then
	for any sequence $\{s_n\}_{n\in\mathbb N}$, $\{t_n\}_{n\in\mathbb N}$ going to $+\infty$ so that $t_n\geq s_n$, for any automorphic form $q$, we have
	
	\begin{equation*}
		\lim_{n\to\infty}
		\left(\sum_{i\in\{1,2\}}(-1)^i
		\int_{S(\gamma_i(s_n),\gamma_i(t_n))}
		\braket{q,\theta^{k}_{x,X_i}}\,\d\sigma\right)=0
	\end{equation*}
	\end{corollary}
	
\subsubsection{Conclusion of the proof of Theorem \ref{thm:rhomb}}
Let $\Omega
$ be the set of those points in $\partial_\infty\mathbb H^2$ which are end points of geodesics whose projection on $\ms U X$ have a  dense $\omega$-limit set. Observe then that any geodesic ending on $\Omega$ satisfies the latter property. By ergodicity and Poincaré recurrence, $\Omega$ is  nonempty, and hence, by $\grf$ invariance, dense.

By density, it is enough to prove the result when $(x_1,x_2,X_1,X_2)$ belongs to $\Omega$. Let $\gamma_{ij}$ be the geodesic joining $x_i$ to $X_j$. Then by hypothesis, for all $i,j\in\{1,2\}$, there exist sequences of points   $\{u^n_{i,j}\}_{n\in\mathbb N}$ and $\{v^n_{i,j}\}_{n\in\mathbb N}$  of $\gamma_{i,j}$ converging  to $x_i$ and $X_j$, respectively, as well as elements $\eta^n_{i,j}$ in $\grf$ so that for all $i,j,k\in\{1,2\}$ ,
\begin{eqnarray}
	d(u^n_{i,j},u^n_{i,k})&\xrightarrow{n\to\infty}&0\ ,\notag\\
	d(v^n_{j,i},v^n_{k,i})&\xrightarrow{n\to\infty}&0\ ,\notag\\
	d(u^n_{i,j},\eta^n_{i,j}(v^n_{i,j}))&\xrightarrow{n\to\infty}&0\ .\label{close-lem}
\end{eqnarray}
Applying the closing lemma, we deduce from  \eqref{close-lem} that there exist $\hat u^n_{i,j}$ and $\hat v^n_{i,j}$  on the axes of
 $\eta^n_{i,j}$ with $\eta^n_{i,j}(\hat  u^n_{i,j})=\hat v^n_{i,j}$ so that
\begin{eqnarray}
	d(u^n_{i,j},\hat u^n_{i,j})&\xrightarrow{n\to\infty}&0\ ,\label{eqn:cu}\\
	d(v^n_{i,j},\hat v^n_{i,j})&\xrightarrow{n\to\infty}&0\label{eqn:cv}\ .
\end{eqnarray}
 Let us denote by $\chi(u,v)$ the characteristic function of $S(u,v)$. By Corollary \ref{defRhauto}, as $n\to\infty$,
\begin{equation*}
	\sum_{i,j}(-1)^{i+j}\chi(u^n_{i,j},v^n_{i,j})\theta^k_{x_i,X_j}\starrightarrow	\sum_{i,j}(-1)^{i+j}\theta^k_{x_i,X_j}\ ,
\end{equation*}
where $\ast$ denote the convergence in the dual of the space $C^{0,1}(\mathbb H^2)$  of Lipschitz function on $\mathbb H^2$.
Then by \eqref{eqn:cu} and  \eqref{eqn:cv}, as well as Corollary \ref{coro:slab}, we obtain,
\begin{equation}
	\sum_{i,j}(-1)^{i+j}\chi(\hat u^n_{i,j},\hat v^n_{i,j})\theta^k_{\hat u^n_{i,j},\hat v^n_{i,j}}\starrightarrow	\sum_{i,j}(-1)^{i+j}\theta^k_{x_i,X_j}\label{eqn:rhomb11}\ .
\end{equation}
As in  Proposition \ref{pro:katok}, since $\hat v^n_{i,j}=\eta^n_{i,j}(\hat u^n_{i,j})$, for any automorphic form $q$,

\begin{equation*}
	r_k\cdot\int_{S(\hat u^n_{i,j},\hat v^n_{i,j})}\braket{q,\theta^k_{\hat u^n_{i,j},\hat v^n_{i,j}}}\d\sigma=\int_{[\hat u^n_{i,j},\hat v^n_{i,j}]}\hat q\,\d s\ ,
\end{equation*}
where $[u,v]$ denotes the geodesic between $u$ and $v$. Observe now that
\begin{eqnarray*}
	d(\hat u^n_{i,j},\hat u^n_{i,k})&\xrightarrow{n\to\infty}&0\ ,\\
	d(\hat v^n_{j,i},\hat v^n_{k,i})&\xrightarrow{n\to\infty}&0.
\end{eqnarray*}
Thus we get from Proposition \ref{approxRhomb}, that
\begin{equation*}
	r_k\cdot \sum_{i,j\in\{1,2\}}(-1)^{i+j}\int_{S(\hat u^n_{i,j},\hat v^n_{i,j})}\braket{q,\theta^k_{\hat u^n_{i,j},\hat v^n_{i,j}}}\d\sigma\xrightarrow{n\to\infty} Rh(x_1,x_1,X_1,X_2;\breve q)\ .
\end{equation*}
Combining this last equation with \eqref{eqn:rhomb11} concludes the proof of the theorem.

\subsection{A remark on triple ratios for $\mathsf{SL}(3,\mathbb R)$} 

For representations into $\mathsf{SL}(3,\mathbb R)$, the {\em triple ratio} of three points is defined as 
$$
T(x,y,w)=b^H(x,y,w,m)\cdotp b^H(w,x,y,m)\cdotp b^H(y,w,x,m)\ , 
$$
for all auxiliary points $m$. It follows from the formulae in Theorem \ref{thm:rhomb} that the variation of the triple ratio along a cubic differential $q$ is given by
$$
\left.\frac{\d}{\d t}\right\vert_{t=0}\log(T_t(x_1,x_2,x_3))=r_3\int_{\mathbb H^2}\braket{q(z),\sum_{i\not=j}\theta^3_{x_i,x_j}}\\ \d\sigma(z)\ .
$$
For $\mathsf{SL}(n,\mathbb R)$ there are $\frac{(n-1)(n-2)}{2}$ cross ratios, and similar formulas give only a small fraction of all triple ratios.

\section{Large $n$ Asymptotics and Applications}\label{sec:appl}

In this section, we describe two phenomena related to the large $n$ asymptotics of our formula. In particular, we shall see that asymptotically the relation between the pressure and  $L^2$-metrics becomes remarkably simpler, once we perform a natural renormalization.
 \subsection{Large $n$ asymptotics of the pressure metric}
For large $n$ asymptotics, it is more natural to consider the {\em renormalized highest eigenvalue}
$$
\mu_\gamma=(\lambda_{\gamma})^{\frac{1}{n-1}}\ .
$$
and the associated renormalized pressure metric. The reason for this normalization is that then, by \eqref{eqn:largest}, the highest eigenvalue does not depend on $n$ along the Fuchsian locus.
 
We prove the following asymptotics
\begin{theorem}{\sc [Large $n$-asymptotics]} \label{thm:large-n}
We have the large $n$-asymptotics for the renormalized highest eigenvalue given by
$$
\dt{\mu_\gamma}\sim\mu_\gamma\cdotp  \frac{(-1)^k}{2^{k-2}(k-1)!} \left(\frac{(2k-1)!}{ 3}\right)^{1/2}
\cdotp\int_\gamma\Re(q_{k}(\gamma,\ldots,\gamma)\ \d s\ .
$$
Finally, the large $n$ asymptotics for the renormalized  pressure metric  for the renormalized deformation $\psi(q)$ associated to  a holomorphic $k$-differential $q$ is given by
$$
{\bf P}(\psi(q),\psi(q))\sim\frac{(2k-1)!}{2^{k-1}\cdotp 3\cdotp \pi \vert\chi(X)\vert}\int_X \Vert q\Vert^2\d \sigma\ .
$$
\end{theorem}

%\subsubsection{Preliminary asymptotics}

We need the following consequence of Lemma \ref{lem:trace}.

\begin{lemma} \label{lem:large-n}
As $n\to \infty$ we have the asymptotic expression,
$$
-\tr\left(E^0_kF^0_k\right)\sim n^{2k+1}\cdotp     \frac{(k!)^2}{(2k+1)!}\ .
$$
\end{lemma}

%\begin{proof}
%From Lemma \ref{lem:trace}
%\begin{align*}
%-\tr\left(E^0_kF^0_k\right)&=\left(k!\right)^2\sum_{p=k+1}^n{p-1\choose k}{n-p+k\choose k}=
%\sum_{q=1}^{n-k}\frac{(q+k-1)!}{(q-1)!}\frac{(n-q)!}{(n-q-k)!}\\
%&=\sum_{q=1}^{n-k}[ (q+k-1)(q+k-2)\cdots q][(n-q)(n-q-1)\cdots (n-q-k+1)] \\
%&\sim \sum_{q=1}^{n-k} q^k(n-q)^k \\
%&= \sum_{q=1}^{n-k} q^k\sum_{j=0}^k(-1)^j{k\choose j}q^jn^{k-j} \\
%&=\sum_{j=0}^k(-1)^j{k\choose j}n^{k-j}\sum_{q=1}^{n-k} q^{k+j}\ .
%\end{align*}
%Now by the integral comparison,
%$$
%\frac{(n-k)^{k+j+1}}{k+j+1}\leq
%\sum_{q=1}^{n-k} q^{k+j}
%\leq \frac{(n-k+1)^{k+j+1}}{k+j+1}-\frac{1}{k+j}\ ,
%$$
%from which we obtain
%$$
%-\tr\left(E^0_kF^0_k\right)\sim
% n^{2k+1}\sum_{j=0}^k{k\choose j}\frac{(-1)^j}{k+j+1}\ .
% $$
% Now 
% $$
% \sum_{j=0}^k{k\choose j}\frac{(-1)^j}{k+j+1}=  \sum_{j=0}^k{k\choose j}(-1)^j \int_0^1  t^{k+j}\, \d t
% =\int_0^1 t^k(1-t)^k\, \d t
% =\frac{(k!)^2}{(2k+1)!}\ ,
%$$
% via the relation between the Beta and Gamma functions.
%\end{proof}

\begin{proof}[Proof of Theorem \ref{thm:large-n}]
Since $q_kE^0_{k-1}=q_k\eta_{k-1} E_{k-1}$, where $\eta_{k-1}$ is as in \eqref{eqn:eta}, we apply Theorem \ref{thm:gardiner} to $q_k\eta_{k-1}$ to get
$$
\frac{\dt \mu_\gamma}{\mu_\gamma}=\frac{1}{n-1}\frac{\dt \lambda_\gamma}{\lambda_\gamma}
= (-1)^k
\frac{\eta_{k-1}}{2^{k-2}}\frac{(n-2)!}{(n-k)!} \int_0^{\ell_\gamma} \Re\left(q_k({\gamma},\ldots,{\gamma})\right)\ \d s\ .
$$
We have $d(n)\sim n^3/6$. Using \eqref{eqn:eta},  and combining with
   Lemma \ref{lem:large-n}, this implies
$$
\eta_{k-1}\sim \frac{1}{n^{k-2}(k-1)!}\left(\frac{(2k-1)!}{6}\right)^{1/2}\ .
$$
The result then follows from the fact that
$$
\frac{(n-2)!}{(n-k)!}\sim n^{k-2}\ .
$$
\end{proof}
\subsection{Asymptotic freedom of eigenvalues}
Fix $p\in {\mathbb N}$. As previously, let $\lambda_\gamma$ denote the largest eigenvalue.
 Let $\vec \gamma=(\gamma_1,\ldots,\gamma_p)$ be pairwise distinct conjugacy classes of primitive elements of $\grf$. We can then construct a map
 \begin{align*}
\Lambda^{(n)}_{\vec \gamma} &: {\mathcal H}(\Sigma, n)\longrightarrow{\mathbb  R}^p \ , \\
\delta & \mapsto \left(\lambda_{\gamma_1}(\delta),\ldots,\lambda_{\gamma_p}(\delta)\right)\ .
\end{align*}
We then have 
\begin{theorem}\label{imageopen}
	For $n$ sufficiently large (depending on $\vec\gamma$) the image of    $\Lambda^{(n)}_{\vec \gamma}$ contains an open set.
\end{theorem}

This theorem will be an easy consequence of the following, which is of independent interest:
\begin{proposition}\label{pro:inject}
	Let $\mathcal C$ be the set of conjugacy classes of primitive  elements of $\grf$. Let $k_0\geq 2$ be some integer. Let $\mathbb R^{\mathcal C}$ be the real vector space freely spanned  by elements of $\mathcal C$. Then the real linear  map defined by
	\begin{align}
	\mathbb R^{\mathcal C}&\to \bigoplus_{k=k_0}^\infty H^0(X, K^k)\ ,\cr
	\gamma&\mapsto(\Theta^{(k_0)}_\gamma\ldots,\Theta^{(k)}_\gamma,\ldots)\ .\end{align}
	is injective.
\end{proposition}
\subsubsection{Proof of Proposition \ref{pro:inject}}
The proof  relies on the following
\begin{lemma}\label{fiq}
Let $\{f_{i} : \, i\in\mathbb N\}$ be nonzero holomorphic functions defined on the unit disk $\mathbb D$, and $\{\alpha_i ;\, i\in \mathbb N\}$ be a bounded infinite sequence of  complex numbers . Assume that:
\begin{enumerate}
	\item the real analytic functions $\vert f_{i}\vert^2$ are all pairwise distinct;
	\item  the series $\sum_{i=1}^{\infty} f_{i}(z)$ is pointwise absolutely convergent for every $z\in \mathbb D$;
	\item there exists a sequence $\{N_m\}_{m\in\mathbb N}$, $N_m\to \infty$,  such that for all $m$ and all $z\in \mathbb D$,
\begin{equation*}
	 \sum_{i=1}^\infty \alpha_j\cdot f^{N_m}_{i}(z)=0\ .
\end{equation*}
\end{enumerate}	
Then, for all $i$, $\alpha_i=0$, 
\end{lemma}
\begin{proof} 
Since  $\sum_{i=1}^{\infty}  f_{i}(z)$ is pointwise absolutely convergent, it follows that given $z_0\in \mathbb D$, there exists:
\begin{itemize}
\item an integer $i_0$;
\item a finite set $I_0\subset \mathbb N$ containing $I_0$;
\item a neighborhood $U_0$ of $z_0$;
\item a real number $K_0>1$,
\end{itemize}
 such that 
$$(\ast)\quad
\begin{cases} 
|f_{i}(z_0)|=|f_{i_0}(z_0)| & \text{ for all } i\in J_0 \\
|f_{i_0}(z)| \geq K_0|f_{i}(z)| &\text{ for all } i\not \in I_0\ ,\text{ and all } z\in U_0 \ .
\end{cases}
$$
 If $\sharp\,  I_0>1$, then using the fact that the functions $\vert f_{i}\vert$ are distinct one can find $z_1$ near $z_0$, satisfying conditions $(\ast)$ for $z_1,I_1,U_1,K_1$ (instead of $(z_0,I_0,U_0,K_0)$) where $$
\sharp\,  I_1<\sharp\,  I_0\ .
$$ 
In other words, by induction, we may as well assume that $I_0=\{i_0\}$.

We now proceed by contradiction. Suppose some $\alpha_i\neq 0$. Then
after relabeling,  we may assume  that all $\alpha_i\neq 0$.  Applying the previous argument
we may find $z_0$, $K_0>1$, and $i_0$, $j_0$, so that 
\begin{equation} \label{eqn:key}
|f_{i_0}(z_0)| \geq K_0|f_{i}(z_0)|
\end{equation}
for all $i\neq i_0$.Then, the equation 
\begin{equation*}
	\sum_{i=1}^\infty\alpha_i\cdot f^{N_m}_{i}(z_0)=0\ ,
\end{equation*}
yields
\begin{equation*}
	\alpha_{i_0}     
	        =-\sum_{i\not= i_0}\alpha_i \left(\frac{f_{i}(z_0)}{f_{i_0}(z_0)}\right)^{N_m}\ ,
\end{equation*}
where the right end term converges. Then since
$$
\left\vert \alpha_i \left(\frac{f_{i}(z_0)}{f_{i_0}(z_0)}\right)^{N_m} \right\vert \leq \left\vert \left(\frac{f_{i}(z_0)}{f_{i_0}(z_0)}\right) \right\vert \leq A\cdotp\left\vert f_{i}(z_0) \right\vert,
$$
where $A=\vert f_{i_0}(z_0)\vert^{-1}\sup_{i\in\mathbb N} \vert \alpha_i\vert$,  
Lebesgue dominated convergence theorem yields that
$$
\lim_{m\to\infty}\sum_{i\not= i_0}\alpha_i \left(\frac{f_{i}(z_0)}{f_{i_0}(z_0)}\right)^{N_m}=0.
$$
Thus $\alpha_{i_0}=0$; a contradiction.
\end{proof}

We can now pass to the proof of Proposition \ref{pro:inject}. Let $\gamma_1,\ldots, \gamma_p$ be different conjugacy classes. We may as well assume that we have an involution $\sigma$ of $\{1,\ldots, p\}$ so that $\gamma_i=\gamma_{\sigma(i)}^{-1}$. Assume that there exist $\alpha_1,\ldots,\alpha_p$ real numbers so that for all $k$,
\begin{equation*}
 	\sum_{i=1}^p\alpha_i\cdotp\Theta_{\gamma_i}^{(k)}=0\ .
\end{equation*}
Let $(x_i,X_i)$ be the end points at infinity of $\gamma_i$, then we can write
\begin{eqnarray}
	0&=&\sum_{i=1}^p\alpha_i\cdotp\Theta_{\gamma_i}^{(k)}\cr
	&=&\left(\sum_{i=1}^p\alpha_i\sum_{\eta\in\grf/\braket{\gamma_i}}\cdotp\left(\eta^*\frac{\d z}{(z-x_i)(z-X_i)}\right)^k\right)\ .
\end{eqnarray}
Defining  $g_{i,\eta}$ by
$$g_{i,\eta}(z)\cdotp\d z =\eta^*\left(\frac{\d z}{(z-x_i)(z-X_i)}\right)\ ,$$
we get 
\begin{eqnarray*}
0&=&
\sum_{i=1}^p
\sum_{\eta\in\grf/\braket{\gamma_i}}\alpha_i\cdotp\left(g_{i,\eta}(z)\right)^k\ .
\end{eqnarray*}
%Now let
%$$
%I^+_m=\{i \, |\,  q_i>0\}\ .
%$$
Hence, for all even $k$, 
\begin{eqnarray}
0&=&\sum_{i=1}^p\sum_{\eta\in\grf/\braket{\gamma_i}}\left(\alpha_i +\alpha_{\sigma(i)}\right)\cdotp\left(g_{i,\eta}(z)\right)^k\ .\label{qeven}
\end{eqnarray}
Similarly, for all odd $k$, 
\begin{eqnarray}
0&=&\sum_{i=1}^p\sum_{\eta\in\grf/\braket{\gamma_i}}\left(\alpha_i -\alpha_{\sigma(i)}\right)\cdotp\left(g_{i,\eta}(z)\right)^k\ .\label{qodd}
\end{eqnarray}
Next we observe that if for all $z$,
$
\vert g_{i,\eta}(z)\vert=\vert g_{j,\zeta}(z)\vert
$, then $\eta=\zeta$, and $i=j$ or $i=\sigma(j)$. 
Recall that $g_{\sigma(i),\eta}=-g_{i,\eta}$.
Therefore, applying Lemma \ref{fiq} to the sequence of even integers, we get from  \eqref{qeven}, that $\alpha_i+\alpha_{\sigma(i)}=0$ for all $i$. Similarly,  applying Lemma \ref{fiq} to the sequence of odd integers, we get from  \eqref{qodd}, that $\alpha_i-\alpha_{\sigma(i)}=0$ for all $i$. As a consequence,  $\alpha_i$ vanishes for all $i$. This concludes the proof of the proposition.
\subsubsection{Proof of Theorem \ref{imageopen}} The proof is an immediate consequence of the following
\begin{proposition}
Let $\delta$  be a Fuchsian representation. For $n$ large enough, 
$\Lambda_{\vec\gamma}^{(n)}$ is a submersion at $\delta$.
\end{proposition}
\begin{proof}
Let $\delta$ be a Fuchsian representation with values in $\mathsf{SL}(n,\mathbb R)$, $\gamma$ be a nontrivial element of $\grf$, let $\lambda_\gamma(\delta)$ be the highest eigenvalue $\delta(\gamma)$. Then by
Theorem \ref{thm:hamil},
 the Hamiltonian vector field $H_\gamma$ of $\log\lambda_\gamma$ at $\delta$  is given by
 \begin{equation} \label{eqn:general-hamiltonian}
 H_\gamma=\sum_{k=2}^n C(n,k)\cdotp \psi\left(i\cdot\Theta^{(k)}_\gamma\right)\ ,
 \end{equation}
where $C(n,k)$ are all nonzero (see
 Remark \ref{rem:nonzero-coefficients}), and are independent of $\gamma$. Now if $\gamma_1,\ldots,\gamma_p$ are primitive elements of $\grf$,  the proof of the proposition will be complete if we can show that the set of  vector fields $\{H_{\gamma_i}\}_{i=1}^p$ is linearly independent. 
Let $W$ be the finite dimensional subspace of $\mathbb R^{\mathcal C}$ spanned by the $\gamma_i$, $i=1,\ldots, p$. Then from the finite dimensionality of $W$ and Proposition \ref{pro:inject}, there exists some $k_1$ so that the map 
\begin{eqnarray*}
\phi_W:	W&\to& \sum_{k=2}^{k_1}H^0(X, K^k)\ ,\\
	\gamma&\mapsto& (\Theta_\gamma^{(2)},\ldots,\Theta_\gamma^{(k_1)})\ ,
\end{eqnarray*} 
 is injective. Suppose that we have $\alpha_i\in \mathbb R$, 
  so that 
 $$
 \sum_{i=1}^p \alpha_i \cdot H_{\gamma_i}=0\ .
 $$
 Since the $\psi\left(i\cdot \Theta_{\gamma_i}^{(k)}\right)$ are orthogonal for different values of $k$,  from \eqref{eqn:general-hamiltonian} we obtain  that
  $$
C(n,k) \sum_{i=1}^p \alpha_i \cdotp\psi\left(i\cdot \Theta_{\gamma_i}^{(k)}\right)=0\ ,
 $$
 for all $k\in\{2,\ldots,k_1\}$.
 Since $C(n,k)\not=0$, and $\psi$ is real linear and injective, we find
 $$
\sum_{i=1}^p \alpha_i \cdot\Theta_{\gamma_i}^{(k)}=0\ ,
 $$
 and thus 
 $$
 \phi_W\left(\sum_{i=1}^p\alpha_i\cdot \gamma_i\right)=0\ .
$$
Finally, by the injectivity of $\phi_W$, we conclude that $\alpha_i=0$ for all $i$. This shows that $\{H_{\gamma_i}\}_{i=1}^p$ is linearly independent and finishes the proof.
     	\end{proof}

\appendix
\section*{Appendices}
%\addcontentsline{toc}{section}{Appendices}
\renewcommand{\thesubsection}{\Alph{subsection}}

\subsection{Computation of traces}
%  Let $V\oplus\overline V$ be a $2$-dimensional complex vector space
%  Then $\operatorname{Sym}^{n-1}(V\oplus \overline V)$ is an $n$-dimensional space spanned by
%   homogeneous polynomials  of degree $n-1$.
 For the irreducible embedding $\kappa_n: \mk{sl}(2)\to \mk{sl}(n)$, set $X=\kappa_n(x)$, $Y=\kappa_n(y)$ (see \eqref{eqn:k2}). %Then the action of $X, Y$ 
 In the basis $\{w_p\}$ from \eqref{eqn:wp}, $X$ and $Y$ act by eqs.\ \eqref{eqn:X} and \eqref{eqn:Y}.
% in the standard basis $\{\hat w_p\}$ of ${\mathbb C}^n$ is given by 
%       \begin{align*}
%(-X)(\hat w_p)&=c_{p} \cdotp\hat w_{p+1}\ ,\\
%Y(\hat w_{p})&=c_{p-1}\cdotp\hat w_{p-1}\ .
%\end{align*}
%where $c_p=(p(n-p)/2)^{1/2}$ (see \cite[p.\ 978]{Kostant:1959wi}). 
%If we make a  change of basis by rescaling
%$$ w_p=\left\{(p-1)!(n-p)!\right\}^{1/2}\cdotp \hat w_p$$
% then
%\begin{align}
%(-X)( w_p)&=\tfrac{1}{\sqrt 2}(n-p)\cdotp   w_{p+1}\ ,\label{eqn:x-action}\\
%Y( w_{p})&=\tfrac{1}{\sqrt 2}(p-1)\cdotp  w_{p-1}\ .\label{eqn:y-action}
%\end{align}
The action of $X$ and $Y$ is then 
 identified with that  of  first order differential operators on  homogeneous polynomials of degree $n-1$ in 
  the variables $z=x+iy$ and $\bar z=x-iy$.
 Namely, identifying $ w_p=z^{p-1}\bar z^{n-p}$ we have
\begin{equation} \label{eqn:diffops}
-\sqrt{2}X=z\frac{\partial}{\partial \bar z} \ ,\
\sqrt{2}Y=\bar z\frac{\partial}{\partial  z}\ .
\end{equation}
For $k=1, \ldots, n-1$, let $E^0_k=(-\sqrt{2}X)^k$, $F^0_k=-(\sqrt{2}Y)^k$. 
% With respect to a hermitian structure where the $\{w_p\}$ form a unitary basis, say, then $F^0_k=\rho(E^0_k)=-(E^0_k)^\ast$.

We first compute the following
\begin{lemma} \label{lem:trace}
\begin{equation*}
-\tr\left(E^0_kF^0_k\right)= (k!)^2 {n+k \choose 2k+1}\ .	
\end{equation*}
In particular,
\begin{equation} \label{eqn:En-1}
-\tr\left(E^0_{n-1}F^0_{n-1}\right)=((n-1)!)^2\ ,
\end{equation}	 
and
\begin{equation} \label{eqn:E-dynkin}
-\tr\left(E^0_{1}F^0_{1}\right)=d(n)\ ,
\end{equation} 
for all $n$ (see \eqref{eqn:dynkin-sln}).
\end{lemma}

\begin{corollary} \label{cor:etak}
	Recall the definition \eqref{eqn:etak}. Then,
	\begin{equation*}
	\eta_k=\frac{1}{k!}
	{n+1 \choose 3}^{1/2}
	\cdotp{ n+k \choose 2k+1}^{-1/2}
\ .\label{eqn:etak2}
	\end{equation*}
\end{corollary}

\begin{proof}[Proof of Lemma \ref{lem:trace}]
From \eqref{eqn:X} and \eqref{eqn:Y}, 
\begin{align*}
E^0_k\left(\frac{w_p}{(n-p)!}\right)&=\frac{w_{p+k}}{(n-p-k)!} \ ,\\
F^0_k\left(\frac{w_p}{(p-1)!}\right)&=-\frac{w_{p-k}}{(p-k-1)!} \ .
\end{align*}
Hence, for $p\geq k+1$,
$$
-E^0_kF^0_k w_p=E^0_k\frac{ (p-1)!}{(p-k-1)!}w_{p-k}=\frac{ (p-1)!}{(p-k-1)!}\frac{(n-p+k)!}{(n-p)!}w_p\
$$
By summing over $p$, we get 
$$
-\tr(E^0_kF^0_k)=(k!)^2\sum_{p=k+1}^n{p-1\choose k}{ n-p+k \choose k}\ .
$$
We claim that
\begin{equation} \label{eqn:combinatorial}
\sum_{p=k+1}^n{p-1\choose k}{ n-p+k \choose k}={n+k\choose 2k+1}\ .
\end{equation}
This is actually clear from the combinatorial interpretation. Alternatively, define a 
generating function
$$
f_k(x)\defeq \sum_{m=k}^\infty x^m{{m\choose k}}\ .
$$
Then the left hand side of \eqref{eqn:combinatorial} is  coefficient of $x^{n+k-1}$ in $f_k^2(x)$.
On the other hand, 
by \eqref{eqn:asym},
$$
f_k(x)=\frac{x^k}{(1-x)^{k+1}}\ ,
$$
from which $xf_k^2(x)=f_{2k+1}(x)$.
Hence,  the  coefficient of $x^{n+k-1}$ in $f_k^2(x)$ is the coefficient of $x^{n+k}$ in $f_{2k+1}(x)$, which is 
  the right hand side of \eqref{eqn:combinatorial}.
  This proves \eqref{eqn:combinatorial} and hence also the lemma.
\end{proof}

The sections $u_p$ from \eqref{eqn:up} correspond to $x^{p-1}y^{n-p}$. Let $\pi_p$ denote the orthogonal projection to the $1$-dimensional space spanned by $u_p$.
 Note that by the hermiticity mentioned above, $\tr(F^0_k\pi_p)=-\tr(E^0_k \pi_p)$.
 We then have the following proposition

\begin{proposition}\label{A:pretrace}
\begin{align}
\tr(E^0_k\cdotp \pi_1)&=\frac{(-1)^k(n-1)!}{2^k(n-k-1)!}\ , \label{A:tr-high} \\
\tr(E^0_k\cdotp
\pi_p)&=\frac{(p-1)!(n-p)!}{2^k(n-k-1)!}\sum_{j=\max(0,k+p-n)}^{\min(k,p-1)} 
{{n-k-1}\choose{p-j-1}}{{k}\choose{j}}^2(-1)^{j+k}\ .\label{A:tr-all}
\end{align} 
\end{proposition}

\begin{proof} Set $a_{n,p}=2^{1-p}(2i)^{p-n}$, so that $u_p=a_{n,p}(z+\bar z)^{p-1}( z-\bar z)^{n-p}$.
As a warmup, consider the easy case when $p=1$ and prove \eqref{A:tr-high}.
Then using \eqref{eqn:diffops}, 
\begin{eqnarray}
E^0_k\cdotp u_1&=&a_{n,1}\left( z\frac{\partial}{\partial \bar z}\right)^k(z-\bar
z)^{n-1}\cr
&=&(-1)^k a_{n,1}z^k \frac{(n-1)!}{(n-1-k)!}(z-\bar z)^{n-k-1}\cr
&=& \frac{a_{n,1}}{(-2)^k} \frac{(n-1)!}{(n-1-k)!}(z-\bar z)^{n-1-k}\left((z+\bar z)+( z -\bar z)\right)^k\cr
&=&\frac{a_{n,1}}{(-2)^k} \frac{(n-1)!}{(n-1-k)!}(z-\bar z)^{n-1}+\sum_{i=2}^{n} A_i\cdotp u_i \cr
&=&(-2)^{-k} \frac{(n-1)!}{(n-1-k)!}u_1+\sum_{i=2}^{n} A_i\cdotp u_i,
\end{eqnarray}
where $A_i$ are some complex numbers. Eq.\  \eqref{A:tr-high} follows.
Let us move to the general case and prove \eqref{A:tr-all}. 
\begin{align*}
 E^0_k\cdotp u_p
&=a_{n,p} z^k\frac{\partial^k}{\partial  \bar z^k}\left((z+\bar z)^{p-1}( z-\bar z)^{n-p}\right)\cr
&=a_{n,p}z^k\sum_{j=0}^k{{k}\choose{j}}
\left(\frac{\partial^j}{\partial  \bar z^j}(z+\bar z)^{p-1}\right)\left(\frac{\partial^{k-j}}{\partial \bar z^{k-j}}( z-\bar z)^{n-p}\right)\cr
&=a_{n,p}z^k\sum_{j=\max(0,k+p-n)}^{\min(k,p-1)} (-1)^{j+k} {{k}\choose{j}}
\left(\frac{(p-1)!}{(p-1-j)!}(z+\bar z)^{p-j-1}\right)\left(\frac{(n-p)!}{(n+j-k-p)!}( z-\bar z)^{n-p-k+j}\right)\cr
&=a_{n,p}\sum_{j=\max(0,k+p-n)}^{\min(k,p-1)} (-1)^{j+k}
\frac{k!(p-1)!(n-p)!}{j!(k-j)!(p-j-1)!(n+j-k-p)!} z^k(z+\bar z)^{p-j-1}(z-\bar z)^{n+j-p-k}\cr
&=k!a_{n,p}\sum_{j=\max(0,k+p-n)}^{\min(k,p-1)} (-1)^{j+k}
{{p-1}\choose{j}}{{n-p}\choose{k-j}}(z+\bar z)^{p-j-1}(z-\bar z)^{n+j-p-k} z^k\cr
&=\frac{k!a_{n,p}}{2^{k}}\sum_{j=\max(0,k+p-n)}^{\min(k,p-1)} (-1)^{j+k}
{{p-1}\choose{j}}{{n-p}\choose{k-j}}(z+\bar z)^{p-j-1}(z-\bar z)^{n+j-p-k}((z+\bar z)+(z-\bar z))^k\cr
&=\frac{k!a_{n,p}}{2^{k}}\sum_{j=\max(0,k+p-n)}^{\min(k,p-1)} \sum_{l=0}^k (-1)^{j+k}
{{p-1}\choose{j}}
{{n-p}\choose{k-j}}{{k}\choose{l}}(z+\bar z)^{p-j+l-1}(z-\bar z)^{n+j-p-l}\cr
&=\frac{k!}{2^k}\sum_{j=\max(0,k+n-p)}^{\min(k,p-1)} 
{{p-1}\choose{j}}
{{n-p}\choose{k-j}}{{k}\choose{j}}(-1)^{j+k}u_p+\sum_{j\not=p}A_j\cdotp u_j,
\end{align*}
for some coefficients $A_i$ that we do not need to make explicit. Thus
\begin{align*}
\tr(E^0_k\cdotp \pi_p)
&=\frac{k!}{2^k}\sum_{j=\max(0,k+p-n)}^{\min(k,p-1)} {{p-1}\choose{j}}
{{n-p}\choose{k-j}}{{k}\choose{j}}(-1)^{j+k}\cr
&=\frac{(p-1)!(n-p)!}{2^k}\sum_{j=\max(0,k+p-n)}^{\min(k,p-1)} 
\frac{1}{(p-j-1)!(n-p-k+j)!}{{k}\choose{j}}^2(-1)^{j+k}\cr
&=\frac{(p-1)!(n-p)!}{2^k(n-k-1)!}\sum_{j=\max(0,k+p-n)}^{\min(k,p-1)} 
{{n-k-1}\choose{p-j-1}}{{k}\choose{j}}^2(-1)^{j+k}\ .
\end{align*}
This concludes the proof of the proposition.
\end{proof}

\subsection{Proof of Theorem \ref{tech}}
%\subsection{Some series}
%We recall that
%\begin{eqnarray}
%\frac{1}{(1-L)^{N}}
%&=&\sum_{m=0}^\infty {{m+N-1}\choose{m}}\cdotp L^m,\label{eqn:asym0}
%\end{eqnarray}

\subsubsection{Closed formula for the integrals}
We shall need the technical formula. Recall that $m!!=m(m-2)\cdots 2$ for $m$ even, $m!!=m(m-2)\cdots 1$ for $m$ odd.
\begin{proposition}
We have 
\begin{eqnarray}
I_{m,d}(R)&=&\sum_{k=0}^{d}R^{m+2k+1}(1-R^2)^{d-k}\cdotp \frac{(2d)!!(m-1)!!}{(2(d-k))!!(m+2k+1)!!}\ .\label{eqn:int}
\end{eqnarray}
\end{proposition}
\begin{proof}
An integration by parts lea\d s to the recursion relation
\begin{equation}
I_{m,d}(R)=\frac{2d}{m+1}I_{m+2,d-1}(R)+\frac{1}{m+1}R^{m+1}(1-R^2)^d\ . \label{eqn:int-rec}
\end{equation}
Indeed,
\begin{eqnarray*}
I_{m,d}(R)&=&\int_0^R S^m(1-S^2)^d\cdotp \d S\\
&=&\frac{2d}{m+1}\int_0^R S^{m+2}(1-S^2)^{d-1}\cdotp \d S+ \frac{1}{m+1}\left[S^{m+1}(1-S^2)^d\right]_0^R\\
&=&\frac{2d}{m+1}I_{m+2,d-1}(R)+\frac{1}{m+1} R^{m+1}(1-R^2)^d\ .
\end{eqnarray*}
Observe that 
$$
I_{m,0}(R)=\frac{1}{m+1}R^{m+1},
$$
satisfies  equation \eqref{eqn:int}. Moreover, the right hand side of 
equation \eqref{eqn:int} satisfies the recursion formula.
\end{proof}

\subsubsection{Asymptotic series}

If $G_m$ and $H_m$ are positive functions on $[0,1]$, we write
$$
G_m(R)\asymp H_m(R),
$$
if 
$$
\left\{\frac{G_m(R)}{H_m(R)}\right\}_{m\in \mathbb N}\xrightarrow[m\to\infty]{C^0} f(R)\, , \hbox{ where }  \lim_{R\to 1}f(R)=1.
$$
We will also use the nonstandard convention that 
$$
{{n-1}\choose{n}}=\frac{1}{n}\ ,
$$
which is coherent with
\begin{equation*}
{{n+k}\choose{n}}\sim n^k\ ,\label{eqn:sim}
\end{equation*}
for all $k\geq -1$. Let also
\begin{eqnarray}\label{eqn:F}
{\bf F}(k,p)&\defeq& \sum_{i+j=k,i,j\leq p-1}\frac{(2p-2)!!^2}{(2(p-1-i))!!(2(p-1-j))!!}\ .
\end{eqnarray}
In particular
$$
{\bf F}(2p-2,p)=(2p-2)!!^2=2^{2p-2}(p-1)!^2\ .
$$

We now prove
\begin{proposition}\label{pro:asymp}
We have 
\begin{equation}
{{m+2p-1}\choose{m}}I^2_{m,p-1} \asymp  R^{2m}\cdotp \sum_{k=0}^{2p-2}R^{2k+2}(1-R^2)^{2p-2-k}\cdotp  {{m+2p-3-k}\choose{m}}\cdotp {\bf F}(k,p)\label{tech:j}\ .
\end{equation}
\end{proposition}

\begin{proof} 
From equation \eqref{eqn:int}, we get
\begin{equation*}
I_{m,d}^2(R)=\sum_{k=0}^{2d}\left(R^{2(m+k+1)}(1-R^2)^{2d-k}\cdotp\sum_{i+j=k, i,j\leq d}\frac{(2d)!!^2(m-1)!!^2}{(2(d-i))!!(2(d-j))!!(m+2i+1)!!(m+2j+1)!!}\right)\ .\label{eqn:int2}
\end{equation*}

Since $(m+2i+1)!!(m+2j+1)!!\sim m^{k+2} (m-1)!!^2$, we obtain
\begin{eqnarray*}
I_{m,p-1}^2(R)&\asymp &\sum_{k=0}^{2p-2}R^{2(m+k+1)}\left((1-R^2)^{2p-2-k}\cdotp\frac{1}{m^{k+2}}\sum_{i+j=k,i,j\leq p-1}\frac{(2p-2)!!^2}{(2(p-1-i))!!(2(p-j-1))!!}\right)\ ,
\end{eqnarray*}
%Recall that by definition 
%$$
%\sum_{i+j=k}\frac{(2p-2)!!^2}{(2(p-1-i))!!(2(p-1-j))!!}=\sum_{u+v=k}\frac{(2p-2)!!^2}{(2(p-1-i))!!(2(p-1-j))!!}{\bf F}(k,p).
%$$
Recall  that
$$
b_m={{m+2p-1}\choose{m}}\sim m^{2p-1}.
$$
Then by \eqref{eqn:F}
\begin{eqnarray*}
b_m\cdotp I^2_{m,p-1}(R)
&\asymp &
\sum_{k=0}^{2p-2}R^{2(m+k+1)}\left((1-R^2)^{2p-2-k}\cdotp m^{2p-3-k}\cdotp {\bf F}(k,p)\right)\cr
&\asymp &\sum_{k=0}^{2p-2}R^{2(m+k+1)}\left((1-R^2)^{2p-2-k}\cdotp {{m+2p-3-k}\choose{m}}\cdotp {\bf F}(k,p)\right)\ .
\end{eqnarray*}
 This finishes the proof of the proposition.
 \end{proof}

\subsubsection{A preliminary lemma}
We need an easy lemma.
\begin{lemma} Assume that the positive functions $G_m(R)$ and $H_m(R)$ satisfy 
\begin{eqnarray}
& & G_m(R)\asymp H_m(R)\label{hyp:sim}\ , \notag\\
& & \forall m,\: \lim_{R\to 1}H_m(R)=\lim_{R\to 1}G_m(R)=0, \label{hyp:0}\notag \\
& & \lim_{R\to 1}\sum_{m=0}^\infty H_m(R)=D<\infty\ .\label{hyp:bound}
\end{eqnarray}
Then 
$$
\lim_{R\to 1}\sum_{m=0}^\infty G_m(R)=\lim_{R\to 1}\sum_{m=0}^\infty H_m(R)\ .
$$
\end{lemma}
\begin{proof} Let $\epsilon >0$, from the hypothesis \eqref{hyp:sim}, there exist $0<\alpha<1$ and $k_0$ such that for all $m>k_0$ and $\alpha<R<1$ 
$$
\left\vert 1-\frac{G_m(R)}{H_m(R)}\right\vert \leq \epsilon\ .
$$
Thus
\begin{equation}
\left\vert \sum_{m=k_0}^\infty \left(G_m(R)-H_m(R)\right)\right\vert\leq 2\epsilon\cdotp D\ .\label{proo:tech1}
\end{equation}
Using \eqref{hyp:0}, we now choose  $\alpha$ so that for all $m\leq k_0$ and all $R\geq \alpha$, 
\begin{equation}
\left\vert \sum_{m=0}^{k_0} \left(G_m(R)-H_m(R)\right)\right\vert\leq \epsilon\ .\label{proo:tech2}
\end{equation}
Combining equations \eqref{proo:tech1} et \eqref{proo:tech2}
one gets that for all $R>\alpha$, one has
\begin{equation*}
\left\vert \sum_{m=0}^\infty \left(G_m(R)-H_m(R)\right)\right\vert\leq \epsilon\cdotp (2D+1)\ .
\end{equation*}
This last assertion concludes the proof of the lemma.
\end{proof}
\subsubsection{Proof of Theorem \ref{tech}}
We now prove the theorem.
\begin{proof} Let 
\begin{eqnarray*}
G_m(R)&\defeq & -{{m+2p-1}\choose{m}}\frac{I^2_{m,p-1}}{\log(1-R)}\\
H_m(R)&\defeq &-\frac{R^{2m}}{\log(1-R)}\cdotp \sum_{k=0}^{2p-2}R^{2k+2}(1-R^2)^{2p-2-k}\cdotp  {{m+2p-3-k}\choose{m}}\cdotp {\bf F}(k,p)
\end{eqnarray*}
From proposition \ref{pro:asymp}, one gets that 
\begin{equation*}
G_m(R)\asymp H_m(R)\ .
\end{equation*}
Using the asymptotic expansion, for $0\leq k\leq 2p-3$,
\begin{eqnarray*}
(1-R^2)^{2p-2-k}{\bf F}(k,p)\cdotp \sum_{m=0}^\infty R^{2m}. {{m+2p-3-k}\choose{m}}&=&(1-R^2)^{2p-2-k}\frac{{\bf F}(k,p)}{(1-R^2)^{2p-2-k}}
\crcr
&=&{\bf F}(k,p)\ .
\end{eqnarray*}
Similarly for $k=2p-2$, we get that
\begin{eqnarray*}
& &(1-R^2)^{2p-2-k}{\bf F}(2p-2,p)\cdotp \sum_{m=1}^\infty R^{2m}.{ {m-1}\choose{m}}\cr
&=& {\bf F}(2p-2,p)\sum_{m=1}^\infty \frac{R^{2m}}{m}\\
&=&{\bf F}(2p-2,p)(1-\log(1-R))\ .
\end{eqnarray*}
It follows that (with the convention that $H_0(R)\defeq 0$)
$$
\sum_{m=0}^{\infty}H_m(R)=R^{2k+2}\cdotp\left({\bf F}(2p-2,p)-\sum_{k=0}^{2p-2}\frac{{\bf F}(k,p)}{\log(1-R)}\right)\ .
$$
In particular
$$
\lim_{R\to 1}\sum_{m=0}^{\infty}H_m(R)={\bf F}(2p-2,p)=2^{2p-2}(p-1)!^2\ .
$$
We now observe that $H_m(R)$ and $G_m(R)$ satisfies the hypothesis \eqref{hyp:0} of the previous lemma, thus
\begin{equation}
\lim_{R\to 1}\sum_{m=0}^{\infty}G_m(R)=2^{2p-2}(p-1)!^2\ .\label{lim:G}
\end{equation}
Thus we obtain Theorem \ref{tech} as a consequence of \eqref{lim:G}.
\end{proof}

\bibliographystyle{amsplain}
\bibliography{./papers}
\end{document}